\documentclass[a4paper,11pt]{amsart}

\usepackage[T1]{fontenc}
\usepackage{amssymb}
\usepackage{amsmath}
\usepackage{amsthm}
\usepackage{enumitem}
\usepackage[usenames,dvipsnames]{xcolor}
\usepackage{marginnote}
\usepackage{nicematrix}
\usepackage{geometry} 
\usepackage[hidelinks]{hyperref}
\newtheorem{theorem}{Theorem}[section]
\newtheorem{lemma}[theorem]{Lemma}
\newtheorem{prop}[theorem]{Proposition}
\newtheorem{definition}[theorem]{Definition}
\newtheorem{corollary}[theorem]{Corollary}

\hypersetup{colorlinks=true}

\theoremstyle{definition}
\newtheorem{remark}[theorem]{Remark}

\def\eps{\varepsilon}
\def\Om{\Omega}
\def\oOm{\overline{\Omega}}
\def\K{\mathcal{K}}
\def\R{\mathbb{R}}
\def\H{\mathcal{H}}
\def\N{\mathbb{N}}
\def\A{\mathcal{A}}
\def\eps{\varepsilon}

\title[Fuglede-type arguments for isoperimetric problems and applications to stability]{Fuglede-type arguments for isoperimetric problems and applications to stability among convex shapes}

\author[R. Prunier]{Raphaël Prunier}
\address[Raphaël Prunier]{Sorbonne Universit\'e and Universit\'e Paris Cit\'e, CNRS, IMJ-PRG, F-75005 Paris, France.
}
\email{  raphael.prunier@imj-prg.fr}

\begin{document}
\maketitle
\begin{abstract}
This paper is concerned with stability of the ball for a class of isoperimetric problems under convexity constraint. Considering the problem of minimizing $P+\eps R$ among convex subsets of $\R^N$ of fixed volume, where $P$ is the perimeter functional, $R$ is a perturbative term and $\eps>0$ is a small parameter, stability of the ball for this perturbed isoperimetric problem means that the ball is the unique (local, up to translation) minimizer for any $\eps$ sufficiently small. We investigate independently two specific cases where $\Om\mapsto R(\Om)$ is an energy arising from PDE theory, namely the capacity and the first Dirichlet eigenvalue of a domain $\Om\subset\R^N$. While in both cases stability fails among all shapes, in the first case we prove (non-sharp) stability of the ball among convex shapes, by building an appropriate competitor for the capacity of a perturbation of the ball. In the second case we prove sharp stability of the ball among convex shapes by providing the optimal range of $\eps$ such that stability holds, relying on the \textit{selection principle} technique and a regularity theory under convexity constraint. \end{abstract}

\section{Introduction}

\subsection{Stability in shape optimization}
In this article we are interested in the question of stability of the ball for isoperimetric-type problems under a convexity constraint. It takes place in the framework of shape optimization problems involving the perimeter functional $P$, which consists in the minimization problems
\[\inf\left\{P(A)+R(A), \ A\in\mathcal{A}\right\}\]
where $\mathcal{A}$ is a class of measurable subsets $A\subset \R^N$ of volume $|A|=1$, $P(A)$ is the perimeter of $A$ in the usual De Giorgi sense, and $R:\mathcal{A}\to\R$ is a functional thought of as a perturbative term. Due to the well-known \textit{isoperimetric inequality}, any ball $B\subset\R^N$ of unit volume is minimal for the minimization of $P+R$ among all sets in $\A$ when $R=0$: for all measurable sets $A\subset\R^N$ with volume $|A|=1$, $P(A)\geq P(B)$ with equality if and only if (up to a set of measure $0$) $A$ is a ball of unit volume. By stability of the ball for the problem $P+R$ we mean that, considering $R_\eps:=\eps R$ for a small parameter $\eps>0$, then provided $\eps$ is close enough to $0$:
\[\textit{The ball }B \textit{ is a local minimizer of }P+\eps R \textit{ in }\mathcal{A}\]
where by locality we mean that the $L^1$ distance of $\Om\in\mathcal{A}$ with $B$ is small (\textit{i.e.} $|\Om\Delta B|\ll1$). In other words, this notion of stability states that the ball is still a minimizer of the perimeter functional when it is perturbed by another functional $R$. Another way of putting it comes from rewriting the minimality of $B$ as 
\[P(A)-P(B)\geq \eps \left(R(B)-R(A)\right)\]
so that, assuming moreover that $R(B)\geq R(A)$ for each $A\in \mathcal{A}$ (which is always verified for the cases we have in mind), then the deficit of perimeter quantifies the deficit of the functional $R$. Let us mention that this point of view on stability encompasses what is usually refered to in the literature as quantitative inequalities, the most famous one being the \textit{sharp quantitative isoperimetric inequality}, proven in \cite{FMP08}. It claims that by setting 
\[\delta_F(A):=\inf\{|A\Delta (B+x)|,\ x\in\R^N\}\]
the Fraenkel asymmetry of a set $A\subset \R^N$ of unit volume, then there exists $c_N>0$ such that 
\[P(A)-P(B)\geq c_N\delta_F(A)^2. \] 
In our stability setting this can be rephrased into stability of $P-\delta_F^2$ among all sets of unit volume. The literature on quantitative inequalities in shape optimization is very prolific, and we refer for instance among many others to \cite{FMP08}, \cite{BdPV15}, \cite{FFMMM15}, \cite{AFM13}, \cite{CL12}, \cite{BNT}, \cite{FMP10} and to \cite{Fus15} for a nice review of stability results linked to the isoperimetric inequality.
\subsection{Stability of the ball under convexity constraint}\label{sect:intro_stab_cccc}
We are more specifically interested in shape optimization problems where $\mathcal{A}$ only contains convex shapes, that is
\[\inf\left\{P(K)+\eps R(K), \ K\in\K^N, \ |K|=1\right\}\]
where $\K^N$ denotes the class of convex bodies of $\R^N$ (that is, compact convex sets with non-empty interior). The addition of the convexity constraint is interesting since stability among all shapes fails for the functionals we will consider. This happens for some problems where $R$ is of PDE-type, by which we mean that $R(K)$ is an energy associated to a PDE which is set on $K$ or $\R^N\setminus K$. In this paper we investigate independently two specific problems falling into this category. Let us now introduce them and state the stability results associated. 

\subsubsection{Weak stability for $P+\text{Cap}^{-1}$}
We are interested in a first problem which involves a PDE set on the exterior of the domain. For $N\geq3$ we introduce the capacity functional $\text{Cap}:\K^N\rightarrow \R$ which we define as the usual \textit{electrostatic capacity}:
\begin{align}\label{eq:capN=3}\text{Cap}(K)&:=\inf\left\{\int_{\R^N}|\nabla u|^2, \ u\in C^{\infty}_c(\R^N) \text{ with } u\geq \textbf{1}_K\right\}.\end{align}

When $N=2$, one can see by looking at the energy of rescaled and truncated versions of the fundamental solution of the Laplacian that the infimum in \eqref{eq:capN=3} is always $0$. Therefore one must proceed differently to define the capacity for $N=2$ (see Remark \ref{rk:2dcap}).

We now set the problem: letting  $\eps>0$ be a small parameter, we are interested in the minimization 
\begin{equation}\label{eq:pb_cap1st}\inf\left\{P(K)+\eps\text{Cap}(K)^{-1},\ K\in\K^N,\ |K|=1\right\}\end{equation}
for $N\geq3$. Before giving motivations and context for this problem let us state the stability result which we obtained.
\begin{theorem}[Weak stability of the ball for the capacity]\label{th:stab_cap}
Let $N\geq3$. There exists $\eps_0=\eps_0(N)>0$ such that for any $\eps \in(0,\eps_0)$ the centered ball of unit volume $B$ is the unique (up to translation) minimizer of \eqref{eq:pb_cap1st}.
\end{theorem}

We call this result \textit{weak} stability in the sense that the $\eps_0$ found is not optimal (contrarily to Theorem \ref{th:mainthm} below), and is in fact not even explicit. On the other hand note that this minimality result is global. It is the $N\geq3$ version of the two dimensional result \cite[Corollary 1.3]{GNR18}, where the authors prove weak stability of the ball with the logarithmic capacity instead. Our approach is however very different, see Section \ref{sect:strat_stabbb}.

Due to the \textit{isocapacitary inequality} (see for instance \cite{dPMM21}) which states that
\[\forall\Om\subset\R^N\text{ open},\ |\Om|=1,\ \text{Cap}(\Om)\geq \text{Cap}(B)\]
and the \textit{isoperimetric inequality}, there is a competition in the minimization \eqref{eq:pb_cap1st} which makes the problem non trivial. 
The introduction of the convexity constraint in the problem comes from the fact that existence does not hold without any additional geometric assumption (for non-existence for any $\eps>0$ and in all dimensions $N\geq2$  see \cite[Theorem 3.2 and Theorem 6.2]{GNR15}). On the other hand \eqref{eq:pb_cap1st} admits minimizers for any $\eps>0$ (see \cite[Theorem 1.1]{GNR18}).

\subsubsection{Strong stability for $P-\lambda_1$}If $\Om\subset \R^N$ is an open set with finite volume we let $\lambda_1(\Om)$ be its first Dirichlet eigenvalue, defined as the smallest number $\lambda\in\R$ such that there exists a non trivial function $u$ verifying 
\[\begin{cases}-\Delta u=\lambda u,\ \text{ in }\Om\\ u\in H^1_0(\Om)\end{cases}\] where the first equation holds in the distributional sense in $\Om$. It has a variational characterization:
\[\lambda_1(\Om):=\inf\left\{\frac{\int_\Om|\nabla u|^2}{\int_\Om|u|^2},\ u\in H^1_0(\Om)\right\}.\] 

For $K\in \K^N$ we set $\lambda_1(K):=\lambda_1(\text{Int}(K))$. 
Consider then the minimization problems

\begin{equation} \nonumber \inf\left\{P(\Om)-c\lambda_1(\Om),\  \Om\subset \R^N\text{ open}, |\Om|=1\right\},\ \ \ \inf\left\{P(K)-c\lambda_1(K),\  K\in\K^N,\ |K|=1\right\}\label{eq:Pc}\end{equation} for any fixed parameter $c>0$. There is a competition between the perimeter and $\lambda_1$, as it is known from the \textit{isoperimetric} and \textit{Faber-Krahn inequalities} that a ball $B$ of volume $1$ minimizes them both among shapes of unit volume. Intuitively, we expect that the perimeter is the dominant term for small values of $c$ while we expect that this is no longer the case in the regime $c\rightarrow+\infty$, so that $B$ might be a local minimizer for small values of $c$ and not for large values of $c$. As such there is no global minimizer to any of the two problems, taking for instance a sequence of long thin rectangles of unit volume. However, even in a loose local sense there is no stability of the ball for the first problem, meaning that for any $c>0$ there exists a sequence $(\Om_{j,c})_{j\in\N}$ of open sets with 
\[|\Om_{j,c}|=1,\ \left|\Om_{j,c}\Delta B\right|\to0\text{ and } (P-c\lambda_1)(\Om_{j,c})<(P-c\lambda_1)(B) \text{ for each }j\in\N\]
as one sees by comparing the energy of the ball to the energy of the ball perforated by a small hole at its center (see for instance \cite[Proposition 6.1]{DL19}). 
A strong geometric constraint such as convexity of the admissible sets forbids this kind of behaviour, so that one might expect stability in this case. This is the object of the second main result of this article, which can be seen as sharp stability of the ball under convexity constraint for the functional $P-\lambda_1$. The result is as follows.

\begin{theorem}[Sharp stability of the ball for $\lambda_1$]\label{th:mainthm}
Let $N\geq2$. Let $\omega_N$ be the volume of a ball of radius $1$, and $p_N:=N\omega_N$, $l_N:=j_{N/2-1}^2$ 
be respectively the perimeter and first eigenvalue of a ball of radius $1$ ($j_{N/2-1}$ is the first zero of the Bessel function of the first kind of order $N/2-1$). Set \begin{equation}\label{eq:defc*}c^*:=\frac{N(N+1)p_N}{4l_N(l_N-N)\omega_N^{\frac{N+1}{N}}}\end{equation}
Let $B$ be a ball of unit volume.
\begin{itemize}
\item Let $0<c<c^*$. Then there exists $\delta_c>0$ such that
\begin{equation}\label{eq:goal_ineq}\forall K\in\K^N, |K|=1 \text{ with }|K\Delta B|\leq \delta_c, \ (P-c\lambda_1)(K)\geq (P-c\lambda_1)(B).\end{equation}
\item Let $c>c^*$. There exists a sequence of smooth convex bodies $(K_{j,c})_{j\in\N}$ of unit volume for which $|K_{j,c}\Delta B|\rightarrow 0$ and
\begin{equation}\label{eq:no_stab} (P-c\lambda_1)(K_{j,c})< (P-c\lambda_1)(B) \text{ for each }j\in\N.\end{equation}
\end{itemize}
\end{theorem}

Note that the novelty of this result comes from the first item (inequality \eqref{eq:goal_ineq}), as \eqref{eq:no_stab} was already obtained by \cite{N14} (see the second point below). We thus give an answer to the question of local minimality of the ball for the problem $P-c\lambda_1$ under convexity constraint for any value $c>0$ (except $c=c^*$). Let us place it among existing results in the literature. 
\begin{itemize}
\item First, in a weak form the stability of the ball for $P-\lambda_1$ (in the sense of \eqref{eq:goal_ineq}) was already known. It was first obtained in two dimensions by Payne and Weinberger in \cite{PW} for the larger class of simply connected domains. More precisely, the Payne-Weinberger inequality states that for any $\Om\subset\R^2$ open, simply connected with unit volume it holds
\begin{equation}\label{eq:PW}\lambda_1(\Om)-\lambda_1(B)\leq\lambda_1(B)\left(J_1(j_{01})^{-2}-1\right)\left(\frac{P(\Om)^2}{4\pi}-1\right)\end{equation}
where $J_1$ is the Bessel function of the first kind of order one, and $j_{01}$ is the first zero of the Bessel function of the first kind and of order zero. While this inequality is much more general since it gives a control of the Faber-Krahn deficit by the isoperimetric deficit for any simple connected set, it implies in particular stability of the ball among simply connected sets $\Om$ for which $P(\Om)$ is bounded from above. One can in fact derive the inequality 
\[\forall \Om\subset\R^2 \text{ open and simply connected with }|\Om|=1,\ \left(P(\Om)-P(B)\right)\geq \frac{\eps_2}{P(B)+C}\left(\lambda_1(\Om)-\lambda_1(B)\right)\] where $\eps_2:=4\pi\lambda_1(B)^{-1}(J_1(j_{01})^{-2}-1)^{-1}$, provided $P(\Om)\leq C$. Note that letting $C\to P(B)$ the constant $\eps_2/(P(B)+C)$ becomes 
\[\eps_2(2P(B))^{-1}=\left(\sqrt{\pi}j_{01}^2(J_1(j_{01})^{-2}-1)\right)^{-1}\approx 0.036,\] while the optimal constant given by \eqref{eq:defc*} equals 
\[c^*=\frac{3}{\sqrt{\pi}j_{01}^2(j_{01}^2-2)}\approx0.077.\]

On the other hand, a Payne-Weinberger type inequality for convex sets was proven in any dimensions $N\geq2$ by Brandolini, Nitsch and Trombetti \cite[Theorem 1.1]{BNT} using the Brunn-Minkowski theory. They prove that for any open convex set $\Om\subset\R^N$ it holds
\[\frac{\lambda_1(\Om)-\lambda_1(B^*)}{\lambda_1(\Om)}\leq C_N\left(\frac{P(\Om)^{\frac{N}{N-1}}-P(B)^{\frac{N}{N-1}}}{P(\Om)^{\frac{N}{N-1}}}\right)\]
for some explicit constant $C_N>0$, where $B^*$ is a ball with same perimeter as $\Om$. This again implies a non-optimal local stability of the ball for convex sets in the sense given by \eqref{eq:goal_ineq}. 
 \item Second, in \cite[Theorem 1.2]{N14} (see also \cite[Proposition 5.5 (ii)]{DL19}) Nitsch conjectured the optimal value $c^*>0$ by proving that (i) if $c<c^*$, the ball $B$ is a minimizer of $P-c\lambda_1$ among unit volume perturbations of the ball by a smooth vector field $\xi$ \textit{i.e.} for  $B_\xi:=(\text{Id}+\xi)(B)$ with $\|\xi\|_{C^\infty}\ll 1$ and (ii) if $c>c^*$, there exists $\|\xi_j\|_{C^\infty}\to0$ with $(P-c\lambda_1)(B_{\xi_j})<(P-c\lambda_1)(B)$. 
\item Finally, one can interpret this result in the context of Blaschke-Santalo diagrams. In fact, it provides the exact value of the tangent at $(x_0,y_0):=(P(B),\lambda_1(B))$ of the upper boundary of the diagram for $(P,\lambda_1,|\cdot|)$ in the class of planar convex sets, that is of the set
\begin{equation}\label{eq:diag_BS_stab}
\mathcal{D}:=\left\{(x,y)\in\R^2,\ \exists K\in\K^2,\ P(K)=x,\ \lambda_1(K)=y,\ |K|=1\right\}\end{equation}
It was proven in \cite{FL21Bla} that this diagram lies between two continuous increasing functions, meaning that 
\[\mathcal{D}=\left\{(x,y)\in[x_0,+\infty[\times\R^+,\ f(x)\leq y\leq g(x)\right\}\] for some continuous increasing $f,g:[x_0,+\infty)\to\R^+$. 
Relying on non minimality of the ball for $c>c^*$, the authors proved that $\limsup_{x\to x_0} \frac{g(x)-g(x_0)}{x-x_0}\geq \frac{1}{c^*}$ (see \cite[Corollary 3.17]{FL21Bla}). On the other hand, minimality for any $c<c^*$ from Theorem \ref{th:mainthm} ensures that the reverse inequality holds, so that the function $g$ admits a tangent at $(x_0,y_0)$ with coefficient $(c^*)^{-1}$. The precise result is thus the following.

\begin{corollary}\label{cor:BS} Let $B\subset\R^2$ be a ball of unit volume and set $(x_0,y_0):=(P(B),\lambda_1(B))$. Let $c^*$ be given by \eqref{eq:defc*}. Then the upper boundary of the diagram \eqref{eq:diag_BS_stab} admits a tangent at $(x_0,y_0)$ with coefficient $(c^*)^{-1}$, i.e. $g'(x_0)=(c^*)^{-1}$.\end{corollary}

\end{itemize}
Although the convexity constraint is a natural class for proving this strong form of stability of the ball, our result opens up the question as to finding the more general class for which this could hold. Since a weak form of stability holds in two dimensions for perturbations of the ball which are simply connected (by the Payne-Weinberger inequality), we believe that it would be interesting to investigate whether the sharp stability we obtained could be extended to this class. As shown in \cite[Remark 6.2]{DL19}, one cannot hope for such a general class when $N\geq3$, but one can however make the same conjecture for the class of Lipschitz perturbations of the ball (see also further (ii) from Proposition \ref{prop:lambda1_stab}).

\subsection{Strategy of proof}\label{sect:strat_stabbb}

Although the two results are independent, the strategy we employ for proving them follows a general scheme, which is recurrent in the literature and not specific to convexity (see \cite{DL19} for a detailed description). We rely on the following steps:
\begin{enumerate}
\item \textit{Fuglede-type computations:} minimality of the ball for the functional among ``smooth" perturbations of the ball (Theorems \ref{th:cap_Fug_stab} and \ref{cor:min_smooth}).
\item Local minimality of the ball for convex sets (Theorems \ref{th:stab_cap} and \ref{th:mainthm}).
\end{enumerate}

The first step of this strategy refers to the seminal work \cite{F89} of  B. Fuglede, where the author obtained it for the perimeter functional. By ``smooth" perturbations of the ball $B$ in the first step we mean that minimality holds for domains $\Om=(\text{Id}+\xi)(B)$ with $\xi$ lying in some normed space of smoothness $X$ and $\|\xi\|_{X}$ is small enough. Since they are independent of convexity, the respective results Theorems \ref{th:cap_Fug_stab} and \ref{cor:min_smooth} constituting the first step bear interest in themselves. On the other hand, in the second step one studies the regularity of minimizers of the associated problem, aiming to prove that each minimizer $\Om$ can be written $\Om=(\text{Id}+\xi)(B)$ with $\xi\in X$ and $X$ is the space obtained in the first step. The two parts of the strategy are thus closely linked to each other through the choice of the space $X$. 

Let us now explain separately how we proceed for proving Theorems \ref{th:mainthm} and \ref{th:stab_cap}.

\subsubsection{Weak stability with Lipschitz regularity}
For proving Theorem \ref{th:stab_cap} we first perform \textit{Fuglede-type computations} for Lipschitz perturbations of the ball. This is done in Theorem \ref{th:cap_Fug_stab}.
It is an improvement of previous results from \cite[Corollary 5.6]{GNR15}, where the authors perform \textit{Fuglede-type computations} for a class of $C^{1,1}$ sets with curvature uniformly bounded from above. The proof of Theorem \ref{th:cap_Fug_stab} relies on a second-order estimate of the variation of the capacity for Lipschitz perturbations of the ball, shown in Lemma \ref{lem:cap}. To obtain this latter bound we take advantage of the fact that the capacity is defined as a minimum, thus enabling us to estimate it from above by providing a natural competitor, for which only low regularity is needed.

Theorem \ref{th:stab_cap} is then obtained from Theorem \ref{th:cap_Fug_stab} by using Lipschitz regularity of convex sets as in \cite{F89}. In reference to the two steps strategy described above, note that in this case the passage from the \textit{Fuglede-type computations} to local minimality of the ball in the class of convex sets is quite direct, due to the fact that we are able to perform these computations for a space $X$ with low regularity.

\subsubsection{Strong stability with a $C^{1,\alpha}$ regularity theory}
Since it was proven in \cite[Theorem 1.2]{N14} that minimality in \eqref{eq:goal_ineq} holds among smooth perturbations of the ball for any $c\in(0,c^*)$, the idea of Theorem \ref{th:mainthm} is to pass from smooth to non-smooth convex perturbations of the ball in the minimality. The strategy we will use is the so-called \textit{selection principle}, which was first introduced by Cicalese and Leonardi in \cite{CL12} as a means to give a new proof to the sharp quantitative isoperimetric inequality. The robustness of their method allowed it to be employed in many other contexts for proving various inequalities for shapes, among which we can quote the \textit{sharp quantitative Faber-Krahn inequality} proven in \cite{BdPV15}. The strategy consists in a refinement of the two steps method described above. Roughly speaking, if one wants to prove local minimality of the ball of unit volume $B$ for a functional $\mathcal{J}$ among a class $\mathcal{A}$ of shapes, the idea is to reduce the proof of the inequality in $\mathcal{A}$ to the inequality in a class of smooth shapes through a regularizing procedure. This is usually based on a regularity theory related to the functional $\mathcal{J}$ under study.

In order to apply this \textit{selection principle} method we first need to prove \textit{Fuglede-type computations} for the functional $P-c\lambda_1$ for $C^{1,\alpha}$ perturbations of the ball (in Theorem \ref{cor:min_smooth}), which are not contained in the available results in the literature (in \cite{N14,DL19,D02,DP00}; see Section \ref{sect:IT} for a justification). On the other hand, to perform the second step of the strategy we prove a convergence result for quasi-minimizers of the perimeter under convexity constraint (Corollary \ref{thm:cor_LP}), which uses the regularity theory from \cite{LP23}. 

Note that the proof of Theorem \ref{th:mainthm} is much more involved than the proof of Theorem \ref{th:stab_cap}. This is related to the fact that it relies on a regularity theory among convex shapes, but is also because in order to prove the \textit{Fuglede-type computations} we are led to perform very technical computations (see the proof of Theorem \ref{th:IClambda}).

\subsection{Plan of the paper}
Section \ref{sect:cap} is dedicated to the proof of Theorem \ref{th:stab_cap}.
Sections \ref{sect:IT} and \ref{sect:select} are independent of this first section, and deal with proving Theorem \ref{th:mainthm}: in Section \ref{sect:IT} we show the first step of the \textit{selection principle} method by proving minimality of the ball for $P-c\lambda_1$ in a $C^{1,\alpha}$ neighborhood; then, in Section \ref{sect:select} we perform the regularizing procedure in itself. We provide a small appendix in Section \ref{sect:app}.\\

\noindent\textbf{Acknowledgements: }
The author is deeply grateful to J. Lamboley for very helpful discussions and careful readings of previous versions of this document. The author also thanks M. Goldman, M. Novaga and B. Ruffini for valuable discussions about this work. The author warmly thanks R. Petit for interesting discussions about convergence of smooth sets. Finally, the author wishes to thank the referees for their precious readings and comments. This work was partially supported by the project ANR-18-CE40-0013 SHAPO financed by the French Agence Nationale de la Recherche (ANR).

\section{Stability of the ball for an isoperimetric problem with convexity constraint involving capacity. Proof of Theorem \ref{th:stab_cap}}\label{sect:cap}

 In this section we prove Theorem \ref{th:stab_cap}, which is the stability result associated to \eqref{eq:pb_cap1st}. Recall that we have defined in \eqref{eq:capN=3} the capacity functional in dimensions $N\geq3$. 
It is proven in \cite[Theorem 2 and 3]{CoSa} that provided $K$ is sufficiently smooth, the minimization in \eqref{eq:capN=3} is uniquely solved by the so-called \textit{capacitary function} $u_K$ which verifies
\begin{equation}\label{eq:cap_funct}\begin{cases}-\Delta u_K=0 \text{ in } \R^N\setminus K\\ u_K=1 \text{ over }  K,\ u_K\in C^0(\R^N)\\ u_K(x)\rightarrow0 \text{ as } |x|\rightarrow+\infty\end{cases}\end{equation}
The number $\text{Cap}(K)$ also appears in the asymptotic expansion of $u_K$: 
\[ u_K(x)\sim|x|^{2-N}\left(\text{Cap}(K)\sigma_N(N-2)^{-(N-1)/(N-2)}\right)\ \text{ as }|x|\rightarrow+\infty\]
 with $\sigma_N$ denoting the $(N-1)$-dimensional measure of the unit sphere in $\R^N$. Finally, let us also note that $\text{Cap}(K)=(I_{N-2}(K))^{-1}$ where $I_{N-2}$ is the Riesz potential energy which is given by 
\[I_{N-2}(K):=\inf\left\{\iint_{\R^N\times\R^N}|x-y|^{2-N}d\mu(x)d\mu(y),\ \mu\in \mathcal{P}(K)\right\}\]
with $\mathcal{P}(K)$ denoting the set of probabilities supported on $K$ (see \cite[Remark 2.5]{GNR15}).

When $N=2$, 
one must proceed differently to define the capacity (see further Remark \ref{rk:2dcap}).
 Stability of the ball among convex sets in two dimensions was obtained in \cite[Corollary 1.3]{GNR18} using the two steps strategy we described in the Introduction, by proving (i) \textit{Fugldede-type computations} in a certain class of "smooth" perturbations and (ii) regularity of minimizers of \eqref{eq:pb_cap_N=2}. Theorem \ref{th:stab_cap} is the $N\geq3$ version of this result. Instead, here we prove (i) for Lipschitz perturbations (in Theorem \ref{th:cap_Fug_stab} below), which will be enough in order to obtain local minimality of the ball for convex sets without having to prove regularity of the minimizers. The result is as follows.
 
\begin{theorem}[Fuglede-type computations for $P+\eps\text{Cap}^{-1}$: minimality for Lipschitz perturbations]\label{th:cap_Fug_stab}
Let $N\geq3$, and let $B$ denote the centered unit ball. For $h\in W^{1,\infty}(\partial B)$ we denote $B_h:=\{tx(1+h(x)),\ t\in[0,1),\ x\in\partial B\}$. There exists $\eta>0$ and $\eps_0>0$ such that for all $h\in W^{1,\infty}(\partial B)$ verifying $\|h\|_{W^{1,\infty}(\partial B)}\leq\eta$ with $|B_h|=|B|$ and such that $B_h$ has barycenter at the origin, and for all $\eps\in(0,\eps_0)$, then
\[P(B_h)+\eps\emph{Cap}(B_h)^{-1} \geq P(B)+\eps\emph{Cap}(B)^{-1} \]
with equality only if $B_h=B$.
\end{theorem}
 
 The proof of Theorem \ref{th:cap_Fug_stab} importantly relies on Lemma \ref{lem:cap} below, which consists in a weak \textbf{(IT)} property (see the statement of Theorem \ref{th:ITJc} for a strong \textbf{(IT)} property).

\subsection{Weak (IT) property}  For $h\in W^{1,\infty}(\partial B)$ with $\|h\|_{L^\infty(\partial B)}\leq1/2$ we set $B_h$ the Lipschitz open set 
\[B_h:=\{tx(1+h(x)),\ t\in[0,1),\ x\in\partial B\}.\]

In the following Lemma we estimate from above the variation of $\text{Cap}$ for a Lipschitz perturbation $B_h$ of $B$ in terms of the $H^1$ norm of $h$. Since the Lemma does not use the convexity of the sets $B_h$ it is stated for general $h\in W^{1,\infty}(\partial B)$.

\begin{lemma}[Weak \textbf{(IT)}$_{H^1,W^{1,\infty}}$]\label{lem:cap}
Let $N\geq3$. There exists $C_N>0$ such that if $h\in W^{1,\infty}(\partial B)$ with $B_h$ of volume $|B_h|=|B|$ and $\|h\|_{L^{\infty}(\partial B)}\leq1/2$ then
\[\emph{Cap}(B_h)-\emph{Cap}(B)\leq C_N\|h\|^2_{H^1(\partial B)}.\]
\end{lemma}

\begin{proof}[Proof of Lemma \ref{lem:cap}]
 Fix $h\in W^{1,\infty}(\partial B)$ with $\|h\|_{L^\infty(\partial B)}\leq1/2$ and $|B_h|=|B|$. We extend $h$ over $\R^N\setminus \{0\}$ by setting $h(x):=h(x/|x|)$, thus getting a $0$-homogenous function $h:\R^N\to\R$. Let then $\phi_h$ be the Lipschitz homeomorphism 
\begin{eqnarray}\nonumber\phi_h:&\R^N&\longrightarrow \R^N\\&x&\longmapsto \frac{x}{(1+h(x))}.\end{eqnarray}
Let us denote more simply $\widetilde{h}:=(1+h)^{-1}$, so that $\phi_h(x)= \widetilde{h}(x)x$. We have $D\phi_h(x)=\widetilde{h}(x)\text{Id}+x\otimes\nabla\widetilde{h}(x)$, so that using the formula $\text{det}(\text{Id}+a\otimes b)=1+a\cdot b$ for vectors $a,b$ and since $\nabla \widetilde{h}(x)\cdot x=0$ ($\widetilde{h}$ being constant on half-lines $\{\lambda x,\ \lambda\geq0\}$) it holds $\text{det}(D\phi_h)= \widetilde{h}^N$. As a consequence, thanks to the change of variable $y=\phi_h(x)$ and using polar coordinates, the hypothesis $|B_h|=|B|$ reads
\begin{equation}\label{eq:volh}\int_{\partial B}(1+h)^N=\int_{\partial B}1\end{equation}
Expanding $(1+h)^N-1=Nh+\sum_{i=2}^N\binom{N}{i}h^i$ we thus get that there exists $C_N>0$ such that
\begin{equation}\label{eq:h_secondorder}\text{if } \|h\|_{L^\infty(\partial B)}\leq1/2,\  \int_{\partial B}h\leq C_N\|h\|^2_{L^2(\partial B)}\end{equation}

We set $u_B(x):=\min\{1,|x|^{2-N}\}$ the capacitary function for $B$, which is the unique solution to \eqref{eq:cap_funct} for $K=B$.

\textbf{Step 1\footnote{This argument for admissibility of $u_B\circ\phi_h$ was suggested to us by M. Goldman, M. Novaga and B. Ruffini.}}: $\text{Cap}(B_h)\leq \int_{\R^N}|\nabla (u_B\circ\phi_h)|^2$. For $R\geq\frac{3}{2}$ (note that we thus have $B_h\subset B_R(0)$) let $\theta_R:\R^N\rightarrow\R$ be some cut-off function such that 
\[\begin{cases} \theta_R\equiv 1 \text{ on } B_R(0),\ \theta_R\equiv0 \text{ on } \R^N\setminus B_{2R}(0)\\
\|\theta_R\|_{L^\infty(\R^N)}\leq 1,\ \|\nabla\theta_R\|_{L^\infty(\R^N)}\leq \frac{1}{R}\end{cases}\]
Then $v_R(x):=\theta_R(x)(u_B\circ\phi_h)(x)\in W^{1,\infty}(\R^N)$ and has compact support, so that by standard mollification $v_R$ can be approached in $H^1$ norm by $C^{\infty}_c(\R^N)$ functions. As a consequence by \eqref{eq:capN=3} we get 
\begin{equation}\label{eq:vR_ad}\text{Cap}(B_h)\leq \int_{\R^N}|\nabla v_R|^2\end{equation} We now verify that $\nabla v_R\rightarrow \nabla (u_B\circ\phi_h)$ in $L^2(\R^N)$ as $R\rightarrow+\infty$. 

We have 
\begin{align}\nonumber\|\nabla v_R-\nabla (u_B\circ\phi_h)\|_{L^2(\R^N)}&=\|\nabla v_R-\nabla (u_B\circ\phi_h)\|_{L^2(\R^N\setminus B_R)}\\ \nonumber
& \leq\|(\theta_R-1)\nabla (u_B\circ\phi_h)\|_{L^2(\R^N\setminus B_R)}+\|\nabla\theta_R\cdot(u_B\circ\phi_h)\|_{L^2(\R^N\setminus B_R)}
\\&\leq\|\nabla (u_B\circ\phi_h)\|_{L^2(\R^N\setminus B_R)}+\|\nabla\theta_R\cdot(u_B\circ\phi_h)\|_{L^2(\R^N\setminus B_R)}\label{eq:vR_compet}\end{align}
Since $\nabla (u_B\circ\phi_h)=(D\phi_h)^T\nabla u_B\circ\phi_h$ with $\phi_h\in W^{1,\infty}(\R^N)$ and $\nabla u_B\in L^2(\R^N)$, it holds 
\[\int_{\R^N}|\nabla (u_B\circ\phi_h)|^2\leq\|D\phi_h\|_{L^\infty(\R^N)}^2\|\text{det}(D\phi_h)\|_{L^\infty(\R^N)}^{-1}\int_{\R^N}|\nabla u_B|^2<\infty\] Hence $\|\nabla (u_B\circ\phi_h)\|_{L^2(\R^N\setminus B_R)}\rightarrow0$. On the other hand, as $\|1+h\|_{L^\infty(\R^N)}\leq 3/2$ we have for $|x|\geq3/2$,

\[|u_B\circ\phi_h(x)|=|\phi_h(x)|^{2-N}\leq \left(\frac{3}{2}\right)^{N-2}|x|^{2-N}\]
Since $|\nabla\theta_R|\leq R^{-1}$, using polar coordinates this gives for $R\geq3/2$
\begin{align*}\int_{B_{2R}\setminus B_R}|\nabla\theta_R\cdot(u_B\circ\phi_h)|^2&\leq P(B)\left(\frac{3}{2}\right)^{N-2}R^{-2}\int_{R}^{2R}r^{2(2-N)}r^{N-1}dr\\
&=P(B)\left(\frac{3}{2}\right)^{N-2}\begin{cases}(4-N)^{-1}R^{-2}((2R)^{4-N}-R^{4-N}) \text{ if } N\neq4\\
R^{-2}\log(2)\text{ if }N=4 \end{cases}
\end{align*}

In any case we thus get 
\[\int_{B_{2R}\setminus B_R}|\nabla\theta_R\cdot(u_B\circ\phi_h)|^2\rightarrow0 \text{ as } R\rightarrow+\infty.\]
Pluging this into \eqref{eq:vR_compet} we deduce in fact
\[\|\nabla v_R-\nabla (u_B\circ\phi_h)\|_{L^2(\R^N)}\rightarrow0,\text{ as } R\rightarrow+\infty\] 
so that
\[\text{Cap}(B_h)\leq \int_{\R^N}|\nabla (u_B\circ\phi_h)|^2\]
thanks to \eqref{eq:vR_ad}.

\textbf{Step 2: Estimate of the energy of $u_B\circ\phi_h$.}
Recall that
\begin{equation}\label{eq:Dphi}\forall x\in \R^N,\ D\phi_h(x)=\widetilde{h}(x)\text{Id}+x\otimes\nabla\widetilde{h}(x)\end{equation}
so that
\begin{align}\nonumber\left|\nabla(u_B\circ\phi_h)\right|^2&=\left|(D\phi_h)^T\cdot\nabla u_B\circ\phi_h\right|^2\\\nonumber&=\left|\widetilde{h}\cdot(\nabla u_B\circ\phi_h)+\nabla \widetilde{h}\left((\nabla u_B\circ\phi_h)\cdot x\right))\right|^2 \end{align}Since $u_B$ is radial, 
the vector $\nabla u_B(\phi_h(x))$ is thus proportionnal to $\phi_h(x)$ and hence to $x$, which yields $\nabla u_B(\phi_h(x))\cdot \nabla \widetilde{h}(x)=0$ (because $\widetilde{h}$ is constant in the direction $x$). We deduce 
\[\left|\nabla(u_B\circ\phi_h)\right|^2=\widetilde{h}^2|\nabla u_B\circ\phi_h|^2+ ((\nabla u_B\circ\phi_h)\cdot x)^2|\nabla \widetilde{h}|^2.\] 
We can therefore write the energy of $u_B\circ\phi_h$ as follows
\begin{equation}\label{eq:energy_ubcircphi}\int_{\R^N}|\nabla(u_B\circ\phi_h)|^2=\int_{\R^N}\widetilde{h}^2|\nabla u_B\circ\phi_h|^2+\int_{\R^N} ((\nabla u_B\circ\phi_h)\cdot x)^2|\nabla \widetilde{h}|^2.\end{equation}

Let us first deal with the term $\int_{\R^N}\widetilde{h}^2|\nabla u_B(\phi_h)|^2$. Recalling that $\text{det}(D\phi_h)= \widetilde{h}^N$ we get by the change of variable $y=\phi_h(x)$ \begin{align}\nonumber\int_{\R^N}\widetilde{h}^2|\nabla u_B\circ\phi_h|^2&=\int_{\R^N}(1+h)^{-2}|\nabla u_B\circ\phi_h|^2\\\nonumber&=\int_{\R^N}(1+h)^{N-2}|\nabla u_B|^2\\\nonumber&= \int_{\R^N}|\nabla u_B|^2+\sum_{i=1}^{N-2}\binom{N-2}{i}\int_{\R^N}h^i|\nabla u_B|^2\nonumber\\&= \text{Cap}(B)+\sum_{i=1}^{N-2}\binom{N-2}{i}\int_{\R^N}h^i|\nabla u_B|^2.\nonumber\end{align} Letting $v(r):=\min\{1,r^{2-N}\}$ be the function such that $v(|x|)=u_B(x)$ for all $x\in\R^N$, then using the co-area formula we have for $i\geq1$ \begin{align}\int_{\R^N}h^i|\nabla u_B|^2&=
%\int_0^\infty dr\int_{|\theta|=r}h^{i}(\theta)v'(r)^2d\theta\nonumber\\&=
\left(\int_0^\infty v'(r)^2r^{N-1}dr\right)\left(\int_{\partial B}h^{i}\right)\nonumber\\ &= a_{N}\int_{\partial B}h^{i}\nonumber\end{align} where we set $a_{N}:=\int_0^\infty v'(r)^2r^{N-1}dr=P(B)^{-1}\int_{\R^N}|\nabla u_B|^2$. For $i\geq2$, since $\|h\|_{L^{\infty}(\partial B)}\leq 1/2$ we thus get
\[\int_{\R^N}h^i|\nabla u_B|^2\leq C_N\|h\|_{L^2(\partial B)}^2,\]
while if $i=1$ thanks to \eqref{eq:h_secondorder} we have $\int_{\partial B}h\leq C_N\|h\|^2_{L^2(\partial B)}$ so that \[\int_{\R^N}h^i|\nabla u_B|^2\leq C_N'\|h\|_{L^2(\partial B)}^2\]
These two give
\begin{equation}\label{eq:energy_ub_1} \int_{\R^N}\widetilde{h}^2|\nabla u_B(\phi_h)|^2\leq \text{Cap}(B)+C_N\|h\|_{L^2(\partial B)}^2\end{equation} for some dimensional constant $C_N>0$.

We now turn to the estimate of $\int_{\R^N} ((\nabla u_B\circ\phi_h)\cdot x)^2|\nabla \widetilde{h}|^2$. Denoting by $a(x):=x/|x|$, then for $x\neq0$ \[\nabla \widetilde{h}(x)=-\widetilde{h}^2Da(x)^T\nabla_\tau h(a(x))\] where $\nabla_\tau$ is the tangential gradient. As each coefficient of $Da(x)$ is controlled by $2|x|^{-1}$, since $\|\widetilde{h}\|_{L^\infty(\partial B)}\leq 2$ this yields \begin{align}\nonumber((\nabla u_B\circ\phi_h)\cdot x)^2|\nabla \widetilde{h}(x)|^2\leq 8|\nabla u_B\circ\phi_h|^2|\nabla_\tau h(a(x))|^2\nonumber\end{align} Changing variables and using polar coordinates in the same fashion as before, this ensures that\begin{equation}\label{estim2}\int_{\R^N}((\nabla u_B\circ\phi_h)\cdot x)^2|\nabla \widetilde{h}|^2\leq C_N\|\nabla_\tau h\|^2_{L^2(\partial B)}\end{equation}
for some $C_N>0$.
Injecting this estimate together with \eqref{eq:energy_ub_1} into \eqref{eq:energy_ubcircphi} finally yields 
\[\int_{\R^N}|\nabla(u_B\circ\phi_h)|^2\leq \text{Cap}(B)+C_N\|h\|_{H^1(\partial B)}^2\] for some $C_N>0$.

\textbf{Conclusion:} Thanks to Step 1 it holds 
\[\text{Cap}(B_h)\leq \int_{\R^N}|\nabla (u_B\circ \phi_h)|^2\]
from which we get using Step 2
\[\text{Cap}(B_h)\leq \text{Cap}(B)+C_N\|h\|_{H^1(\partial B)}^2\]
This finishes the proof of the Lemma.
\end{proof}

\begin{remark}\label{rk:2dcap}
In two dimensions, one defines the capacity functional as follows. We first set the Robin constant $V_K$ of some $K\in\K^2$:
%, which can be defined in three equivalent ways  
letting 
\begin{equation}\label{eq:capR_2d}\text{Cap}_R(K):=\inf\left\{\int_{\R^2}|\nabla u|^2,\ u\in H_0^1(B_R(0)),\ u\geq 1 \text{ over } K\right\},\end{equation}
 then we define $V_K:= \lim\left(2\pi\text{Cap}_R(K)^{-1}-\log(R)\right)$ as $R\rightarrow \infty$.
Note that one also has \begin{equation} V_K:=\inf\left\{-\iint_{K\times K}\log(|x-y|)d\mu(x)d\mu(y), \ \mu \in\mathcal{P}(K)\right\}\nonumber \end{equation} 
(see \cite[Theorem 4]{Bag67} and section 3 in \cite{GNR18}). The {\it logarithmic capacity} of $K$ is then given by 
\[\text{Cap}(K):=e^{-V_K}\]

Instead of \eqref{eq:pb_cap1st} we thus set the minimization problem
\begin{equation}\label{eq:pb_cap_N=2}\inf\left\{P(K)+\eps V_K,\ K\in\K^N,\ |K|=1\right\}\end{equation}
 where $\eps>0$ is a small parameter.

We tried to apply the same strategy in the case $N=2$ in order to retrieve the result from \cite{GNR18}. However, due to the specificity of the definition of capacity in the two dimensional case, this brings additional difficulties and we do not know whether the argument could work.
\end{remark}

\subsection{Proof of Theorems \ref{th:cap_Fug_stab} and \ref{th:stab_cap}}
Relying on Lemma \ref{lem:cap}, we first prove the \textit{Fuglede-type computations} for Lipschitz perturbations from Theorem \ref{th:cap_Fug_stab}.

\begin{proof}[Proof of Theorem \ref{th:cap_Fug_stab}]
Let $h\in W^{1,\infty}(\partial B)$ with $\|h\|_{L^\infty(\partial B)}\leq 1/2$, such that $|B_h|=|B|$ and $B_h$ has barycenter at the origin. It is proven in \cite[Theorem 3.1]{Fus15} that there exists $\eta>0$ such that 
\[\text{if } \|h\|_{W^{1,\infty}(\partial B)}\leq\eta,\ P(B_h)-P(B)\geq \frac{1}{4}\|\nabla_\tau h\|^2_{L^2(\partial B)}.\]
 Thanks to \eqref{eq:volh} there exists $\delta>0$ such that if $\|h\|_{L^\infty(\partial B)}\leq \delta$ one has the Poincaré type inequality (see also \cite[Proof of Theorem 3.1]{Fus15}) 
\[\|h\|_{L^2(\partial B)}^2\leq 2\|\nabla_\tau h\|_{L^2(\partial B)}^2.\]

Setting $\widetilde{\eta}:=\min\{\eta,\delta\}$, we thus deduce thanks to Lemma \ref{lem:cap} that if $\|h\|_{W^{1,\infty}(\partial B)}\leq \widetilde{\eta}$ then 
\begin{align*}\text{Cap}(B)^{-1}-\text{Cap}(B_h)^{-1}=\frac{\text{Cap}(B_h)-\text{Cap}(B)}{\text{Cap}(B)\text{Cap}(B_h)} &\leq \frac{C_N}{\text{Cap}(B)^2}\|h\|^2_{H^1(\partial B)} \\
& \leq\frac{3C_N}{\text{Cap}(B)^2}\|\nabla_\tau h\|^2_{L^2(\partial B)}\\
&\leq \frac{12C_N}{\text{Cap}(B)^2}\left(P(B_h)-P(B)\right)\end{align*}
where we also used the isocapacitary inequality in the first line. Taking $\eps_0:=\frac{\text{Cap}(B)^2}{12C_N}$ we get in fact
\[P(B_h)+\eps\text{Cap}(B_h)^{-1} \geq P(B)+\eps\text{Cap}(B)^{-1} ,\]
for any $\eps\in(0,\eps_0)$.
Furthermore, the equality case enforces $P(B_h)=P(B)$, so that $h=0$ by following the chain of inequalities. This concludes the proof of the Theorem.
\end{proof}

We are now ready to prove Theorem \ref{th:stab_cap}.

\begin{proof}[Proof of Theorem \ref{th:stab_cap}]
Let us first note that \eqref{eq:pb_cap1st} admits a solution for any $\eps>0$ (it was proven in \cite[Theorem 1.1]{GNR18}).

Due to convexity, any $K\in \K^N$ with barycenter at the origin can always be written
\[K=\left\{tx(1+h_K(x)),\ t\in[0,1],\ x\in\partial B\right\}\] where $h_K:\partial B\rightarrow\R$ is such that $1+h_K(x)=\sup\{r\geq0, r x\in K\}$ is the distance function to the origin.

Let $K\in \K^N$ be minimizing \eqref{eq:pb_cap1st}, and suppose that $K$ has barycenter at the origin. Let $\eta>0$ and $\eps_0>0$ be given by Theorem \ref{th:cap_Fug_stab}. Now, by minimality 
\begin{align*}P(K)-P(B)&\leq \eps\frac{\text{Cap}(K)-\text{Cap}(B)}{\text{Cap}(K)\text{Cap}(B)}\\
&\leq\frac{\eps}{\text{Cap}(B)}\end{align*}
Thanks to \cite[Lemma 3.3]{Fus15} there exists $\delta_\eta>0$ such that provided $P(K)-P(B)\leq\delta_\eta$ then $\|h_K\|_{W^{1,\infty}(\partial B)}\leq\eta$. We therefore deduce that if $\eps\in(0,\text{Cap}(B)\delta_\eta)$, we have in fact $\|h_K\|_{W^{1,\infty}(\partial B)}\leq\eta$, so that taking $\eps\in\left(0,\min\{\eps_0,\delta_\eta\text{Cap}(B)\}\right)$ we can apply Theorem \ref{th:cap_Fug_stab} to get that $K$ is a ball. This finishes the proof.
\end{proof}

\subsection{Further stability results}
The strategy we employed for proving Theorem \ref{th:stab_cap}  can be adapted to the case of $\lambda_1$. In fact, one can proceed likewise to get a result analogous to the Lemma \ref{lem:cap} below in the case of $\lambda_1$ (see (i) in Proposition \ref{prop:lambda1_stab}), leading to a result of the same type as Theorem \ref{th:stab_cap}. Note that in this case the minimality of the ball is only local, since the functional has no global minimizer (as one sees by taking a sequence of long thin rectangles of unit volume). Let us state the stability result for the sake of clarity.

\begin{prop}[Weak stability of the ball for $\lambda_1$] \label{prop:lambda1_Itw} Let $N\geq2$. There exists $\eps_0>0$ and $\delta>0$ such that for any $\eps \in(0,\eps_0)$,
\[\forall K\in\K^N, |K|=1 \text{ with }|K\Delta B|\leq \delta, \ (P-\eps\lambda_1)(K)\geq (P-\eps\lambda_1)(B).\]
\end{prop}

Although this gives a simple proof of the stability of the ball in the case of $\lambda_1$, this Proposition is strictly weaker than the stronger result we prove in Theorem \ref{th:mainthm}, since the range of $\eps>0$ for which the ball is locally minimal is not optimal (and was already known, see Section \ref{sect:intro_stab_cccc}).
On the other hand, since the analogues of Lemma \ref{lem:cap} and Theorem \ref{th:cap_Fug_stab} in the case of $\lambda_1$ do not explicitly appear in the literature (up to our knowledge), we think it might be of interest to state them rigorously. This is the object of the next result.

\begin{prop}\label{prop:lambda1_stab}
Let $N\geq2$. 
\begin{enumerate}[label=(\roman{*})]
\item(Weak (IT)$_{H^1,W^{1,\infty}}$). There exists $C_N>0$ such that if $h\in W^{1,\infty}(\partial B)$ with $B_h$ of volume $|B_h|=|B|$ and $\|h\|_{L^{\infty}(\partial B)}\leq1/2$ then
\[\lambda_1(B_h)-\lambda_1(B)\leq C_N\|h\|^2_{H^1(\partial B)}.\]
\item(Stability of the ball for Lipschitz perturbations). There exists $\eta>0$ and $\eps_0>0$ such that for all $h\in W^{1,\infty}(\partial B)$ verifying $\|h\|_{W^{1,\infty}(\partial B)}\leq\eta$ with $|B_h|=|B|$ and such that $B_h$ has barycenter at the origin, and for all $\eps\in(0,\eps_0)$, then
\[P(B_h)-\eps\lambda_1(B_h)\geq P(B)-\eps\lambda_1(B)\]
with equality only if $B_h=B$.
\end{enumerate}
\end{prop}
Let us comment on the second item of this Proposition. While this non-optimal stability of the ball for Lipschitz perturbations is implied by the Payne-Weinberger inequality \eqref{eq:PW} in two dimensions, it does not seem to be known in the case $N\geq3$. It opens up the question regarding a natural class of sets for which the optimal stability from Theorem \ref{th:mainthm} might hold: can one prove optimal stability of the ball for Lipschitz perturbations?

\section{Minimality of the ball in a $C^{1,\alpha}$ neighborhood}\label{sect:IT}

The goal of this section consists in proving the first step of the \textit{selection principle} strategy as we described it in the Introduction, namely that the ball is a strict (up to translation) minimum in a $C^{1,\alpha}$ neighborhood of the functional $P-c\lambda_1$ for $c\in(0,c^*)$ and any $\alpha\in(\frac{1}{2},1)$. This is stated in next result.

Let us set a few preliminary notations. The notation $B$ refers to the open ball of unit volume centered at $0$. In this section, any function $h:\partial B\rightarrow\R$ defined on the sphere is extended to the whole of $\R^N$ by setting $h(x):=\theta(x)h(\frac{x}{|x|})$ (for some smooth $\theta$ with $\theta\equiv1$ near $\partial B$) into  a compactly supported function as smooth as $h$ and which is constant near $\partial B$ along directions normal to $\partial B$. Note that this extension is different from the one we make in Section \ref{sect:cap} (in Lemma \ref{lem:cap}). For any $\alpha\in(0,1]$ and $h\in C^{1,\alpha}(\partial B)$ we denote $\xi_h(x):=h(x)x$, so that $\|\xi_h\|_{(C^{1,\alpha}(\R^N))^N}\leq C\|h\|_{C^{1,\alpha}(\partial B)}$, for some $C>0$. We set 
\[B_h=\{tx(1+h(x)),\ t\in[0,1),\ x\in\partial B\}=(\text{Id}+\xi_h)(B),\]
which is a bounded Lipschitz open set provided $\|\xi_h\|_{(W^{1,\infty}(\R^N))^N}<1$ (which we will always assume in the remainder of this section), with boundary $\partial B_h=\{x(1+h(x)),\ x\in\partial B\}$. 

\begin{theorem}[Fuglede-type computations for $P-c\lambda_1$: minimality for $C^{1,\alpha}$ perturbations]\label{cor:min_smooth} Let $N\geq2$. For $c>0$ set $\mathcal{J}_c:=P-c\lambda_1$ and let $c^*$ be given by \eqref{eq:defc*}. For any $\alpha\in(\frac{1}{2},1)$ and $0<c<c^*$ there exists $\delta_{c,\alpha}>0$ such that for all $h\in C^{1,\alpha}(\partial B)$ with $\|h\|_{C^{1,\alpha}(\partial B)}\leq\delta_{c,\alpha}$ and $|B_h|=|B|$ then
%$\|h\|_{C^{1,1}(\partial B)}\leq C_0$  $\|h\|_{C^{1,\alpha}(\partial B)}\leq \delta_c$ 
\[\ \mathcal{J}_c(B_h)\geq \mathcal{J}_c(B)\] 
with equality only if (up to translating) $B_h=B$.
\end{theorem}
Note that this result is of interest in itself, in  particular no convexity constraint of the sets $B_h$ is assumed.
Let us emphasize on the importance of the space $C^{1,\alpha}$ in which we obtain minimality of the ball, regarding the general goal of proving minimality of $\mathcal{J}_c$ for all convex shapes (see Theorem \ref{th:mainthm}). In fact, Theorem  \ref{cor:min_smooth} is to be compared to \cite[Proposition 5.5]{DL19}, where the authors get minimality of the ball for the same interval of $c$'s but in a $W^{2,p}$ neighborhood (for any $p>N$), which was an improvement of previous works for $C^{2,\alpha}$ perturbations from \cite{DP00,D02}. Note that neither Theorem  \ref{cor:min_smooth} nor \cite[Proposition 5.5]{DL19} implies the other result, as the Hölder space $C^{1,\alpha}$ and the Sobolev space $W^{2,p}$ are not comparable in general. On the other hand, the cited $C^{2,\alpha}$ and $W^{2,p}$ results are not enough to apply the \textit{selection principle} strategy, as this procedure does not give more than convergence in any $C^{1,\alpha}$ sense of quasi-minimizers (see Corollary \ref{thm:cor_LP} and Remark \ref{rk:optH}).
 
One of the main ingredients of the proof of Theorem \ref{cor:min_smooth} consists in proving a so-called {\bf (IT)} property for the functional \[\mathcal{J}_c:=P-c\lambda_1.\] This is achieved in Theorem \ref{th:ITJc}. This property was introduced in \cite[p. 3012]{DL19}, and describes a suitable second-order Taylor expansion at the ball $B$ for the functional $\mathcal{J}_c$, where one identifies the remainder as the product of some ``weak" Sobolev norm of the perturbation by something which goes to $0$ as the perturbation goes to $0$ in a much stronger sense (see Theorem \ref{th:ITJc} for a precise statement).

Let us first define the notion of shape differentiability for a shape functional. If $\Om\subset\R^N$ is a bounded open set and $\xi\in W^{1,\infty}(\R^N,\R^N)$ we denote by $\Om_\xi:=(\text{Id}+\xi)(\Om)$ the open Lipschitz deformation of $\Om$ by $\xi$.  
\begin{definition}\label{def:diffGeneral}
 Let $N\geq2$. Let $\mathcal{J}:\{\Om\subset \R^N,\ \Om \text{ open bounded}\}\rightarrow\R$ be a functional. Let $\Om\subset \R^N$ be open bounded and let $X \subset W^{1,\infty}(\R^n,\R^n)$ be a normed space. For $k\in\{1,2\}$ we say that $\mathcal{J}$ is $k$-times shape differentiable around $\Om$ (for the space $X$) in the direction $\xi\in X$ if the function \[\mathcal{J}_\Om:\xi\in X\mapsto \mathcal{J}(\Om_{\xi})\] is $k$-times differentiable Fréchet-differentiable in a neighborhood of $0$. We denote by $\mathcal{J}_\Om'(\xi)\in \mathcal{L}^1(X,\R)$ (respectively $\mathcal{J}_\Om''(\xi)\in \mathcal{L}^2(X\times X,\R)$) the first (respectively second) derivative at $\xi\in X$.
\end{definition}

\begin{remark}Note that although $\mathcal{J}_\Om'(\xi)$ and $\mathcal{J}_\Om''(\xi)$ are \textit{a priori} linear and bilinear continuous forms over $X$, provided the set $\Om$ enjoys some regularity properties it happens very often that they can be naturally extended to spaces of much lower regularity; for instance, the perimeter functional $P$ has its first derivative continuous for the $L^2$ norm, while its second derivative can be continuously extended in $H^1$. In the case of $\lambda_1$ it is respectively the $L^2$ and $H^{1/2}$ spaces over which the first and second derivatives can be defined (see for instance \cite[Lemma 2.8]{DL19}). \end{remark}
We can now state the second main result of this section.

\begin{theorem}[\textbf{(IT)} property for $\mathcal{J}_c$]\label{th:ITJc}Let $N\geq2$ and $\alpha\in(\frac{1}{2},1)$. For $c>0$ set $\mathcal{J}_c:=P-c\lambda_1$. For any $c>0$
%for any $c>0$ the functional $\mathcal{J}_c$ verifies an \emph{\textbf{(IT)}}$_{H^1,C^{1,\alpha}}$ condition at the ball $B$. More precisely, for any $c>0$ %and $C_0>0$ 
there exists $\delta_{c}>0$ and a modulus of continuity $\omega_c$ such that for all $h\in C^{1,\alpha}(\partial B)$ with 
%$\|h\|_{C^{1,1}(\partial B)}\leq C_0$ 
$\|h\|_{C^{1,\alpha}(\partial B)}\leq \delta_c$ it holds
\[\mathcal{J}_c(B_h)=\mathcal{J}_c(B)+(\mathcal{J}_c)_B'(0)\cdot (\xi_h)+ \frac{1}{2}(\mathcal{J}_c)_B''(0)\cdot (\xi_h,\xi_h)+\omega_c(\|h\|_{C^{1,\alpha}(\partial B)})\|h\|^2_{H^1(\partial B)}\] \end{theorem}

This condition was defined in \cite{DL19} as the \textbf{(IT)}$_{H^1,C^{1,\alpha}}$ condition, meaning that the functional $\mathcal{J}_c$ verifies a second-order Taylor expansion with the remainder term behaving as the product of $\|h\|_{H^1}^2$ with a modulus of continuity of $\|h\|_{C^{1,\alpha}}$.
It was shown by Fuglede in \cite{F89} that the perimeter satisfies a stronger \textbf{(IT)}$_{H^1,W^{1,\infty}}$ condition (see for instance \cite[Proposition 4.5]{DL19}: there exists $\omega_P$ a modulus of continuity and $\delta_P>0$ such that for all $h\in W^{1,\infty}(\partial B)$ with $\|h\|_{W^{1,\infty}(\partial B)}\leq \delta_P$ it holds
\begin{equation} \label{eq:IT_P}P(B_h)=P(B)+P_B'(0)\cdot (\xi_h)+\frac{1}{2}P_B''(0)\cdot(\xi_h,\xi_h)+\omega_P\left(\|h\|_{W^{1,\infty}(\partial B)}\right)\|h\|^2_{H^1}\end{equation}
 As a consequence, proving Theorem \ref{th:ITJc} reduces to proving an \textbf{(IT)}$_{H^1,C^{1,\alpha}}$ condition result for $\lambda_1$ (see the proof of Theorem \ref{th:ITJc} for this reduction),  and in fact we will prove a stronger \textbf{(IT)}$_{H^{1/2},C^{1,\alpha}}$ condition for $\lambda_1$ (for any $\alpha\in(\frac{1}{2},1)$). In order to do so we follow the strategy laid out by \cite{DL19}: it will be convenient to show that $\lambda_1$ verifies the so-called condition \textbf{(IC)}$_{H^{1/2},C^{1,\alpha}}$, as stated in next Theorem, which constitutes the core result of this section.

\begin{theorem}[\textbf{(IC)} property for $\lambda_1$]\label{th:IClambda}Let $N\geq2$ and $\alpha\in(\frac{1}{2},1)$. For any $t\in [0,1]$ and $h\in C^{1,\alpha}(\partial B)$ let $\lambda_1(t):=\lambda_1(B_{th})$. Then the functional $\lambda_1$ satisfies an \textbf{\emph{(IC)}}$_{H^{1/2},C^{1,\alpha}}$ condition at the ball $B$, {\it i.e.} there exists $\delta>0$ and a modulus of continuity $\omega_{\lambda_1}$ such that for any $h\in C^{1,\alpha}(\partial B)$ with $\|h\|_{C^{1,\alpha}(\partial B)}\leq\delta$ we have
\begin{equation}\label{eq:lambda_1_t}\forall t\in[0,1],\ |\lambda_1''(t)-\lambda_1''(0)|\leq \omega_{\lambda_1}(\|h\|_{C^{1,\alpha}(\partial B)})\|h\|^2_{H^{1/2}(\partial B)}.\end{equation}
\end{theorem}

The proof of this result is inspired by the strategy of \cite[Theorem 1.4]{DL19} for proving that $\lambda_1$ satisfies an \textbf{(IC)}$_{H^{1/2},W^{2,p}}$ condition. Nevertheless, as $C^{1,\alpha}$ functions may not be twice differentiable even in a weak sense, some estimates require a refined analysis (see Lemma \ref{lem:geom}) and new methods (see Lemma \ref{th:lemma_vt_mod}). We believe that this result is of independent interest, since it goes strictly below spaces with second derivatives as $C^{2,\alpha}$ or $W^{2,p}$ spaces, which are the usual spaces for which this kind of property is obtained (see for instance \cite{D02} or \cite{AFM13}).

Let us mention that in order to prove Theorem \ref{th:IClambda} we will first prove it for functions $h\in C^{1,1}(\partial B)$ instead of $h\in C^{1,\alpha}(\partial B)$, as it will allow us to consider the second-order geometric quantities of $B_h$ (such as the mean curvature and second fundamental form) in the classical sense as functions of $L^\infty(\partial B_h)$, thus easing the computations (in particular in Lemmas \ref{lem:geom}, \ref{lem:lambda''Om} and \ref{lem:lambda''t}). We then remove this additional regularity assumption by a density argument.

The expression of $\lambda_1''(t)$ (see Lemma \ref{lem:lambda''t}) involves both PDE-type terms and geometric terms. We start this section by proving three preparatory Lemmas in Section \ref{sect:contOm}, which provide continuity-type estimates in the domain $\Om$ of the quantities involved in the expression of $\lambda_1''(t)$. 

\subsection{Continuity in the domain $\Om$.}\label{sect:contOm}

In this section we prove three preparatory Lemmas. Lemmas \ref{th:lemma_vt_mod} and \ref{lem:est_v'} will be useful for us to estimate the PDE terms in the variation $|\lambda_1''(t)-\lambda_1''(0)|$, while Lemma \ref{lem:geom} will enable us to estimate the geometric terms. Note that we will also use them as a means to justify the expression of $\lambda_1''(t)$ from Lemma \ref{lem:lambda''t}. Let us set some notations for this section.\\

\textbf{Geometric notation.}
Let $\Om$ be a $C^{1,1}$ bounded open set. For the remainder of this section we consider a vector field $\xi\in W^{1,\infty}(\R^N,\R^N)$ such that $\|\xi\|_{(W^{1,\infty}(\R^N))^N}<1$, so that $(\text{Id}+\xi)(\Om)$ is a Lipschitz open set. 

We consider $\xi\in C^{1,1}(\R^N,\R^N)$ and set $\Om_\xi$ the $C^{1,1}$ open set $(\text{Id}+\xi)(\Om)$.
The operator $\nabla_{\tau_\xi}$ denotes the tangential gradient over $\partial \Om_\xi$, $\text{div}_{\tau_\xi}$ and $D_{\tau_\xi}$ respectively the tangential divergence and jacobian. Setting $n_\xi\in C^{0,1}(\partial\Om_\xi)$ the outer unit normal of $\Om_\xi$ (in particular $n_0$ denotes the outer unit normal of $\Om$), we set $H_\xi:=\text{div}_{\tau_\xi}(n_\xi)\in L^\infty(\partial \Om_\xi)$ (respectively $b_\xi:=D_{\tau_\xi}n_\xi\in (L^\infty(\partial \Om_\xi\times\partial \Om_\xi))^{N\times N}$) the mean curvature (respectively, second fundamental form) on $\partial \Om_\xi$. 

Letting $\phi_\xi$ be the Lipschitz homeomorphism $\text{Id}+\xi$, when a function $f_\xi$ is defined on $\Om_\xi$ (respectively $\partial\Om_\xi$) we denote $\widehat{f_\xi}$ the function $f_\xi\circ \phi_\xi$ defined over $\Om$ (repectively $\partial \Om$). 
We also introduce $\tilde{J}_\xi$ the surface Jacobian from $\partial \Om$ to $\partial \Om_\xi$ given by the expression $\tilde{J}_\xi= \text{det}(D\phi_\xi)|D\phi_\xi^{-T}n_0|$, meaning that for a $C^1(\partial \Om_\xi)$ function $f_\xi$ we have
\begin{equation}\label{eq:change_surface}
\int_{\partial \Om_\xi}f_\xi=\int_{\partial\Om}\tilde{J}_\xi\widehat{f_\xi}.
\end{equation}

\textbf{PDE notation.} For $\xi\in W^{1,\infty}(\R^N,\R^N)$, we denote by $v_\xi$ the first $L^2$ normalized and nonnegative Dirichlet eigenfunction over $\Om_\xi$, and set $\lambda_\xi:=\lambda_1(\Om_\xi)$. 

We start by proving a continuity-type estimate of $v_\xi$ and $\lambda_\xi$ in the spirit of \cite[Lemma 4.8]{DL19}.

\begin{lemma}\label{th:lemma_vt_mod}Let $\Om$ be a $C^{1,1}$ bounded open set. Let $0<\alpha'<\alpha<1$. There exists a modulus of continuity $\omega:\R^+\rightarrow\R^+$ only depending on $\alpha$, $\alpha'$ and $\Om$ such that for all $\xi\in C^{1,\alpha}(\R^N,\R^N)$ it holds
\begin{equation}\label{eq:vtmod}\|\widehat{v_\xi}-v_0\|_{C^{1,\alpha'}(\oOm)}\leq\omega(\|\xi\|_{(C^{1,\alpha}(\R^N))^N}).\end{equation}
Moreover, it holds
\begin{equation}\label{eq:lambdatmod}\lambda_\xi\rightarrow\lambda_0\text{ as }\|\xi\|_{(W^{1,\infty}(\R^N))^N}<1\text{ and }\|\xi\|_{(L^\infty(\R^N))^N}\rightarrow0.\end{equation}
\end{lemma}
To prove \eqref{eq:vtmod} we adapt the method used by \cite[Proposition 4.1]{DP00}, which is based on a compactness argument itself relying on a bound for an appropriate norm of $\widehat{v_\xi}$. 
\begin{proof} 
{\bf Proof of \eqref{eq:lambdatmod}.}
The condition on $\xi$ ensures that the $\Om_\xi$ are uniformly Lipschitz open sets, and the result therefore follows for instance from \cite[Theorem 2.3.18]{H06}. 

{\bf Proof of estimate \eqref{eq:vtmod}.} The proof is divided in two steps.

\textbf{Step 1: $C^{1,\alpha}$ bound of $\widehat{v_\xi}$.} Let us first prove that provided $\xi\in C^{1,\alpha}(\R^N,\R^N)$ verifies $\|\xi\|_{(C^{1,\alpha}(\R^N))^N}\leq C_1$ for some $C_1>0$ then it holds

\begin{equation}\label{eq:C1apha_vt}\|\widehat{v_\xi}\|_{C^{1,\alpha}(\oOm)}\leq C_2\end{equation} for some constant $C_2$ independent of $\xi$. 

This bound relies on standard elliptic estimates. In fact, the equation verified by $v_{\xi}$
\begin{equation}\label{eq:vt}\begin{cases}-\Delta v_{\xi}=\lambda_\xi v_{\xi}\ \text{ in } \Om_\xi,\\ v_{\xi}\in H^1_0(\Om_\xi).\end{cases}\end{equation} translates into the following  elliptic equation for $\widehat{v_\xi}$ 
\begin{equation}\begin{cases}-\text{div}(A_\xi\nabla \widehat{v_\xi})=\lambda_{\xi}J_\xi \widehat{v_\xi}\ \text{ over } \Om\\ \widehat{v_\xi}\in H^1_0(\Om)\end{cases} \text{where}\ \begin{cases} J_\xi:=\text{det}(\text{Id}+D\xi)\\ A_\xi:=J_\xi(\text{Id}+D\xi)^{-1}\left((\text{Id}+D\xi)^{-1}\right)^T\end{cases} \label{eq:AtJt}\end{equation}
We now apply first order Schauder estimates (see \cite[Theorems 8.33, 8.34]{GT}) to get 
\begin{equation}\label{eq:1storderS}\|\widehat{v_\xi}\|_{C^{1,\alpha}(\oOm)}\leq C\left(\|\widehat{v_\xi}\|_{L^{\infty}(\Om)}+\|\lambda_\xi J_\xi \widehat{v_\xi}\|_{L^{\infty}(\Om)}\right)\end{equation} where $C=C_N(\gamma_\xi,\|A_\xi\|_{(C^{0,\alpha}(\oOm))^{N\times N}})$ with $\gamma_\xi$ the ellipticity constant of $A_\xi$. Now, there exists $C>0$ such that for all $\xi$ it holds \begin{equation}\label{eq:coeffS}\|J_\xi\|_{L^\infty(\Om)}, \|A_\xi\|_{(C^{0,\alpha}(\oOm))^{N\times N}}\leq C\left(1+\|\xi\|_{(C^{1,\alpha}(\oOm))^N}\right).\end{equation}

By assuming that $\|\xi\|_{(W^{1,\infty}(\R^N))^N}\leq \delta$ for some $\delta$ small enough we can suppose that $\gamma_\xi\geq 1/2$, and also that $\Om_\xi$ contains a fixed ball for any $\xi$, thus ensuring that $\lambda_\xi$ is bounded thanks to the monotonicity of $\lambda_1$. Moreover, we have the $L^\infty$ bound (see \cite[Example 2.1.8]{Dav})
\begin{equation}\label{eq:vtLinfty}\|\widehat{v_\xi}\|_{L^{\infty}(\Om)}=\|v_\xi\|_{L^{\infty}(\Om_\xi)}\leq e^{1/8\pi}\lambda_\xi^{N/4}.\end{equation}
Inserting \eqref{eq:coeffS}
%\eqref{eq:LipS}
and \eqref{eq:vtLinfty} into \eqref{eq:1storderS} provides the desired estimate \eqref{eq:C1apha_vt}.

{\bf Step 2.} We proceed by contradiction, therefore assuming that there exists $\eps_0>0$ and a sequence $\|\xi_j\|_{(C^{1,\alpha}(\R^N))^N}\rightarrow0$ such that 
\begin{equation}\label{eq:lemmamod_abs}\forall j\geq0 ,\ \|\widehat{v_{\xi_j}}-v_0\|_{C^{1,\alpha'}(\oOm)}\geq\eps_0\end{equation}
Thanks to the bound \eqref{eq:C1apha_vt} we can use the Arzela-Ascoli theorem to infer the existence of $v\in C^{1,\alpha}(\oOm)$ such that up to subsequence
\[\widehat{v_{\xi_j}}\rightarrow v \text{ in } C^{1,\alpha'}({\oOm})\] 
Now, since $\|\xi_j\|_{C^{1,\alpha}(\R^N,\R^N)}\rightarrow0$ we have that $A_{\xi_j}$ and $J_{\xi_j}$ go respectively to $\text{Id}$ and $1$ in $C^{0,\alpha}(\oOm)$, and furthermore $\lambda_{\xi_j}\rightarrow\lambda_0$ (thanks to \eqref{eq:lambdatmod}). We can therefore pass to the limit in the sense of distribution in \eqref{eq:AtJt} to get 
\[\begin{cases}-\Delta v=\lambda_0 v,\ \text{ in } \Om\\ v\in H^1_0(\Om)
\end{cases}\]
Now, since we also have $v\geq0$ and $\|v\|_{L^2(\Om)}=1$ we deduce that $v=v_0$, which enters in contradiction with \eqref{eq:lemmamod_abs}. This concludes the proof of estimate \eqref{eq:vtmod} and hence the proof of the Lemma.

\end{proof}

In the following Lemma we prove continuity type-estimates of several geometric quantities associated to a set $\Om$. Let us recall that most notations of the Lemma have been set at the beginning of Section \ref{sect:contOm} ($n_\xi$, $\widetilde{J_{\xi}}$, $H_\xi$, $b_\xi$ and so on). We also denote  $\alpha_\xi:=n_\xi\cdot n_0$ and $\beta_\xi:=\alpha_\xi n_\xi-n_0$. 
\begin{lemma}\label{lem:geom}Let $\Om$ be a $C^{1,1}$ bounded open set. For any $\alpha\in(0,1)$ there exists $C=C(\alpha)>0$ and $\delta=\delta(\alpha)>0$ independent of $\xi\in C^{1,\alpha}(\R^N,\R^N)$ such that if $\|\xi\|_{C^{1,\alpha}(\R^N,\R^N)}\leq \delta$ then 
\begin{itemize}[label=\textbullet]
\item $\|\widetilde{J_{\xi}}-1\|_{C^{0,\alpha}(\partial \Om)}\leq C\|\xi\|_{(C^{1,\alpha}(\R^N))^N}$.
\item $\|\widehat{n_\xi}-n_0\|_{C^{0,\alpha}(\partial\Om)}\leq C\|\xi\|_{(C^{1,\alpha}(\R^N))^N}$, $\|\widehat{\alpha_\xi}-1\|_{C^{0,\alpha}(\partial\Om)}\leq C\|\xi\|_{(C^{1,\alpha}(\R^N))^N}$, $\|\widehat{\beta_\xi}\|_{C^{0,\alpha}(\partial \Om)}\leq C\|\xi\|_{(C^{1,\alpha}(\R^N))^N}$.
\end{itemize}
Let $\alpha\in(0,1)$ and $1-\alpha<s<1$. Let $p\in(1,\infty)$ and denote by $p'$ its conjugate exponent. There exists $\delta>0$ such that if
$\xi\in C^{1,1}(\R^N,\R^N)$ with $\|\xi\|_{(C^{1,\alpha}(\R^N))^N}\leq \delta$, we have the following expansions
\[\widehat{H_\xi}-H_0=\omega_{s,p,\alpha}(\xi), \ \ \ \widehat{b_\xi}-b_0=\omega_{s,p,\alpha}(\xi),\]
\[\widehat{\nabla_{\tau_\xi}\alpha_\xi}=\omega_{s,p,\alpha}(\xi)\]
where the notation $\omega_{s,p,\alpha}(\xi)$ means that there exists $a_{1,\xi}$, $a_{2,\xi}$, $b_{1,\xi}$, $b_{2,\xi}$  (independent of $s$, $p$ and $\alpha$) such that $\omega_{s,p \alpha}(\xi)=a_{1,\xi}b_{1,\xi}+a_{2,\xi}b_{2,\xi}$ with 
\begin{eqnarray*}&\|a_{1,\xi}\|_{W^{-s,p'}(\partial\Om)}\leq C\|\xi\|_{(C^{1,\alpha}(\R^N))^N}, &\|b_{1,\xi}\|_{C^{0,\alpha}(\partial\Om)}\leq C,\\
\\ &\|a_{2,\xi}\|_{W^{-s,p'}(\partial\Om)}\leq C, &\|b_{2,\xi}\|_{ C^{0,\alpha}(\partial\Om)}\leq C\|\xi\|_{(C^{1,\alpha}(\R^N))^N}.\end{eqnarray*} 
\end{lemma}

\begin{proof}All of these estimates rely on an appropriate expression for $n_\xi$. Following \cite[Lemma 4.3]{DL19} we write $\Om=\{w<0\}$ for some $w\in C^{1,1}(\R^N)$ with $\nabla w$ not vanishing in a neighborhood of $\partial \Om$, so that $\Om_\xi=\{w\circ \phi_\xi^{-1}<0\}$ and 
\[n_\xi=\frac{\nabla (w\circ \phi_\xi^{-1})}{|\nabla (w\circ \phi_\xi^{-1})|}=\frac{D\phi_\xi^{-T}\nabla w(\phi_\xi^{-1})}{|D\phi_\xi^{-T}\nabla w(\phi_\xi^{-1})|}.\] 

Notice that $n_\xi$ only involves $\nabla w\in C^{0,1}(\partial \Om)$ and $D\xi\in C^{0,\alpha}(\R^N,\R^N)$. As a consequence, expanding the maps $A\mapsto(A^{-1})^T$ and $A\mapsto \text{det}(A)$ around $\text{Id}$ and $y\mapsto |y|$ around $n_0$ we get in fact $\|\widetilde{J_{\xi}}-1\|_{C^{0,\alpha}(\partial \Om)}\leq C\|\xi\|_{(C^{1,\alpha}(\R^N))^N}$. 
As for $\widehat{n_\xi}$, $\widehat{\alpha_\xi}$ and $\widehat{\beta_\xi}$, we expand $x\mapsto \frac{x}{|x|}$ around $\frac{\nabla w}{|\nabla w|}$ and get likewise the announced estimates.

The case of $\widehat{H_\xi}$, $\widehat{b_\xi}$ and $\widehat{\nabla_{\tau_\xi}\alpha_\xi}$ is more involved as the second derivatives of $\xi$ and $w$ come into play. As the argument is analogous in the three cases we only prove the estimate for $\widehat{H_\xi}$. Write $a(x):=\frac{x}{|x|}$ and $\psi_\xi:=D\phi_\xi^{-T}\nabla w(\phi_\xi^{-1})$. Then
\[\widehat{H_\xi}=\text{div}(a\circ\psi_\xi)\circ \phi_\xi=Da(\psi_\xi\circ\phi_\xi):D\psi_\xi^T(\phi_\xi)\]
where $:$ is the matrix dot product. In particular, one has
\[H_0=\text{div}(a\circ \nabla w)=Da(\nabla w):D^2 w.\]
Writing $z_\xi=(D\phi_\xi^{-T}(\phi_\xi)-\text{Id})\nabla w$, we let $c_1:=D\psi_\xi^{T}(\phi_\xi)-D^2w$ and $c_2:=Da(\nabla w+z_\xi)-Da(\nabla w)$. We therefore rewrite $\widehat{H_\xi}=(Da(\nabla w)+c_2):(D^2w+c_1)$ and we thus want to estimate
\begin{align}
    \nonumber \widehat{H_\xi}-H_0&=(Da(\nabla w)+c_2):(D^2w+c_1)-Da(\nabla w):D^2 w\\
    &=Da(\nabla w):c_1+c_2:D^2w+c_2:c_1.\label{eq:Hxi}
    \end{align}

By expanding $c_1$ at $\xi=0$, we see that it is a sum of terms of the form $d_{ijk}\partial_{ij}\xi_k$ where $d_{ijk}$ only involves first derivatives of $w$ and $\xi$, and of terms of the form $d'_{ij}\partial_{ij}w$ where $d'_{ij}$ only involves first derivatives of $w$ and $\xi$ with $\|d'_{ij}\|_{C^{0,\alpha}(\partial\Om)}\leq C \|\xi\|_{(C^{1,\alpha}(\R^N))^N}$. 
Using \ref{prop:difflaw} and the embedding $C^{0,\alpha}(\partial\Om)\subset W^{1-s,p'}(\partial\Om)$ from \ref{prop:Sobinj} of Proposition \ref{prop:factFspq} (note that we have in fact $\alpha>1-s$),  there exists $C>0$ such that
\begin{align*}\|\partial_{ij}\xi_k\|_{W^{-s,p'}(\partial \Om)}&\leq C \|\nabla \xi_k\|_{W^{1-s,p'}(\partial \Om)}\\
&\leq C\|\xi\|_{(C^{1,\alpha}(\R^N))^N}.
\end{align*}
On the other hand, it holds $\|d_{ijk}\|_{C^{0,\alpha}(\partial\Om)}\leq C$. Proceeding likewise for the terms $d_{ij}'\partial_{ij}w$ we deduce that $c_1$ has the form $c_1=\omega_{s,p,\alpha}(\xi)$.

We now expand $c_2$ at $\xi=0$: one has 

\[c_2=Da(\nabla w+z_\xi)-Da(\nabla w)=\int_0^1D^2a(\nabla w+tz_\xi)\cdot z_\xi dt\]
As $\|z_\xi\|_{(C^{0,\alpha}(\R^N))^N}\leq C \|\xi\|_{(C^{1,\alpha}(\R^N))^N}$, using the same ideas we get
\begin{align*}
\left\|\int_0^1D^2a(\nabla w+tz_\xi)\cdot z_\xi dt\right\|_{C^{0,\alpha}(\partial\Om)}&\leq C\left\|\int_0^1D^2a(\nabla w+tz_\xi)dt\right\|_{C^{0,\alpha}(\partial\Om)}\|z_\xi\|_{(C^{0,\alpha}(\R^N))^N}\\
&\leq C\left(\|w\|_{C^{1,\alpha}(\partial\Om)}+\|z_\xi\|_{(C^{0,\alpha}(\R^N))^N}\right) \|z_\xi\|_{(C^{0,\alpha}(\R^N))^N}\\
&\leq C\|\xi\|_{(C^{1,\alpha}(\R^N))^N}
\end{align*}
for some $C>0$. As a consequence $\|c_2\|_{C^{0,\alpha}(\partial\Om)}\leq C\|\xi\|_{(C^{1,\alpha}(\R^N))^N}$.

Since $c_1=\omega_{s,p,\alpha}(\xi)$ and $\|c_2\|_{C^{0,\alpha}(\partial\Om)}\leq C\|\xi\|_{(C^{1,\alpha}(\R^N))^N}$, we deduce from \eqref{eq:Hxi} the announced expansion for $\widehat{H_\xi}$, thus finishing the proof of the Lemma.
\end{proof}

We now prove a final preparatory Lemma, which consists in proving a continuity estimate in terms of $\xi$ and $\theta$ for the $H^1$ norm of the derivative of the first Dirichlet eigenfunction on $(\text{Id}+t\theta)(\Om_\xi)$.

\begin{lemma}\label{lem:est_v'} Let $\Om$ be a $C^{1,1}$ bounded open set and let $\alpha\in(\frac{1}{2},1)$. For any $\xi\in C^{1,\alpha}\left(\R^N,\R^N\right)$ and $\theta\in C^{1,\alpha}\left(\R^N,\R^N\right)$ we let $v_{\xi,\theta}'$ be the derivative at $0$ of the map $t\mapsto v_{(\text{Id}+t\theta)(\Om_\xi)}\in L^2(\R^N)$, where $v_{(\text{Id}+t\theta)(\Om_\xi)}$ denotes the first Dirichlet eigenfunction on $(\text{Id}+t\theta)(\Om_\xi)$. Then there exists a modulus of continuity $\omega$ such that for all $\xi$ and $\theta$ with $\|\xi\|_{(C^{1,\alpha}(\R^N))^N}$ and $\|\theta\|_{(C^{1,\alpha}(\R^N))^N}$ sufficiently small it holds
\begin{equation}
    \label{eq:est_v0theta'}
\|v_{0,\theta}'\|_{H^1(\Om)}\leq C\|\theta\|_{(H^{1/2}(\partial\Om))^N}
\end{equation}
and
\begin{equation}\label{eq:estvt'}\|\widehat{v_{\xi,\theta}'}-v_{0,\theta}'\|_{H^1(\Om)}= \omega_{C^{1,\alpha},H^{1/2}}\left(\xi,\theta\right)\end{equation} 
where 
\[\omega_{C^{1,\alpha},H^{1/2}}\left(\xi,\theta\right):= \omega(\|\xi\|_{(C^{1,\alpha}(\R^N))^N})\|\theta\|_{(H^{1/2}(\partial\Om))^N}+\omega(\|\theta\|_{(C^{1,\alpha}(\R^N))^N})\|\xi\|_{(H^{1/2}(\partial\Om))^N}.\]
\end{lemma}

\begin{proof}

We denote $\lambda_{\xi,\theta}' :=(\lambda_1)_{\Om_\xi}'(0)\cdot(\theta)$. It is classical (see for instance \cite[Theorem 5.3.1]{HP18}) that the functions $v_{\xi,\theta}'$ satisfies the following equations
\[\begin{cases}-\Delta v_{\xi,\theta}'=\lambda_{\xi}v_{\xi,\theta}'+\lambda'_{\xi,\theta}v_\xi,\text{ in }\Om_\xi\\ v_{\xi,\theta}'=-(\partial_{n_\xi}v_\xi)\theta\cdot n_\xi,\ \text{ over } \partial \Om_\xi\\ \int_{\Om_\xi}v_{\xi,\theta}'v_\xi=0\end{cases}.\]
Let $\textbf{H}_{\xi,\theta}$ be the harmonic extension on $\Om_\xi$ of $(\partial_{n_\xi}v_\xi)\theta\cdot n_\xi$.
Then recalling the expression of $\lambda_{\xi,\theta}'$ (see for instance \cite[Section 5.9.3]{HP18}) we can write 
\[\lambda_{\xi,\theta}'=-\int_{\partial \Om_\xi}(\partial_{n_\xi}v_\xi)^2\theta\cdot n_\xi=-\int_{\partial\Om_\xi}(\partial_{n_\xi}v_\xi)\textbf{H}_{\xi,\theta}=\lambda_\xi\int_{\Om_\xi}v_\xi\textbf{H}_{\xi,\theta}\]
where we used Green's formula and $\int_{\Om_\xi}\nabla v_\xi\nabla\textbf{H}_{\xi,\theta}=0$. We decompose $v_{\xi,\theta}'=-\pi_{\xi}\textbf{H}_{\xi,\theta}+w_{\xi,\theta}$ 
where $\pi_{\xi}$ 
is the orthogonal projection onto $\{ v_{\xi}\}^{\perp}$ for the $L^2$ scalar product on $L^2(\Om_\xi)$. Thanks to the above expression for $\lambda_{\xi,\theta}'$ we know that

$w_{\xi,\theta}$ solves
\begin{equation}\begin{cases}(-\Delta-\lambda_\xi)w_{\xi,\theta}=-\lambda_\xi\pi_\xi\textbf{H}_{\xi,\theta},\text{ in } \Om_{\xi}\\ w_{\xi,\theta}=0,\text{ over } \partial \Om_{\xi}\\ \int_{\Om_{\xi}}v_\xi w_{\xi,\theta}=0.\end{cases}, 
\label{eq:pdewt}\end{equation}
We estimate separately $\textbf{H}_{\xi,\theta}$ and $w_{\xi,\theta}$. 

{\bf Estimate of $\textbf{H}_{\xi,\theta}$.} Since ${\textbf{H}_{\xi,\theta}}$ is harmonic inside $\Om_{\xi}$, $\widehat{\textbf{H}_{\xi,\theta}}$ satisfies $-\text{div}(A_\xi\nabla \widehat{\textbf{H}_{\xi,\theta}})=0$ in $\Om$, where $A_{\xi}$ was defined in \eqref{eq:AtJt}. As a consequence, we have that \[\Delta(\widehat{\textbf{H}_{\xi,\theta}}-\textbf{H}_{0,\theta})=\Delta\widehat{\textbf{H}_{\xi,\theta}}=-\text{div}((A_\xi-\text{Id})\nabla\widehat{\textbf{H}_{\xi,\theta}})\] so that by using standard elliptic estimates (see for instance \cite[Corollary 8.7]{GT} combined with a trace estimate from \cite[Proposition 3.31]{DD13}) we get
\begin{equation}\label{eq:Harmonic_two}\|\widehat{\textbf{H}_{\xi,\theta}}-\textbf{H}_{0,\theta}\|_{H^1(\Om)}\leq C_N\left(\|(A_\xi-\text{Id})\nabla \widehat{{\textbf{H}_{\xi,\theta}}}\|_{L^2(\Om)}+ \|\widehat{\textbf{H}_{\xi,\theta}}-\textbf{H}_{0,\theta}\|_{H^{1/2}(\partial \Om)}\right)\end{equation} for some $C_N>0$. 
There exists $C>0$ such that
\begin{align}\nonumber\|(A_\xi-\text{Id})\nabla \widehat{\textbf{H}_{\xi,\theta}}\|_{L^2(\Om)}&\leq \|A_\xi-\text{Id}\|_{L^{\infty}(\Om)}\left(\|\nabla \widehat{\textbf{H}_{\xi,\theta}}-\nabla \textbf{H}_{0,\theta}\|_{L^2(\Om)}+\|\nabla \textbf{H}_{0,\theta}\|_{L^2(\Om)}\right)\\&\leq C \|\xi\|_{(W^{1,\infty}(\R^N))^N}\left(\|\widehat{\textbf{H}_{\xi,\theta}}- \textbf{H}_{0,\theta}\|_{H^1(\Om)}+\|\theta\cdot n_0\|_{H^{1/2}(\partial \Om)}\right)\nonumber\end{align}
where we used that \[\|\nabla \textbf{H}_{0,\theta}\|_{L^2(\Om)}\leq \| \textbf{H}_{0,\theta}\|_{H^{1/2}(\partial \Om)}=\|\partial_{n_0}v_0(\theta\cdot n_0)\|_{H^{1/2}(\partial \Om)}\leq C \|\theta\cdot n_0\|_{H^{1/2}(\partial \Om)}\] using $v_0\in C^{1,\alpha}(\oOm)$ and the product law $C^{0,\alpha}(\partial\Om)\cdot H^{1/2}(\partial \Om)\subset H^{1/2}(\partial\Om)$ thanks to $\alpha>\frac{1}{2}$ (see \ref{prop:prodCalpha} of Proposition \ref{prop:factFspq}). Assuming that $\|\xi\|_{(W^{1,\infty}(\R^N))^N}\leq\frac{1}{2C_NC}$ we thus get from \eqref{eq:Harmonic_two} that there exists $C>0$ such that
\begin{equation}\label{eq:estHt_1st}\|\widehat{\textbf{H}_{\xi,\theta}}-\textbf{H}_{0,\theta}\|_{H^1(\Om)}\leq C\left(\|\xi\|_{(W^{1,\infty}(\R^N))^N}\|\theta\|_{(H^{1/2}(\partial\Om))^N}+\|\widehat{\textbf{H}_{\xi,\theta}}-\textbf{H}_{0,\theta}\|_{H^{1/2}(\partial \Om)}\right).\end{equation}
Now, to estimate $\|\widehat{\textbf{H}_{\xi,\theta}}-\textbf{H}_{0,\theta}\|_{H^{1/2}(\partial \Om)}=\|\widehat{\partial_{n_\xi}v_\xi}(\widehat{\theta\cdot n_\xi})-\partial_{n_0}v_0(\theta\cdot n_0)\|_{H^{1/2}(\partial \Om)}$, 
we write
\begin{equation}\label{eq:decomplem}\widehat{\partial_{n_\xi}v_\xi}(\widehat{\theta\cdot n_\xi})-\partial_{n_0}v_0(\theta\cdot n_0)=
\widehat{\partial_{n_\xi}v_\xi}\left((\widehat{\theta\cdot n_\xi})-(\theta\cdot n_0)\right)+(\theta \cdot n_0)\left(\widehat{\partial_{n_\xi}v_\xi}-\partial_{n_0}v_0\right)\end{equation} 
Let us decompose 
\begin{align*}\widehat{\partial_{n_\xi}v_\xi}-\partial_{n_0}v_0&=(\widehat{\nabla v_\xi}-\nabla v_0)\cdot n_0+\widehat{\nabla v_\xi}\cdot(\widehat{n_\xi}-n_0)\\&=\left(D\phi_\xi^{-T}\nabla\widehat{v_\xi}-\nabla v_0\right)\cdot n_0+D\phi_\xi^{-T}\nabla\widehat{v_\xi}\cdot(\widehat{n_\xi}-n_0).\end{align*}
 Choosing some $\alpha'\in(\frac{1}{2},\alpha)$ and relying on Lemmas \ref{th:lemma_vt_mod} and \ref{lem:geom} we thus have $\|\widehat{\partial_{n_\xi}v_\xi}-\partial_{n_0}v_0\|_{C^{0,\alpha'}(\partial\Om)}\leq \omega\left(\|\xi\|_{(C^{1,\alpha}(\R^N))^N}\right)$.
Using again the product law $C^{0,\alpha'}(\partial\Om)\cdot H^{1/2}(\partial \Om)\subset H^{1/2}(\partial\Om)$ we obtain 
\begin{align}\nonumber\|\left(\widehat{\partial_{n_\xi}v_\xi}-\partial_{n_0}v_0\right)(\theta \cdot n_0)\|_{H^{1/2}(\partial\Om)}&\leq C\|\widehat{\partial_{n_\xi}v_\xi}-\partial_{n_0}v_0\|_{C^{0,\alpha'}(\partial \Om)}\|\theta\cdot n_0\|_{H^{1/2}(\partial\Om)}\\
&\leq \omega(\|\xi\|_{(C^{1,\alpha}(\R^N))^N})\|\theta\|_{(H^{1/2}(\partial\Om))^N}\label{eq:decomp_partial}\end{align}
On the other hand, we have
\[\widehat{\theta\cdot n_\xi}-\theta\cdot n_0=(\widehat{\theta}-\theta)\cdot \widehat{n_\xi}+\theta\cdot(\widehat{n_\xi}-n_0)\]
With the same tools as before we get $\|\theta\cdot(\widehat{n_\xi}-n_0)\|_{H^{1/2}(\partial \Om)}\leq C\|\theta\|_{(H^{1/2}(\partial\Om))^N}\|\xi\|_{(C^{1,\alpha}(\R^N))^N}$. For the other term we write
\[\widehat{\theta}(x)-\theta(x)=\int_0^1\nabla\theta(x+t\xi(x))\cdot\xi(x)dt\]
so that the product law $C^{0,\alpha}(\partial\Om)\cdot H^{1/2}(\partial \Om)\subset H^{1/2}(\partial\Om)$ again gives
\begin{align*}\|\widehat{\theta}-\theta\|_{H^{1/2}(\partial\Om)}&\leq C\left\|\int_0^1\nabla\theta(x+t\xi(x))\right\|_{(C^{0,\alpha}(\R^N))^N} \|\xi\|_{(H^{1/2}(\partial\Om))^N} \\
&\leq C\|\theta\|_{(C^{1,\alpha}(\R^N))^N}\left(1+\|\xi\|_{(C^{0,\alpha}(\R^N))^N}\right)\|\xi\|_{(H^{1/2}(\partial\Om))^N}\\
&\leq C\|\theta\|_{(C^{1,\alpha}(\R^N))^N}\|\xi\|_{(H^{1/2}(\partial\Om))^N}
\end{align*}
yielding $\|(\widehat{\theta}-\theta)\cdot \widehat{n_\xi}\|_{H^{1/2}(\partial\Om)}\leq C\|\xi\|_{(H^{1/2}(\partial\Om))^N}\|\theta\|_{(C^{1,\alpha}(\R^N))^N}$. Combining the two estimates we thus get 
\[\|\widehat{\theta\cdot n_\xi}-\theta\cdot n_0\|_{H^{1/2}(\partial\Om)}= \omega_{C^{1,\alpha},H^{1/2}}\left(\xi,\theta\right)\]
This estimate together with \eqref{eq:decomp_partial} enable to estimate $\|\widehat{\textbf{H}_{\xi,\theta}}-\textbf{H}_{0,\theta}\|_{H^{1/2}(\partial \Om)}$ thanks to the decomposition \eqref{eq:decomplem}, so that \eqref{eq:estHt_1st} becomes
\begin{equation}\label{eq:Harmonic_final}\|\widehat{\textbf{H}_{\xi,\theta}}-\textbf{H}_{0,\theta}\|_{H^1(\Om)}= \omega_{C^{1,\alpha},H^{1/2}}\left(\xi,\theta\right). \end{equation} This finishes the proof of the estimate of $\textbf{H}_{\xi,\theta}$.

{\bf Estimate of $w_{\xi,\theta}$.} To estimate $\|\widehat{w_{\xi,\theta}}-w_{0,\theta}\|_{H^1(\Om)}$ let us write the equation verified by $\widehat{w_{\xi,\theta}}-w_{0,\theta}$. If one writes $\mathcal{L}_{\xi}:=\text{div}(A_\xi\nabla\cdot)$, then for a function $f_\xi:\Om_\xi\rightarrow\R$ one has $\widehat{\Delta f_\xi}=\mathcal{L}_\xi\widehat{f_\xi}$. Recalling \eqref{eq:pdewt} we therefore have
\[\begin{cases} (-\Delta-\lambda_0)(\widehat{w_{\xi,\theta}}-w_{0,\theta})=\big[(-\Delta-\lambda_0)-(-\mathcal{L}_\xi-\lambda_\xi)\big]\widehat{w_{\xi,\theta}}-\lambda_\xi\widehat{\pi_\xi\textbf{H}_{\xi,\theta}}+\lambda_0\pi_0\textbf{H}_{0,\theta}\text{ in } \Om\\\widehat{w_{\xi,\theta}}-w_{0,\theta}=0\text{ over } \partial \Om\end{cases}\]
Using that $(-\Delta-\lambda_0)^{-1}:H^{-1}(\Om)\cap\{f\in H^{-1}(\Om), \ \langle f,v_0\rangle=0\}\rightarrow \{v_0\}^\perp\cap H^1_0(\Om)$ is an isomorphism (recall that $\lambda_0$ is simple, so that $-\Delta-\lambda_0$ is one-to-one on $\{v_0\}^\perp$ thanks to the Fredholm alternative), there exists $C_N>0$ such that 
\begin{equation}\nonumber\begin{split}\|\widehat{w_{\xi,\theta}}-w_{0,\theta}-\gamma_{\xi,\theta}v_0\|_{H^1(\Om)}\leq C_N\left(\|(A_\xi-\text{Id})\nabla \widehat{w_{\xi,\theta}}\|_{L^2(\Om)}+|\lambda_\xi-\lambda_0|\|\widehat{w_{\xi,\theta}}\|_{L^2(\Om)}\right.
\\ \left.+\|\lambda_\xi\widehat{\pi_\xi\textbf{H}_{\xi,\theta}}-\lambda_0\pi_0\textbf{H}_{0,\theta}\|_{L^2(\Om)}\right)\end{split}\end{equation} where $\gamma_{\xi,\theta}\in\R$ is chosen so that $\widehat{w_{\xi,\theta}}-w_{0,\theta}-\gamma_{\xi,\theta}v_0\in \{v_0\}^\perp$. Now, since $\|A_\xi-\text{Id}\|_{L^\infty(\Om)}\leq C\|\xi\|_{(W^{1,\infty}(\R^N))^N}$ and $|\lambda_\xi-\lambda_0|=\omega(\|\xi\|_{(L^{\infty}(\R^N))^N})$ (thanks to Lemma \ref{th:lemma_vt_mod}) we get 
\begin{equation}\label{eq:estwt_interm1}\|\widehat{w_{\xi,\theta}}-w_{0,\theta}-\gamma_{\xi,\theta}v_0\|_{H^1(\Om)}\leq C\|\xi\|_{(W^{1,\infty}(\R^N))^N}\left(\|\widehat{w_{\xi,\theta}}\|_{H^1(\Om)}+\|\lambda_\xi\widehat{\pi_\xi\textbf{H}_{\xi,\theta}}-\lambda_0\pi_0\textbf{H}_{0,\theta}\|_{L^2(\Om)}\right)\end{equation}
If we denote $\langle\cdot,\cdot\rangle_\xi$ the scalar product in $L^2(\Om_\xi)$, we write
\begin{equation}\label{eq:decompH-H0}\widehat{\pi_\xi\textbf{H}_{\xi,\theta}}-\pi_0\textbf{H}_{0,\theta}=(\widehat{\textbf{H}_{\xi,\theta}}-\textbf{H}_{0,\theta})-(\langle \textbf{H}_{\xi,\theta},v_\xi\rangle_\xi\widehat{v_\xi}-\langle \textbf{H}_{0,\theta},v_0\rangle_0 v_0)\end{equation} Using estimate \eqref{eq:vtmod} from Lemma \ref{th:lemma_vt_mod} and the Harmonic estimate \eqref{eq:Harmonic_final} we have
\begin{equation}\label{eq:eq:scal_pdc_est}|\langle \textbf{H}_{\xi,\theta},v_\xi\rangle_\xi-\langle \textbf{H}_{0,\theta},v_0\rangle_0|=\left|\int_{\Om}\widehat{ \textbf{H}_{\xi,\theta}}\widehat{v_\xi}J_\xi-\int_\Om \textbf{H}_{0,\theta}v_0\right|= \omega_{C^{1,\alpha},H^{1/2}}\left(\xi,\theta\right)\end{equation} ($J_\xi$ was defined in \eqref{eq:AtJt}) so that using again \eqref{eq:vtmod}, \eqref{eq:lambdatmod} and \eqref{eq:Harmonic_final} we deduce 
\begin{equation}\label{eq:estw_t_pit}\|\lambda_\xi\widehat{\pi_\xi\textbf{H}_{\xi,\theta}}-\lambda_0\pi_0\textbf{H}_{0,\theta}\|_{L^2(\Om)}=\omega_{C^{1,\alpha},H^{1/2}}\left(\xi,\theta\right)\end{equation}
Estimate \eqref{eq:estwt_interm1} thus becomes
\begin{equation}\label{eq:estwt_interm2}\|\widehat{w_{\xi,\theta}}-w_{0,\theta}-\gamma_{\xi,\theta}v_0\|_{H^1(\Om)}\leq C\left(\|\xi\|_{(W^{1,\infty}(\R^N))^N}\|\widehat{w_{\xi,\theta}}\|_{H^1(\Om)}+\omega_{C^{1,\alpha},H^{1/2}}\left(\xi,\theta\right)\right)\end{equation}
Now, since $\int_{\Om_\xi}w_{\xi,\theta}v_\xi=\int_\Om w_{0,\theta}v_{0}=0$ we have
\[|\gamma_{\xi,\theta}|=\left|\int_{\Om}(\widehat{w_{\xi,\theta}}-w_{0,\theta})v_0\right|=\left|\int_\Om\widehat{w_{\xi,\theta}}(\widehat{v_\xi}J_\xi-v_0)\right|\leq \omega(\|\xi\|_{(C^{1,\alpha}(\R^N))^N})\|\widehat{w_{\xi,\theta}}\|_{L^2(\Om)}\]
using again Lemma \ref{th:lemma_vt_mod}, which gives 
\begin{align}\nonumber \|\widehat{w_{\xi,\theta}}-w_{0,\theta}\|_{H^1(\Om)}&\leq \omega(\|\xi\|_{(C^{1,\alpha}(\R^N))^N})\|\widehat{w_{\xi,\theta}}\|_{H^1(\Om)}+\omega_{C^{1,\alpha},H^{1/2}}\left(\xi,\theta\right)\\&\leq \omega(\|\xi\|_{(C^{1,\alpha}(\R^N))^N})\left(\|\widehat{w_{\xi,\theta}}-w_{0,\theta}\|_{H^1(\Om)}+\|w_{0,\theta}\|_{H^1(\Om)}\right)+\omega_{C^{1,\alpha},H^{1/2}}\left(\xi,\theta\right)\end{align}
Using $\|w_{0,\theta}\|_{H^1(\Om)}\leq C\|\textbf{H}_{0,\theta}\|_{L^2(\Om)}\leq C\|\theta\|_{(H^{1/2}(\partial\Om))^N}$ based on \eqref{eq:pdewt}, and taking $\omega(\|\xi\|_{(C^{1,\alpha}(\R^N))^N})\leq \frac{1}{2}$ finally yields
\begin{equation}\label{eq:estwt_final} \|\widehat{w_{\xi,\theta}}-w_{0,\theta}\|_{H^1(\Om)}=  \omega_{C^{1,\alpha},H^{1/2}}\left(\xi,\theta\right).\end{equation}

\textbf{Conclusion:} Working from the decomposition \eqref{eq:decompH-H0} and using the harmonic estimate \eqref{eq:Harmonic_final}, estimate \eqref{eq:eq:scal_pdc_est} and also \eqref{eq:vtmod} from Lemma \ref{th:lemma_vt_mod}, we obtain 
\[\|\widehat{\pi_\xi\textbf{H}_{\xi,\theta}}-\pi_0\textbf{H}_{0,\theta}\|_{H^1(\Om)}= \omega_{C^{1,\alpha},H^{1/2}}\left(\xi,\theta\right).\]
This latter bound together with \eqref{eq:estwt_final} provides the desired \eqref{eq:estvt'}. As for \eqref{eq:est_v0theta'}, we have seen that $\|w_{0,\theta}\|_{H^1(\Om)}\leq C\|\theta\|_{(H^{1/2}(\partial\Om))^N}$ based on \eqref{eq:pdewt}. On the other hand, $\|\pi_0\textbf{H}_{0,\theta} \|_{L^2(\Om)}\leq \|\textbf{H}_{0,\theta}\|_{L^2(\Om)}$ by definition and 
\[\|\nabla\pi_0 \textbf{H}_{0,\theta}\|_{L^2(\Om)}^2=\|\nabla \textbf{H}_{0,\theta}\|_{L^2(\Om)}^2-\langle \textbf{H}_{0,\theta},v_0\rangle^2\]
since $\int_\Om\nabla \textbf{H}_{0,\theta}\nabla v_0=0$, thus yielding also 
\[\|\nabla\pi_0 \textbf{H}_{0,\theta}\|_{L^2(\Om)}\leq\|\textbf{H}_{0,\theta}\|_{H^1(\Om)}\leq C\|\textbf{H}_{0,\theta}\|_{H^{1/2}(\partial \Om)}\leq C\|\theta\|_{(H^{1/2}(\partial\Om))^N}\]
Hence $\|\pi_0\textbf{H}_{0,\theta}\|_{H^1(\Om)}\leq C\|\theta\|_{(H^{1/2}(\partial\Om))^N}$, so that we finally have $\|v_{0,\theta}'\|_{H^1(\Om)}\leq \|\pi_{0}\textbf{H}_{0,\theta}\|_{H^1(\Om)}+\|w_{0,\theta}\|_{H^1(\Om)}\leq C\|\theta\|_{(H^{1/2}(\partial\Om))^N}$, thus giving \eqref{eq:est_v0theta'}.
\end{proof}

\subsection{Second derivative of $\lambda_1$.}
A final preparatory step to estimate $\lambda_1''(t)-\lambda_1''(0)$ is to justify that the expression \eqref{eq:2nd_deriv_lambda} below of the second derivative $(\lambda_1)_\Om''(0)$ is valid when $\Om=(\text{Id}+\xi)(B)$ for some vector field $\xi$ which is only $C^{1,1}$. Formula \eqref{eq:2nd_deriv_lambda} is indeed well-known for $C^3$ domains (see for instance \cite[Theorem 5.9.2 and Section 5.9.6]{HP18}), but for $C^{1,1}$ domains it does not seem to have been justified in the literature. As a matter of fact, the expression \eqref{eq:2nd_deriv_lambda} has been implicitly used in \cite{DL19} without further justification for domains $(\text{Id}+\xi)(B)$ with $\xi\in W^{2,p}(\R^N,\R^N)$. From \eqref{eq:2nd_deriv_lambda} we will immediately deduce a corresponding expression for $\lambda_1''(t)$, see Lemma \ref{lem:lambda''t}.

In this paragraph, if $\xi$ is a Lipschitz vector field and $f_\xi$ is defined on $\Om_\xi$ or $\partial\Om_\xi$ we still write $\widehat{f_\xi}$ the function $f_\xi\circ\phi_\xi$ defined on $\Om$ or $\partial\Om$.
\begin{lemma}\label{lem:lambda''Om}[Second derivative of $\lambda_1$] Let $\Om\subset\R^N$ be a bounded open set given by $\Om=(\text{Id}+\zeta)(B)$ for some $\zeta\in C^{1,1}(\R^N,\R^N)$. Let $n$, $H$ and $b$ denote respectively its outer unit normal, curvature and second fundamental form. Let $\alpha\in(\frac{1}{2},1)$. Denote by $\lambda_\Om:=\lambda_1(\Om)$. Then for any $\theta\in C^{1,\alpha}(\R^N,\R^N)$ it holds
\begin{equation}\label{eq:2nd_deriv_lambda}\lambda_\Om''(0)\cdot(\theta,\theta)= 2\left(\int_{\Om}|\nabla v'|^2-\lambda_\Om\int_\Om|v'|^2\right)+\int_{\partial \Om}(\partial_{n}v)^2\left[H(\theta\cdot n)^2-b(\theta_\tau,\theta_\tau)+2\nabla_\tau(\theta\cdot n)\cdot\theta_\tau\right]\end{equation}
where $\theta_\tau$ denotes the tangential component of $\theta$, $v$ is the first $L^2$ normalized and nonnegative eigenfunction of $\Om$ and $v'$ is uniquely determined by the equations 
\[\begin{cases}-\Delta v'=\lambda_\Om v'+(\lambda_\Om'(0)\cdot \theta) v,  &\text{in }\Om\\
v'=-\partial_n v(\theta\cdot n), &\text{over }\partial\Om\\
\int_\Om v'v=0.
\end{cases}\]
\end{lemma}
 
Note that when $\Om$ is $C^3$ the term $\int_{\Om}|\nabla v'|^2-\lambda_\Om\int_\Om|v'|^2$ is more commonly written as the boundary term $\int_{\partial\Om}v'\partial_nv'$ (which we cannot justify if $\Om$ is merely $C^{1,1}$). 

In our case one does not have enough regularity over $\Om$ to perform the classical integration by parts leading to expression \eqref{eq:2nd_deriv_lambda}. As a consequence, in order to prove \eqref{eq:2nd_deriv_lambda} in the $C^{1,1}$ case we rely on  \eqref{eq:2nd_deriv_lambda} in the smooth case combined with a low-regularity formula of $(\lambda_1)_\Om''(0)\cdot(\theta,\theta)$ which holds true for bounded Lipschitz open sets, proven in \cite[Theorem 2.1]{BuBog22} (see also \cite{Lau20} for an expression for the Dirichlet energy in the same spirit).
\begin{proof}[Proof of Lemma \ref{lem:lambda''Om}]
One can apply the second derivative formula from \cite[Theorem 2.1]{BuBog22} which holds true for any Lipschitz domain and $\theta\in W^{1,\infty}(\R^N,\R^N)$:
\begin{align}\nonumber\lambda_\Om''(0)\cdot(\theta,\theta)=\int_\Om\left(-2|\nabla \dot{v}_\theta|^2+2\lambda_\Om|\dot{v}_\theta|^2+2(\textbf{S}_\Om:D\theta)\text{div}(\theta)+(\lambda_\Om|v|^2-|\nabla v|^2)(\text{div}(\theta)^2+D\theta^T:D\theta)\right)
\\+\int_\Om\left(2(2D\theta^2+D\theta D\theta^T)|\nabla v|^2 -2\lambda_\Om'(0)\cdot(\theta)\text{div}(\theta)|v|^2\right)
\label{eq:lambda''low}\end{align}
where $:$ is the matrix dot product, $\textbf{S}_\Om:=(|\nabla v|^2-\lambda_\Om|v|^2)\text{Id}-2\nabla v\otimes\nabla v$, and $\dot{v}_\theta$ denotes the material derivative of $v$ in the direction $\theta$, meaning the derivative at $0$ in the direction $\theta$ of the function $\theta\in W^{1,\infty}\mapsto v_\theta\circ(\text{Id}+\theta)\in H^1_0(\Om)$ where $v_\theta$ is the first $L^2$ normalized eigenfunction of $\Om_\theta$. Note that $\dot{v}_\theta$ verifies the elliptic equation
\[\begin{cases}-\Delta \dot{v}_\theta-\lambda_\Om \dot{v}_\theta=\lambda_\Om v\text{div}(\theta)+\left(\lambda_\Om'(0)\cdot(\theta)\right)v+\text{div}(A_\theta'\nabla v), \text{ in } \Om\\
\dot{v}_\theta=0 \text{ over }\partial\Om\\
\int_\Om\dot{v}_\theta v=-\frac{1}{2}\int_\Om|v|^2\text{div}(\theta)\end{cases}\]
with $A_\theta':=\text{div}(\theta)\text{Id}-D\theta-D\theta^T$. 

Recall that $\Om=B_\zeta=(\text{Id}+\zeta)(B)$. Since $\zeta\in C^{1,1}(\R^N,\R^N)$, there exists $\zeta_j\in C^\infty(\R^N,\R^N)$ converging locally to $\zeta$ in $C^{1,\beta}$ for each $\beta\in(0,1)$. Letting $\Om_j:=B_{\zeta_j}=(\text{Id}+\zeta_j)(B)$, we have $\Om_j=(\text{Id}+\zeta_j)\circ(\text{Id}+\zeta)^{-1}(B_\zeta)$ so that $\Om_j=(\text{Id}+\xi_j)(\Om)$ with $\xi_j:=(\text{Id}+\zeta_j)\circ(\text{Id}+\zeta)^{-1}-\text{Id}$ converging locally to $0$ in $C^{1,\beta}$ for each $\beta\in(0,1)$. As $\Om_j$ is smooth, $\lambda_{\Om_j}''(0)\cdot(\theta,\theta)$ both equals \eqref{eq:2nd_deriv_lambda} (see \cite[Theorem 5.9.2 and Section 5.9.6]{HP18}) and \eqref{eq:lambda''low} and we will pass to the limit in both expressions.

Through the proof we give an additionnal index $j$ to the notations linked to $\Om_j$: $v_j$ is the first eigenfunction of $\Om_j$,  $\lambda'_j:=\lambda_{\Om_j}'(0)\cdot(\theta)$, $\dot{v}_j$ is the material derivative of $v_j$ in direction $\theta$ and $v_j':=v_{\xi_j,\theta}' $(the notation was introduced in Lemma \ref{lem:est_v'}). As for the geometric quantites related to $\Om_j$, $n_j$ denotes the outer unit normal to $\Om_j$, and so on. The proof of the Lemma is divided into two steps.

\textbf{Step 1: continuity of $\lambda_\Om''(0)\cdot(\theta,\theta)$ in $\Om$.} In this step one can assume that $\theta\in W^{1,\infty}(\R^N,\R^N)$. Let us prove that $\lambda_{\Om_j}''(0)\cdot(\theta,\theta)\rightarrow\lambda_{\Om}''(0)\cdot(\theta,\theta)$.

Since $\Om_j\to\Om$ in $C^{1,\beta}$ (in the sense introduced above), one can find $D\subset\R^N$ open bounded such that $\Om_j,\Om\subset D$ for every $j$. Setting $\gamma_j:=\int_D\dot{v}_jv_j$ we have
\[\|\dot{v}_j-\gamma_jv_j\|_{H^1_0(D)}\leq \|\lambda_jv_j\text{div}(\theta)\|_{H^{-1}(\Om_j)}+\|\lambda'_jv_j\|_{H^{-1}(\Om_j)}+\|\text{div}(A_\theta'\nabla v_j)\|_{H^{-1}(\Om_j)}\]
Now, since $\lambda_j'=\int_{\Om_j}\textbf{S}_j:D\theta$ thanks to \cite[Theorem 2.1]{BuBog22}, and using that $\lambda_j$ is bounded and $v_j$ is bounded in $H^1$ (thanks to Lemma \ref{th:lemma_vt_mod}) we deduce that $\dot{v}_j-\gamma_jv_j$ is bounded in $H^1_0(D)$, yielding in turn that $\dot{v}_j$ is bounded in $H^1_0(D)$ as $\gamma_j=-\frac{1}{2}\int_{\Om_j}|v_j|^2\text{div}(\theta)$. As a consequence, $\dot{v}_j$ converges (up to subsequence) towards some $\tilde{v}\in H^1_0(D)$ weakly in $H^1$, strongly in $L^2$ and almost everywhere. Using again Lemma \ref{th:lemma_vt_mod} we have that $\lambda_j\rightarrow \lambda_\Om$, $\lambda_j'\rightarrow\lambda_\Om'(0)\cdot(\theta)$ and $v_j\rightarrow v$ in $H^1$, so that relying also on the Hausdorff convergence $\Om_j\to\Om$ one can pass to the limit in the sense of distributions in 
\[\begin{cases}-\Delta \dot{v}_j-\lambda_j \dot{v}_j=\lambda_j v_j\text{div}(\theta)+\lambda_j'v_j+\text{div}(A_\theta'\nabla v_j), \text{ in } \Om_j\\
\dot{v}_j=0 \text{ over }\partial\Om_j\\
\int_{\Om_j}\dot{v}_jv_j=-\frac{1}{2}\int_{\Om_j}|v_j|^2\text{div}(\theta)\end{cases}\]
to deduce that $\tilde{v}$ verifies 
\[\begin{cases}-\Delta \tilde{v}-\lambda_\Om \tilde{v}=\lambda_\Om v\text{div}(\xi)+\left(\lambda_\Om'(0)\cdot(\xi)\right)v+\text{div}(A_\theta'\nabla v), \text{ in } \Om\\
\int_\Om\tilde{v}v=-\frac{1}{2}\int_\Om|v|^2\text{div}(\theta)\end{cases}\]
Since $\Om$ is Lipschitz it suffices to prove that $\tilde{v}=0$ a.e. outside $\Om$ to deduce that $\tilde{v}\in H^1_0(\Om)$ (see for instance \cite[Proposition 3.2.16]{HP18}) and therefore that $\tilde{v}=\dot{v}$. But this is seen directly by passing to the limit a.e. in the identity $\textbf{1}_{D\setminus\Om_j}\dot{v}_j=0$, since $\dot{v}_j\to\tilde{v}$ a.e. and $|\Om_j\Delta\Om|\to0$ (as $\Om_j\to\Om$ in $C^{1,\beta}$). Now, the convergence of $\dot{v}_j$ towards $\dot{v}$ is strong in $H^1$, since by multiplying the equation by $\dot{v}_j$ and integrating by parts we get $\|\nabla \dot{v}_j\|_{L^2(D)}\rightarrow\|\nabla \dot{v}\|_{L^2(D)}$. We can therefore pass to the limit as $j\rightarrow+\infty$ in \eqref{eq:lambda''low} to deduce that
we have in fact $\lambda_{\Om_j}''(0)\cdot(\theta,\theta)\rightarrow\lambda_{\Om}''(0)\cdot(\theta,\theta)$.

\textbf{Step 2: continuity of \eqref{eq:2nd_deriv_lambda} in  $\Om$.} In this step we rather assume $\theta\in C^{1,\alpha}(\R^N,\R^N)$ for some $\alpha\in(\frac{1}{2},1)$. We want to pass to the limit in the expression
\begin{equation}\label{eq:dscnde_omj}2\left(\int_{\Om_j}|\nabla v_j'|^2-\lambda_j\int_{\Om_j}|v_j'|^2\right)+\int_{\partial \Om_j}(\partial_{n_j}v_j)^2\left[H_j(\theta\cdot n_j)^2-b_j(\theta_{\tau_j},\theta_{\tau_j})+2\nabla_{\tau_j}(\theta\cdot n_j)\cdot\theta_{\tau_j}\right]\end{equation}
Thanks to a change of variable, the integral on $\Om_j$ can be written
\begin{align*}2\int_{\Om}\left(|\widehat{\nabla v_j'}|^2-\lambda_j|\widehat{v_j'}|^2\right)J_j&=2\int_{\Om}\left(\langle A_j\nabla\widehat{v_j'},\nabla\widehat{v_j'}\rangle-\lambda_jJ_j|\widehat{v_j'}|^2\right)\end{align*}
where $J_j:=J_{\xi_j}$ and $A_j:=A_{\xi_j}$ are defined in \eqref{eq:AtJt}. We have that $A_j$ and $J_j$ go respectively to $\text{Id}$ and $1$ in $L^\infty(\Om)$ as $j\rightarrow+\infty$, and also $\lambda_j\rightarrow \lambda_\Om$, while on the other hand $\widehat{v_j'}$ converges to $v_{0,\theta}'$ in $H^1(\Om)$ thanks to Lemma \ref{lem:est_v'}. As a consequence the integral on $\Om_j$ converges to $2\left(\int_{\Om}|\nabla v'|^2-\lambda_\Om\int_\Om|v'|^2\right)$.

For the convergence of the integral on $\partial\Om_j$, we show that $\int_{\partial \Om_j}(\partial_{n_j}v_j)^2H_j(\theta\cdot n_j)^2\rightarrow \int_{\partial \Om}(\partial_{n}v)^2H(\theta\cdot n)^2$ and the other terms can be treated similarly. We can write
\begin{align*}\int_{\partial \Om_j}(\partial_{n_j}v_j)^2H_j(\theta\cdot n_j)^2&=\int_{\partial\Om}\tilde{J}_j\widehat{(\partial_{n_j}v_j)^2}\widehat{H_j}\widehat{(\theta\cdot n_j)^2}
\\ &= \int_{\partial\Om}\widehat{H_j}\widehat{g_j}\end{align*}
for $\widehat{g_j}:=\tilde{J}_j\widehat{(\partial_{n_j}v_j)^2}\widehat{(\theta\cdot n_j)^2}$, where $\tilde{J}_j:=\tilde{J}_{\xi_j}$ is the surface jacobian defined in \eqref{eq:change_surface}. We claim that $\widehat{g_j}\rightarrow(\partial_{n}v)^2(\theta\cdot n)^2=:g$ in $C^{0,\alpha}$. This comes from Lemmas \ref{th:lemma_vt_mod} and \ref{lem:geom}, and moreover one estimates $\left(\widehat{(\theta\cdot n_j)}-\theta\cdot n\right)$ as follows
\begin{align*}\|\widehat{(\theta\cdot n_j)}-\theta\cdot n\|_{C^{0,\alpha}(\partial\Om)}&\leq \|\theta\cdot(\widehat{n_j}-n)\|_{C^{0,\alpha}(\partial\Om)}+\|(\widehat{\theta}-\theta)\widehat{n_j}\|_{C^{0,\alpha}(\partial\Om)}\\
&\leq \|\theta\|_{(C^{0,\alpha}(\partial\Om))^N}\|\xi_j\|_{(C^{1,\alpha}(\partial\Om))^N}+C\|\theta\|_{(C^{1,\alpha}(\partial\Om))^N}\|\xi_j\|_{(C^{0,\alpha}(\partial\Om))^N}\end{align*}
where the estimate of $\widehat{\theta}-\theta$ comes from writing $\widehat{\theta}-\theta=\int_0^1\nabla\theta((1-t)\cdot+t\phi_{\xi_j})\cdot\xi_jdt$.  
On the other hand, picking $s\in(0,1)$ such that $1-\alpha<s<\alpha$ and any $p\in (1,\infty)$ it holds $\widehat{H_j}-H=\omega_{s,p,\alpha}(\xi_j)$ thanks to Lemma \ref{lem:geom}, so that we also have $\widehat{H_j}\widehat{g_j}-Hg=\omega_{s,p,\alpha}(\xi_j)$ thanks to the convergence of the $\widehat{g_j}$. We may therefore use a $W^{-s,p'}\cdot W^{s,p}$ duality estimate and the embedding $C^{0,\alpha}(\partial \Om)\subset  W^{s,p}(\partial \Om)$ (see \ref{prop:Sobinj} from Proposition \ref{prop:factFspq}) to deduce that 
\[\int_{\partial\Om}\widehat{H_j}\widehat{g_j}\rightarrow\int_{\partial\Om}Hg.\]
This proves the convergence of \eqref{eq:dscnde_omj} towards the corresponding expression for $\Om$, and thus concludes Step 2.

\textbf{Conclusion.} Since each $\Om_j$ is smooth we have that $\lambda_{\Om_j}''(0)\cdot(\theta,\theta)$ both equals \eqref{eq:dscnde_omj} and \eqref{eq:lambda''low} for each $j$. The two previous steps ensure that we can pass to the limit on both sides to deduce that \eqref{eq:2nd_deriv_lambda} holds for $\Om$.
\end{proof}

\subsection{Estimate of $\lambda_1''(t)$ and proof of the main results.}
Relying on the expression of the second derivative given by Lemma \ref{lem:lambda''Om} we can now tackle estimating the variation of the second derivative of $\lambda_1(t)=\lambda_1\left((\text{Id}+t\xi_h)(B)\right)$. We first obtain an expression for $\lambda_1''(t)$, which we state in the next Lemma. 

Let us first recall and set some notations for the remainder of this section.
We use the notations from Theorems \ref{cor:min_smooth}, \ref{th:ITJc} and \ref{th:IClambda}: any $h:\partial B\rightarrow\R$ is extended to the whole of $\R^N$ into some $h$ locally constant in normal directions around $\partial B$, and we then set $\xi_h(x):=h(x)x$ so that $\|\xi_h\|_{(C^{1,\alpha}(\R^N))^N}\leq C\|h\|_{C^{1,\alpha}(\partial B)}$ for any $\alpha\in(0,1]$ and $h\in C^{1,\alpha}(\partial B)$. We let $B_h:=(\text{Id}+\xi_h)(B)$ and $B_t:=B_{th}$ for $t\in[0,1]$. The notations $n_t$, $H_t$ and $b_t$ refer respectively to the outer unit normal, curvature and second fundamental form of $B_t$.

\begin{lemma}\label{lem:lambda''t}Let $h\in C^{1,1}(\partial B)$. For all $t\in[0,1]$ it holds
\begin{equation}\label{eq:lambda''}\lambda_1''(t)=2\left(\int_{B_t}|\nabla v_t'|^2-\lambda_1(t)\int_{B_t}|v_t'|^2\right)+\int_{\partial B_t}(\partial_{n_t}v_t)^2\left[H_t(\xi_h\cdot n_t)^2-b_t((\xi_h)_{\tau_t},(\xi_h)_{\tau_t})+2\nabla_{\tau_t}(\xi_h\cdot n_t)\cdot(\xi_h)_{\tau_t}\right],\end{equation} where $(\xi_h)_{\tau_t}$ is the tangential (over $\partial B_t$) component of $\nabla \xi_h$, $v_t$ is the first $L^2$ normalized eigenfunction of $B_t$ and $v_t'$ is determined by the equations 
\[\begin{cases}-\Delta v_t'=\lambda_1(t)v_t'+\lambda_1'(t)v_t,  &\text{in }B_t\\
v_t'=-\partial_{n_t} v_t(\xi_h\cdot n_t), &\text{over }\partial B_t\\
\int_{B_t} v_t'v_t=0.
\end{cases}\]
\end{lemma}

\begin{proof}[Proof of Lemma \ref{lem:lambda''t}]
By Definition \ref{def:diffGeneral} it holds
\[\lambda_1''(t)=\lambda_B''(t\xi_h)\cdot(\xi_h,\xi_h)\]
but since $B_{s+t}=(\text{Id}+(s+t)\xi_h)(B)=(\text{Id}+s\xi_h)(B_t)$ when $\|h\|_{L^\infty(\partial B)}$ is small we immediately get 
\[\lambda_B''(t\xi_h)\cdot(\xi_h,\xi_h)=\lambda_{B_t}''(0)\cdot(\xi_h,\xi_h)\] 
The result then follows from applying Lemma \ref{lem:lambda''Om} with $\Om=B_t$.
\end{proof}

Relying on the expression of $\lambda_1''(t)$ given by this Lemma we are now ready to prove Theorem \ref{th:IClambda}.
\begin{proof}[Proof of Theorem \ref{th:IClambda}]
Let us first suppose that $h\in C^{1,1}(\partial B)$. We can therefore use the expression of $\lambda_1''(t)$ from Lemma \ref{lem:lambda''t} which we rewrite in the following way 
%the decomposition of \cite[Equation (24)]{DL}: 
\begin{eqnarray}
\lambda_1''(t)&=&2\left(\int_{B_t}|\nabla v_t'|^2-\lambda_1(t)\int_{B_t}|v_t'|^2\right)+\int_{\partial B_t}(\partial_{n_t}v_t)^2\left[H_t\alpha_t^2-b_t(\beta_t,\beta_t)  -2\nabla_{\tau_t}\alpha_t\cdot\beta_t\right]h^2 \nonumber
\\&& -2\int_{\partial B_t}(\partial_{n_t}v_t)^2\alpha_t(\beta_t\cdot \nabla_{\tau_t}h)h \label{eq:lambda_decompo}
\\ &:=&\mathcal{T}_1(t)+\mathcal{T}_2(t)+\mathcal{T}_3(t)\nonumber
\end{eqnarray} 
where we put $\alpha_t:=n_t\cdot n$ and $\beta_t:=\alpha_tn_t-n$. 
We thus prove for each $1\leq i\leq3$
\begin{equation}\label{eq:Ti_est}\forall t\in[0,1],\ |\mathcal{T}_i(t)-\mathcal{T}_i(0)|\leq \omega(\|h\|_{C^{1,\alpha}(\partial B)})\|h\|^2_{H^{1/2}(\partial B)}.\end{equation}

{\bf Estimate of $\mathcal{T}_1(t)$.} See that $v_t'=v_{t\xi_h,\xi_h}'$ in the notations of Lemma \ref{lem:est_v'}. Writing $\int_{B_t}|\nabla v_t'|^2=\int_{B}\langle A_t\nabla \widehat{v_t'},\nabla \widehat{v_t'}\rangle$ where $A_t:=A_{t\xi_h}$ is defined in \eqref{eq:AtJt}, we get
\begin{align}\nonumber\left|\int_{B_t}|\nabla v_t'|^2-\int_{B}|\nabla v_0'|^2\right|&=\left|\int_B\langle (A_t-\text{Id})\nabla \widehat{v_t'},\nabla \widehat{v_t'}\rangle+(\nabla \widehat{v_t'}-\nabla v_0')\cdot(\nabla \widehat{v_t'}+\nabla v_0')\right|\\&\leq \omega(\|h\|_{C^{1,\alpha}(\partial B)})\|h\|^2_{H^{1/2}(\partial B)}\nonumber\end{align} thanks to Lemma \ref{lem:est_v'}. On the other hand, 

\begin{align}\nonumber\left|\lambda_1(t)\int_{B_t}|v_t'|^2-\lambda_1(0)\int_{B}|v_0'|^2\right|&=\left|(\lambda_1(t)-\lambda_1(0))\int_{B_t}|v_t'|^2+\lambda_1(0)\int_B(J_t-1)|\widehat{v_t'}|^2+(\widehat{v_t'}-v_0)(\widehat{v_t}+v_0)\right|\\&\leq \omega(\|h\|_{C^{1,\alpha}(\partial B)})\|h\|^2_{H^{1/2}(\partial B)}\nonumber\end{align} using Lemmas \ref{lem:est_v'} and \ref{th:lemma_vt_mod}. Putting these two together finally yields
\[\left|\mathcal{T}_1(t)-\mathcal{T}_1(0)\right|\leq  \omega(\|h\|_{C^{1,\alpha}(\partial B)})\|h\|^2_{H^{1/2}(\partial B)}\] thus finishing the proof of the estimate of $\mathcal{T}_1(t)$.

{\bf Estimate of $\mathcal{T}_2(t)$.} Thanks to a surface change of variables we have $\mathcal{T}_2(t)=\int_{\partial B}\widehat{\sigma_t}h$ where 
\[\widehat{\sigma_t}:=(\widehat{\partial_{n_t}v_t})^2\left[\widehat{H_t}\widehat{\alpha_t}^2-\widehat{b_t}(\widehat{\beta_t},\widehat{\beta_t})-2\widehat{\nabla_{\tau_t}\alpha_t}\cdot\widehat{\beta_t}\right]\tilde{J}_t\]
%Thus, using 1. from Proposition \ref{prop:factFspq} 
with $\tilde{J}_t$ the surface Jacobian. We have
\begin{equation}\label{eq:T2_duality}\left|\mathcal{T}_2(t)-\mathcal{T}_2(0)\right|=\left|\int_{\partial B}(\widehat{\sigma_t}-\sigma_0)h^2\right|\end{equation}
Notice that $\widehat{\sigma_t}$ is a sum of terms of the form $\widehat{y_t}\times\widehat{ z_t}$ for some $\widehat{y_t}\in\{\widehat{H_t},(\widehat{b_t})_{ij},(\widehat{\nabla_{\tau_t}\alpha_t})_i\}$ with $\widehat{z_t}$ which is a product of terms in $\{\widehat{\alpha_t},\widehat{\partial_{n_t}v_t},\widehat{\beta_t},\tilde{J}_t\}$. Let $p\in(1,2)$ be given by \ref{prop:prodH1/2} from Proposition \ref{prop:factFspq} and choose some $s\in(0,1)$ verifying $1-\alpha<s<\frac{1}{2}<\alpha$, ensuring that the triple $(\alpha,s,p)$ satisfies both the hypotheses of Lemma \ref{lem:geom} and Proposition \ref{prop:factFspq} \ref{prop:prodH1/2}. We thus have that $\widehat{y_t}-y_0=\omega_{s,p,\alpha}(\xi_h)$ (uniformly in $t$), which we denote more simply $\omega_{s,p,\alpha}(h)$. On the other hand, thanks to Lemmas \ref{th:lemma_vt_mod} and \ref{lem:geom}, choosing some $\alpha'\in(\frac{1}{2},\alpha)$ there exists $\omega$ such that $\|\widehat{z_t}-z_0\|_{C^{0,\alpha'}(\partial B)}\leq \omega(\|h\|_{C^{1,\alpha}(\partial B)})$. %(using also the embedding $C^{0,\alpha}(\partial\Om)\subset W^{s,p}(\partial\Om)$, see \ref{prop:Sobinj} from Proposition \ref{prop:factFspq}). 
As a consequence we get that $\widehat{\sigma_t}-\sigma_0=\omega_{s,p,\alpha'}(h)$, which we write $\widehat{\sigma_t}-\sigma_0=a_{1,h}b_{1,h}+a_{2,h}b_{2,h}$ in the notations of Lemma \ref{lem:geom}. 

Thanks to \ref{prop:prodH1/2} from Proposition \ref{prop:factFspq} the product law $H^{1/2}(\partial B)\cdot H^{1/2}(\partial B)\subset W^{s,p}(\partial B)$ holds, and we thus have
\[\|b_{i,h}h^2\|_{W^{s,p}(\partial B)}\leq C \|b_{i,h}h\|_{H^{1/2}(\partial B)}\|h\|_{H^{1/2}(\partial B)}.\]
Using the product law $C^{0,\alpha'}(\partial\Om)\cdot H^{1/2}(\partial \Om)\subset H^{1/2}(\partial\Om)$ from \ref{prop:prodCalpha} of Proposition \ref{prop:factFspq} we get
\[\|b_{1,h}h\|_{H^{1/2}(\partial B)}\leq C\|h\|_{H^{1/2}(\partial B)},\]
\[\|b_{2,h}h\|_{H^{1/2}(\partial B)}=\omega\left(\|h\|_{C^{1,\alpha}(\partial B)}\right)\|h\|_{H^{1/2}(\partial B)}.\]
Using a duality estimate $W^{-s,p'}\cdot W^{s,p}$ 
%(see (i) from Proposition \ref{prop:factFspq})  
we can therefore estimate \eqref{eq:T2_duality} 
\begin{align*}\left|\mathcal{T}_2(t)-\mathcal{T}_2(0)\right|&\leq \|a_{1,h}\|_{W^{-s,p'}(\partial B)}\|b_{1,h}h^2\|_{W^{s,p}(\partial B)}+\|a_{2,h}\|_{W^{-s,p'}(\partial B)}\|b_{2,h}h^2\|_{W^{s,p}(\partial B)}\\
&=  \omega(\|h\|_{C^{1,\alpha}(\partial B)})\|h\|^2_{H^{1/2}(\partial B)}.\end{align*}

{\bf Estimate of $\mathcal{T}_3(t)$.} Using a surface change of variables we have $\mathcal{T}_3(t)=\int_{\partial B}h(\widehat{\rho_t}\cdot \nabla_{\widehat{\tau_t}}h)$ where 
\[\widehat{\rho_t}:=(\widehat{\partial_{n_t}v_t})^2\widehat{\alpha_t}\widehat{\beta_t}\tilde{J}_t,\  \nabla_{\widehat{\tau_t}}h:=\nabla h -(\nabla h\cdot \widehat{n_t})\widehat{n_t}.\]

We write
\begin{align}\nonumber \left|\mathcal{T}_3(t)-\mathcal{T}_3(0)\right|&=\left|\int_{\partial B}h\widehat{\rho_t}\cdot(\nabla_\tau h-\nabla_{\widehat{\tau_t}}h)+\int_{\partial B}h(\widehat{\rho_t}-\rho_0)\cdot \nabla_\tau h\right|.\end{align}

Now, $\nabla_\tau h-\nabla_{\widehat{\tau_t}}h=(\nabla h\cdot \widehat{n_t})\widehat{n_t}=\nabla_\tau h\cdot (\widehat{n_t}-n)\widehat{n_t}$ since $\nabla h\cdot n=0$. On the other hand $\|\widehat{\rho_t}-\rho_0\|_{C^{0,\alpha'}(\partial B)}\leq\omega(\|h\|_{C^{1,\alpha}(\partial B)})$ choosing some $\frac{1}{2}<\alpha'<\alpha$, thanks to Lemmas \ref{th:lemma_vt_mod} and \ref{lem:geom}. Using a duality estimate $H^{-1/2}\cdot H^{1/2}$, the product law $C^{0,\alpha'}\cdot H^{1/2}\subset H^{1/2}$ and the fact that $\|\nabla_\tau h\|_{H^{-1/2}(\partial B)}\leq C\|h\|_{H^{1/2}(\partial B)}$ (see \ref{prop:prodCalpha} and \ref{prop:difflaw} from Proposition \ref{prop:factFspq}) we thus get
\begin{align*}\left|\mathcal{T}_3(t)-\mathcal{T}_3(0)\right|&\leq \|h\widehat{\rho_t}(\widehat{n_t}-n)\widehat{n_t}\|_{H^{1/2}(\partial B)}\|\nabla_\tau h\|_{H^{-1/2}(\partial B)}+\|h(\widehat{\rho_t}-\rho_0)\|_{H^{1/2}(\partial B)}\|\nabla_\tau h\|_{H^{-1/2}(\partial B)}\\
&\leq\omega(\|h\|_{C^{1,\alpha}(\partial B)})\|h\|_{H^{1/2}(\partial B)}^2 
\end{align*}
using Lemma \ref{lem:geom}.
This finishes the proof of the $\mathcal{T}_3(t)$ estimate.\\

We have thus proved \eqref{eq:Ti_est} and hence Theorem \ref{th:IClambda} in the case where $h\in C^{1,1}(\partial B)$. We then reduce the regularity hypothesis made over $h$ to $h\in C^{1,\alpha}(\partial B)$ with a density argument. Recall that $\lambda_1''(t)=\lambda_{B_{th}}''(0)\cdot(\xi_h,\xi_h)$ since $B_{s+t}=(\text{Id}+(s+t)\xi_h)(B)=(\text{Id}+s\xi_h)(B_t)$ when $\|h\|_{L^\infty(\partial B)}$ is small. Let then $h_j$ be smooth and converging to $h$ in $C^{1,\beta}(\partial B)$ for each $0<\beta<\alpha$ with $\|h_j\|_{C^{1,\alpha}(\partial B)}\leq \|h\|_{C^{1,\alpha}(\partial B)}$.  With an argument similar to \textbf{Step 1} of the proof of Lemma \ref{lem:lambda''Om} we can pass to the limit in the expression $\lambda_{B_{th_j}}''(0)\cdot(\xi_{h_j},\xi_{h_j})$ to get that $\lambda_{B_{th_j}}''(0)\cdot(\xi_{h_j},\xi_{h_j})\rightarrow\lambda_1''(t)$. We can thus let $j\to+\infty$ in
\[\left|\lambda_{B_{th_j}}''(0)\cdot(\xi_{h_j},\xi_{h_j})-\lambda_B''(0)\cdot(\xi_{h_j},\xi_{h_j})\right|\leq \omega(\|h_j\|_{C^{1,\alpha}(\partial B)})\|h_j\|^2_{H^{1/2}(\partial B)}\]
and get the desired estimate.
This finishes the proof in the general case.

\end{proof}

Theorem \ref{th:ITJc} is now a consequence of Theorem \ref{th:IClambda}. The way to pass from an \textbf{(IC)} to an \textbf{(IT)} condition was shown in \cite{DL19} (see \cite[p.3014]{DL19}), but we reproduce the short proof for the convenience of the reader.

\begin{proof}[Proof of Theorem \ref{th:ITJc} ] 
Fix $c>0$ and $\alpha\in(\frac{1}{2},1)$. Thanks to Theorem \ref{th:IClambda} we find $\delta>0$ and a modulus of continuity $\omega_{\lambda_1}$ such that for all $h\in C^{1,\alpha}(\partial B)$ 
%with $\|h\|_{C^{1,1}(\partial B)}\leq C_0$ 
with $\|h\|_{C^{1,\alpha}(\partial B)}\leq \delta$ it holds
\begin{equation}\label{eq:lambda''_IC_IT}\forall t\in[0,1],\ |\lambda_1''(t)-\lambda_1''(0)|\leq \omega_{\lambda_1}(\|h\|_{C^{1,\alpha}(\partial B)})\|h\|^2_{H^{1/2}(\partial B)}\end{equation}
We can write
\[\lambda_1(B_h)=\lambda_1(B)+(\lambda_1)_B'(0)\cdot (\xi_h)+\frac{1}{2}(\lambda_1)_B''(0)\cdot (\xi_h,\xi_h)+\int_0^1(\lambda_1''(t)-\lambda_1''(0))(1-t)dt\]
using a second-order Taylor expansion with integral remainder. Using \eqref{eq:lambda''_IC_IT} we thus get 
\[\lambda_1(B_h)=\lambda_1(B)+(\lambda_1)_B'(0)\cdot (\xi_h)+\frac{1}{2}(\lambda_1)_B''(0)\cdot (\xi_h,\xi_h)+\omega_{\lambda_1}(\|h\|_{C^{1,\alpha}(\partial B)})\|h\|^2_{H^{1/2}(\partial B)}\] 
hence that $\lambda_1$ satisfies an \textbf{(IT)}$_{H^{1/2},C^{1,\alpha}}$ condition.
Combining this together with the expansion for $P$ (see \eqref{eq:IT_P}) we get the expansion for $\mathcal{J}_c=P-c\lambda_1$:
\[\mathcal{J}_c(B_h)=\mathcal{J}_c(B)+(\mathcal{J}_c)_B'(0)\cdot (\xi_h)+\frac{1}{2}(\mathcal{J}_c)_B''(0)\cdot (\xi_h,\xi_h)+\omega_c(\|h\|_{C^{1,\alpha}(\partial B)})\|h\|^2_{H^1(\partial B)}\]
with $\omega_c:=\omega_P-c\omega_{\lambda_1}$. This concludes the proof.
\end{proof}

We are now able to prove Theorem \ref{cor:min_smooth}, relying on the stability results proved in \cite{DL19}.
\begin{proof}[Proof of Theorem \ref{cor:min_smooth}]
We proved in Theorem \ref{th:ITJc} that the functional $\mathcal{J}_c$ satisfies an \textbf{(IT)}$_{H^1,C^{1,\alpha}}$ condition in the sense of \cite[Theorem 1.3]{DL19}. It also satisfies a \textbf{(C}$_{H^1}$\textbf{)} hypothesis (see \cite[Lemma 2.8]{DL19}). On the other hand it was proven in \cite[Theorem 1.2]{N14} (see also \cite[Proposition 5.5 (ii)]{DL19}) that $B$ is a critical and strictly stable shape under volume constraint and up to translations for $\mathcal{J}_c$ whenever $c\in (0,c^*)$. For any such $c$ we therefore apply \cite[Theorem 1.3]{DL19} and get that there exists $\delta_c>0$ such that for any $h\in C^{1,\alpha}(\partial B)$ with $\|h\|_{C^{1,\alpha}(\partial B)}\leq\delta_c$ and $|B_h|=|B|$ it holds 
\[\mathcal{J}_c(B_h)\geq \mathcal{J}_c(B)\]
with equality only if $B_h$ is a ball. 
This gives strict minimality (up to translations) of $B$ in a $C^{1,\alpha}$ neighborhood, thus concluding the proof of the Theorem.
\end{proof}

\section{\textit{Selection principle}: minimality of the ball among convex sets. Proof of Theorem \ref{th:mainthm}.} \label{sect:select}
This section is dedicated to the second part of the \textit{selection principle} strategy which we described in the Introduction, namely the regularizing procedure which enables to reduce the proof of the inequality from Theorem \ref{th:mainthm}
\[\forall K\in\K^N \text{ with }|K\Delta B|\leq \delta_c, \ (P-c\lambda_1)(K)\geq (P-c\lambda_1)(B)\]
for general convex perturbations $K$ of $B$ to the same inequality for $C^{1,\alpha}$ perturbations of $B$. As is usual in this procedure (as was originally done by \cite{CL12}, see also among many others \cite{AFM13}, \cite{BdPV15}, \cite{AKN21Lin}) the argument goes by contradiction: we assume that \eqref{eq:goal_ineq} does not hold, meaning that there exists a sequence $(K_j)$ converging to the ball in the $L^1$ sense but for which $\mathcal{J}_c(K_j)<\mathcal{J}_c(B)$. The strategy is then to replace the sequence $K_j$ by a sequence $\widetilde{K_j}$ also converging to $B$, for which \eqref{eq:goal_ineq} is still not verified and in which each $\widetilde{K_j}$ is meant to be much smoother than $K_j$. The convergence of the sequence $\widetilde{K_j}$ and the fact that it still contradicts \eqref{eq:goal_ineq} is somehow built-in the construction of these sets itself, as minimizers of an auxiliary minimization problem involving the $K_j$. The regularity of $\widetilde{K_j}$ then comes from the fact that it is a minimizer of an isoperimetric problem under convexity constraint: it was shown in \cite{LP23} that such minimizers are $C^{1,1}$, and we will provide a uniform version of this result (see Theorem \ref{thm:reg_LP}). This will enable us to apply the result of minimality in a smooth neighborhood proven in Section \ref{sect:IT} (see Theorem \ref{cor:min_smooth}) to finally get a contradiction.

In this section $B\subset \R^N$ still denotes the open unit ball centered at $0$.

\subsection{Regularity theory for the quasi-minimizer of the perimeter under convexity constraint}

This subsection is dedicated to the regularity theorem which is central to the \textit{selection principle} we perform in Section \ref{sect:select}. When working in the framework of quasi-minimizer of the perimeter without convexity constraint, a very useful type of results concerns the strengthening of convergence for a sequence of quasi-minimizers converging to the ball: if a sequence $(E_j)$ of (uniform) quasi-minimizers converges to the ball in a $L^1$ sense, then $E_j$ is $C^{1,1/2}$ for large $j$ and it converges (up to subsequence) to the ball in $C^{1,\alpha}$ for each $\alpha\in(0,1/2)$ (see for instance \cite[Theorem 4.2]{AFM13} for a rigorous statement). This is a compactness-type  result, which is a direct consequence of the $C^{1,1/2}$ regularity of quasi-minimizers and an estimate of their norm. We want here to prove an analogous result in our convexity constrained case. The regularity result we state below (Theorem \ref{thm:reg_LP}) importantly relies on the $C^{1,1}$ regularity results from \cite{LP23} (see \cite[Theorem 2.3]{LP23}). Nevertheless, in comparison with \cite[Theorem 2.3]{LP23} we have to follow the constants in the proof in order to show that a quasi-minimizer is locally parametrized in cartesian graphs by $C^{1,1}$ functions with norm only depending on the relevant constants. We then pass from this quantified local cartesian $C^{1,1}$ regularity to a global spherical estimate. Although this passage often comes as classical in the literature, it does not seem to be so well referenced and we believe that a careful examination of all the arguments can be of use (see also \cite[Appendix B]{Pet22} for similar arguments).

Let us first define the notion of quasi-minimizer of the perimeter under convexity constraint which was introduced in \cite[Definition 2.1]{LP23}. 

\begin{definition}[$(\Lambda,\eps)$-q.m.p.c.c.]\label{def:qmpcc}Let $N\geq2$. Let $\Lambda>0$, $\eps>0$. We say that $K\in \K^N$ is a $(\Lambda,\eps)$-quasi-minimizer of the perimeter under convexity constraint (or $(\Lambda,\eps)$-q.m.p.c.c. for short) if
\begin{equation}\label{eq:qmpc_chap2} \forall \widetilde{K}\in \K^N\textrm{ such that } \widetilde{K}\subset K \text{ and } |K\setminus\widetilde{K}|\leq\eps,\ \ P(K)\leq P(\widetilde{K})+\Lambda|K\setminus\widetilde{K}|.\end{equation}\end{definition}
For $x\in \R^N$, let $\nu_B(x):=x$. Any function $h:\partial B\rightarrow\R$ is extended to $\R^N$ as in Section \ref{sect:IT} (see the beginning of Section \ref{sect:IT}). For $r>0$ and $z\in \R^N$, the notation $B_r(z)$ denotes the ball of radius $r$ centered at $z$. The q.m.p.c.c. regularity result is the following.

\begin{theorem}[Regularity of q.m.p.c.c]\label{thm:reg_LP}Let $N\geq2$, $\Lambda>0$, $\eps>0$ and $0<m<M$. Let $K$ be a $(\Lambda,\eps)$-q.m.p.c.c. verifying $B_m(z)\subset K\subset B_M(z)$ for some $z\in\R^N$. Then there exists $h\in C^{1,1}(\partial B)$ such that (up to translation) $K$ can be written
\[ K=(\text{Id}+h\nu_B)(B)=\{tx(1+h(x)),\ t\in[0,1],\ x\in\partial B\}, \text{ with } \|h\|_{C^{1,1}(\partial B)}\leq  C\]
with $C=C(N,\Lambda,\eps,m,M)>0$ only depending on the indicated parameters.
\end{theorem}

Let us postpone the proof of this Theorem and show first how we deduce from this result a convergence type result of quasi-minimizers (as in the classical setting).

\begin{corollary}[Convergence of q.m.p.c.c]\label{thm:cor_LP}Let $N\geq2$, $\Lambda>0$, $\eps>0$. If $(K_j)$ is a sequence of $(\Lambda,\eps)$-q.m.p.c.c. such that $|K_j\Delta B|\rightarrow0$, then there exists a sequence $h_j\in C^{1,1}(\partial B)$ such that 
\[\forall j\in\N,\ K_j=(\text{Id}+h_j\nu_B)(B),\]
%\text{ with } \|h_j\|_{C^{1,1}(\partial B)}\leq  C\]
and for all $\alpha\in(0,1)$ it holds $h_j\rightarrow0$ in $C^{1,\alpha}$.
\end{corollary}

\begin{remark}\label{rk:optH}
Let us note here that this $C^{1,\alpha}$ convergence for all $\alpha\in (0,1)$ is essentially optimal in the sense that one cannot hope for more than $C^{1,1}$ regularity for a q.m.p.c.c. (see \cite[Proposition 3.18]{LP23} for a counter-example to higher Hölder regularity in two dimensions).
\end{remark}

\begin{proof}[Proof of Corollary \ref{thm:cor_LP}]
Since $|K_j\Delta B|\rightarrow0$, one also has that $K_j\rightarrow B$ in the Hausdorff sense thanks to Proposition \ref{prop:cvgconv_hausd}. As a consequence, there exists $z\in \R^N$ and $0<m<M$ such that
\begin{equation}\nonumber\forall j\in\N, \ B_{m}(z)\subset K_j\subset B_M(z).\end{equation}
The existence of the upper ball follows directly from the definition of the Hausdorff convergence; we refer for instance to \cite[Proposition 2.8, 2.]{LP23} for the existence of a lower ball. Therefore, thanks to Theorem \ref{thm:reg_LP} we deduce that there exists $h_j:\partial B\rightarrow\R$ such that
\[\forall j\in\N, \ K_j=(\text{Id}+h_j\nu_B)(B) \text{ and }\|h_j\|_{C^{1,1}(\partial B)}\leq C\] for some $C>0$ independent of $j$. From this bound on the $C^{1,1}$ norms we deduce for each $\alpha\in(0,1)$ the convergence (up to subsequence) of $h_j$ in $C^{1,\alpha}$ norm to some $h\in C^{1,1}(\partial B)$, using the Arzela-Ascoli Theorem. Since $K_j\rightarrow B$ in the Hausdorff sense we must have $h=0$, which ensures also that the whole sequence $(h_j)$ converges to $0$. This finishes the proof of Corollary \ref{thm:cor_LP}.
\end{proof}

We can now pass to the proof of Theorem \ref{thm:reg_LP}. 

\begin{proof}[Proof of Theorem \ref{thm:reg_LP}]\textbf{Step 1: cartesian estimates of $K$.} Let $\widehat{x_0}\in \partial K$ be fixed.
We claim that there exists 
\begin{itemize}\item a hyperplane $H\subset \R^N$ containing $\widehat{x_0}$ and a unit vector $\xi\in\R^N$ normal to $H$, \\
\item a $(N-1)$ dimensional ball $B_\beta^{N-1}:=B_\beta^{N-1}(\widehat{x_0})$ centered at $\widehat{x_0}$ and of radius $\beta=\beta(m,M)>0$ with $B_\beta^{N-1}\subset H$, \end{itemize} 
such that, denoting by $(x,t)$ a point in $H\times\R$ coordinates (according to the orthonormal frame $H\times\R\xi$)  and defining $u:B_\beta^{N-1}\rightarrow\R$ by the formula $u(x):=\min\{t\in \R, (x,t)\in K\}$ we have
\begin{align}\label{eq:K_Lip1}\{(x,u(x)), \ x\in  B_\beta^{N-1}\}&\subset \partial K \\ K\cap (B_\beta^{N-1}\times\R\xi)&\subset\{(x,t)\in B_\beta^{N-1} \times\R\xi,\ u(x)\leq t\}\label{eq:K_Lip2}\end{align}
and $u\in C^{1,1}\left(\overline{B_\beta^{N-1}}\right)$ with \begin{equation}\label{eq:C11_cart_est}\|u\|_{C^{1,1}(\overline{B_\beta^{N-1}})}\leq C, \text{ where }C=C(N,\Lambda,\eps,m, M)\end{equation} 
In this Step, for any $z\in H$ and $r>0$, we denote by $ B_r^{N-1}(z)\subset H$ the $(N-1)$-dimensional ball of radius $r$ centered at $z$. 

The existence of $H, \xi, \beta', u$ such that 
\[\beta'=\beta'(m,M)\]
\[u\in C^{0,1}(B_{\beta'}^{N-1}(\widehat{x_0})), \text{ with } \|\nabla u\|_{L^{\infty}(B_{\beta'}^{N-1}(\widehat{x_0}))}\leq C(m,M)\]
and such that \eqref{eq:K_Lip1} and \eqref{eq:K_Lip2} are satisfied for $B_{\beta'}^{N-1}(\widehat{x_0})$ comes from the convexity of $K$ (it is proven for instance in \cite[Proposition 4.3]{LP23}). We now prove \eqref{eq:C11_cart_est} for $\beta:=\beta'/2$. Let $y\in B_\beta^{N-1}=B_{\beta}^{N-1}(\widehat{x_0})$, $p\in\partial u(y)$ and set for each $r\in(0,\beta)$ 
\[M_r(y):=\sup_{B_r^{N-1}(y)}(u-\left(u(y)+\langle p,\cdot-y\rangle\right)\]
Using quasi-minimality of $K$ and the estimates of $\beta'$ and $\|\nabla u\|_{L^{\infty}(B_{\beta'}^{N-1}(\widehat{x_0}))}$ above, we can use \cite[Theorem 2.3]{LP23} (see last equation of the proof of Theorem 2.3) to deduce that there exists \begin{align*}C&=C\left(N,\Lambda, \eps,m,M\right)\\ r_0&=r_0\left(\eps,m,M\right)\end{align*}
such that 
\[\forall y\in B_{\beta}^{N-1},\ \forall r\in(0,r_0),\ M_r(y)\leq Cr^2\]
We now apply Lemma 3.2 in \cite{DF} which ensures that $u\in C^{1,1}\left(\overline{B_{\beta}^{N-1}}\right)$. More precisely, it is proven in \cite[Lemma 3.2]{DF} that there exists $\rho_0>0$ and $\eta>0$ only depending on $r_0$ and the Lipschitz character of $B_\beta^{N-1}$ (hence only on $\beta$) such that
\[\forall x\in B_{\beta}^{N-1},\ \forall y\in B_{\beta}^{N-1}\cap B_{\rho_0}^{N-1}(x), |\nabla u(x)-\nabla u(y)|\leq C'|x-y|\]
where $C'=6\eta^{-1}$. As we also have 
\begin{equation}\begin{split}\forall x\in B_{\beta}^{N-1},\ \forall y\in B_{\beta}^{N-1} \text{ with } |y-x|\geq \rho_0,\\ |\nabla u(x)-\nabla u(y)|\leq 2\|\nabla u\|_{L^\infty( B_{\beta}^{N-1})}\rho_0^{-1}|x-y|\end{split}\nonumber\end{equation}
 then gathering the two we get that 
\[\forall x,y\in B_{\beta}^{N-1}, |\nabla u(x)-\nabla u(y)|\leq \widetilde{C}|x-y|\]
by setting $\widetilde{C}:=\max\left\{C', 2\|\nabla u\|_{L^\infty(B_{\beta}^{N-1})}\rho_0^{-1}\right\}$. This together with the bound on $\|\nabla u\|_{L^\infty(B_{\beta'}^{N-1}(\widehat{x_0}))}$ above, and \[\|u\|_{L^\infty(B_\beta^{N-1})}\leq \text{diam}(K)\] yield the desired estimate on $\|u\|_{C^{1,1}(\overline{B_\beta^{N-1}})}$.

\textbf{Step 2: local spherical estimates of $\partial K$.} This step and the next one are similar to \cite[Appendix B]{Pet22}. Fix $\widehat{x_0}\in\partial K$. We apply Step 1 at $\widehat{x_0}$, and up to translating and rotating we assume without loss of generality that $z=0$ (so that $B_m(0)\subset K\subset B_M(0)$) and $\xi=e_N$ is the $N^{\text{th}}$ canonical direction. In this step we consider $z$ as the origin, so that the coordinates $(x,t)\in H\times\R$ will now take this into account. As a consequence, $\widehat{x_0}$ is now written $\ \widehat{x_0}=(0,t_0)$ for some $t_0<0$, and if $\Om:=B_\beta^{N-1}$ denotes the $(N-1)$-dimensional ball found in Step 1, we have
\[\forall x\in\Om,\ \widehat{u}(x):=(x,u(x)+t_0)\in \partial K\]
Since $H$ is orthogonal to $\xi=e_N$ and contains $\widehat{x_0}$ we have 
\[H=\widehat{x_0}+\{x_N=0\}\]
 We write more simply $B_m:=B_m(0)$. Let $\theta$ be the map
\begin{align}\nonumber \theta:\Om&\rightarrow \partial B_m\\
x&\mapsto m\frac{\widehat{u}(x)}{\left|\widehat{u}(x)\right|}\nonumber\end{align}
which associates to $x\in \Om$ the spherical coordinates corresponding to $\widehat{u}(x)$. Let now $\rho_{K}:\partial B_m\rightarrow (0,\infty)$ be the distance function of the convex set $K$, meaning that for any $\phi\in \partial B_m$, $\rho_{K}(\phi)$ is the unique $t>0$ such that $t\phi\in \partial K$. Then for each $x\in \Om$, it holds $\theta(x)\rho_{K}(\theta(x))=\widehat{u}(x)$, so that 
\begin{equation}\rho_{K}(\theta(x))=\frac{|\widehat{u}(x)|}{m}\label{eq:rho_circ_theta}\end{equation}
As a consequence, $\rho_{K}\circ \theta\in C^{1,1}(\Om)$ with norm only depending on the $C^{1,1}$ norm of $u$. The rest of Step 2 consists in showing that $\rho_{K}$ itself is $C^{1,1}$ in a neighborhood of $\theta(x_0)$ and to estimate its norm, by using a suitable version of the inverse function Theorem. 

We let $\theta'$ be the map
\begin{align}\nonumber \theta':\Om\times \R&\rightarrow \R^N\\
(x,t)&\mapsto \theta(x)(1+t)\nonumber\end{align}
Note that $\theta'\in C^{1,1}(\Om\times\R)$ and $\theta'_{|\Om\times\{0\}}=\theta$. Then it holds

\begin{equation}\label{eq:Dtheta'}D\theta'(\widehat{x_0})=m
\begin{pNiceArray}{cw{c}{1cm}c|c}
\Block{3-3}{|t_0|^{-1}I_{N-1}} & & & 0 \\
& & & \Vdots \\
& & & 0 \\
\hline
0 & \Cdots& 0 & 1
\end{pNiceArray}\end{equation}

Using a quantitative version of the inverse function Theorem (see Theorem \ref{thm:quantIFT}) we deduce the existence of a radius $r_0=r_0\left(\|\theta'\|_{C^{1,1}(V_0)}, |D\theta'(x_0)^{-1}|,\beta\right)$, $V_0$ and $W_0$ respectively open neighborhoods of $\widehat{x_0}$ and $\theta'(\widehat{x_0})$ such that $\theta'$ is a $C^{1,1}$-diffeomorphism from $V_0$ onto $W_0$, with 
\begin{equation}\label{eq:V0_W0}B_{r_0}(\widehat{x_0})\subset V_0, B_{r_0}(\theta'(\widehat{x_0}))\subset W_0\end{equation} and 
\begin{equation}\label{eq:theta'C11} \|(\theta')^{-1}\|_{C^{1,1}(W_0)}\leq C\left(\|\theta'\|_{C^{1,1}(V_0)},|D\theta'(x_0)^{-1}|\right)\end{equation}
Now, by definition of $\theta'$ and $\theta$, and since $|\widehat{u}(x)|\geq m$, it holds $\|\theta'\|_{C^{1,1}(V_0)}\leq C\left(\|u\|_{C^{1,1}(\Om)},m^{-1}\right)$. On the other hand, by \eqref{eq:Dtheta'} we have $|(D\theta')(x_0)^{-1}|\leq C(|t_0|,m^{-1})$. As $K\subset B_M(0)$ we have $|t_0|\leq M$ so that \eqref{eq:V0_W0} and \eqref{eq:theta'C11} become respectively
\begin{equation}\label{eq:spherical_radius}B_{r_0}(\widehat{x_0})\subset V_0, B_{r_0}(\theta'(\widehat{x_0}))\subset W_0, \text{ with } r_0=r_0\left(\beta,m,M,\|u\|_{C^{1,1}(\Om)}\right)\end{equation}
and 
\[\|(\theta)'^{-1}\|_{C^{1,1}(W_0)}\leq C\left(m,M,\|u\|_{C^{1,1}(\Om)}\right)\]
Recalling that $\theta'_{|\Om\times\{0\}}=\theta$ there exists a constant $C>0$ such that $\|\theta^{-1}\|_{C^{1,1}(W_0\cap \partial B_m)}\leq\|(\theta')^{-1}\|_{C^{1,1}(W_0)}$. Hence, by \eqref{eq:rho_circ_theta} and the latter estimate of $\theta'$ we get
\begin{align}\nonumber\|\rho_{K}\|_{C^{1,1}(W_0\cap \partial B_m)}&\leq \|\rho_{K}\circ\theta\|_{C^{1,1}(V_0\cap H)}\|\theta^{-1}\|_{C^{1,1}(W_0\cap \partial B_m)}\\ &\leq C, \text{ with } C=C\left(m,M,\|u\|_{C^{1,1}(\Om)}\right)\label{eq:spherical_rho}\end{align}

Using the estimates of $\beta$ and $\|u\|_{C^{1,1}(\Om)}$ found in Step 1 we finally get that the radius $r_0$ and constant $C$ respectively from \eqref{eq:spherical_radius} and \eqref{eq:spherical_rho} only depend on $N,\Lambda,\eps,m,M$.

\textbf{Step 3: global estimate of $\rho_{K}$:} relying on the local estimate \eqref{eq:spherical_rho} of $\rho_{K}$ proven in Step 2 we now estimate $\|\rho_{K}\|_{C^{1,1}(\partial B_m)}$.

According to Step 2, for any $\phi\in\partial B_m$ there exists $W_\phi\subset \R^N$ an open neighborhood of $\phi$ such that 
\[B_{r_0}(\phi)\subset W_\phi\]
\[\|\rho_{K}\|_{C^{1,1}(W_\phi\cap\partial B_m)}\leq C\]
where $r_0= r_0(N,\Lambda,\eps,m,M)$ and $C=C(N,\Lambda,\eps,m,M)$. Using a standard compactness argument over $\partial B_m$ we find $\eta>0$ only depending on $r_0$ such that for any $\phi,\psi\in\partial B_m$ with $|\phi-\psi|\leq\eta$
\[\frac{|\rho_{K}(\phi)-\rho_{K}(\psi)|}{|\phi-\psi|}\leq C,\ 
\frac{|\nabla \rho_{K}(\phi)-\nabla\rho_{K}(\psi)|}{|\phi-\psi|}\leq C\] 
Combining these with a global bound $\|\rho_{K}\|_{W^{1,\infty}(\partial B_m)}\leq C$ we deduce that the same estimates hold for $|\phi-\psi|\geq\eta$ so that we finally have
\begin{equation}\label{eq:spherical_C11} \|\rho_{K}\|_{C^{1,1}(\partial B_m)}\leq C\end{equation} for some $C=C(N,\Lambda,\eps,m,M)$.

\textbf{Conclusion.} Let $h(\phi):=\rho_{K}(m\phi)-1$ for each $\phi\in \partial B$. Then we have that $K=(\text{Id}+hn_B)(B)$ and \eqref{eq:spherical_C11} ensures that 
\[||h\|_{C^{1,1}(\partial B)}\leq  C\]
for some $C=C(N,\Lambda,\eps,m,M)>0$. This concludes the proof.
\end{proof}

\subsection{Proof of Theorem \ref{th:mainthm}}
In this subsection we perform the \textit{selection principle}, relying both on the convergence result for quasi-minimizers (Corollary \ref{thm:cor_LP}) and the strict minimality of the ball in a $C^{1,\alpha}$ neighborhood shown in Theorem \ref{cor:min_smooth}. 

We will use the fact that $\lambda_1$ satisfies some kind of Lipschitz hypothesis for the Volume distance. This is stated in next Proposition.

\begin{prop}\label{eq:lambdaR_reform}
Let $N\geq2$. Let $D\in\K^N$ and $0<V_0<|D|$. There exists $C=C(V_0,D)$ such that for all convex bodies $K_1,K_2\subset D$ with $|K_1|,|K_2|\geq V_0$ it holds
\begin{equation}\label{eq:strongLip}|\lambda_1(K_1)-\lambda_1(K_2)|\leq C|K_1\Delta K_2|.\end{equation}
\end{prop}

Note that this Lipschitz-type property is an improvement of the result obtained for $\lambda_1$ in \cite[Theorem 3.2 and Remark 3.3]{LP23}. We refer also to \cite[Theorem 4.23] {BL08}  for a similar result proven for a different class of sets.

\begin{proof}[Proof of Proposition \ref{eq:lambdaR_reform}]
It was proven in \cite[Theorem 3.2 and Remark 3.3]{LP23} that for any $D'\subset D\in \K^N$ there exists $C=C(D',D)$ such that for all $D'\subset K_1\subset D$, $D'\subset K_2\subset D$ it holds
\begin{equation}|\lambda_1(K_1)-\lambda_1(K_2)|\leq C|K_1\Delta K_2|\label{eq:lambdaLip}\end{equation}

Let then $K_1,K_2\subset D$ with $|K_1|,|K_2|\geq V_0$. Assume first that $|K_1\Delta K_2|\geq V_0/2$. Thanks to (ii) in Proposition \ref{lem:compact} we can find $\eps=\eps(V_0, D)>0$ independent of $K_1, K_2$ such that the inradii satisfy $r_{K_1},r_{K_2}\geq\eps$. As a consequence, for $i=1,2$ there exists $x_i\in \R^N$ such that $B_\eps(x_i)\subset K_i$. By monotonicity of $\lambda_1$ we deduce
\begin{equation}\label{eq:lambda_Lip_pf}|\lambda_1(K_1)-\lambda_1(K_2)|\leq 2\lambda_1(B_\eps(0))\leq 4V_0^{-1}\lambda_1(B_\eps(0))|K_1\Delta K_2|\end{equation}
Assume otherwise that $|K_1\Delta K_2|\leq V_0/2$. Then we have 
\[|K_1\cap K_2|=|K_1|-|K_1\setminus K_2|\geq V_0/2\]
Using again (ii) from Proposition \ref{lem:compact} we can therefore find $\eps'=\eps'(V_0,D)>0$  such that the inradius of the convex body $K_1\cap K_2\subset D$ satisfies $r_{K_1\cap K_2}\geq\eps'$. Hence, there exists $x\in D$ such that $B_{\eps'}(x)\subset K_1\cap K_2\subset K_i$, $i=1,2$. Letting $R>0$ be such that $B_R(x)\supset D$, we have $B_{\eps'}(0)\subset K_i-x\subset B_{R}(0)$ for $i=1,2$ and we therefore use property \eqref{eq:lambdaLip} to deduce 
\begin{align*}|\lambda_1(K_1)-\lambda_1(K_2)|&=|\lambda_1(K_1-x)-\lambda_1(K_2-x)|\\&\leq C|(K_1-x)\Delta(K_2-x)|\\&=C|K_1\Delta K_2|\end{align*}
This estimate together with \eqref{eq:lambda_Lip_pf} gives the conclusion.
\end{proof}

We can now pass to the proof of Theorem \ref{th:mainthm}. 

For convex bodies $K_1$ and $K_2$ the notation $d_H(K_1,K_2)$ refers to the usual Hausdorff distance between $K_1$ and $K_2$ (see Section \ref{sect:Apphausd} in the Appendix for some facts about the Hausdorff distance for convex sets).

\begin{proof}[Proof of Theorem \ref{th:mainthm}]
\textbf{Step 1: penalization.} As a preparation of the selection procedure from Step 2 below, we prove in this step that if we let $D\in\K^N$, $0<V_0<|D|$, $a\in\R,\mu\geq0$ and set $R=-\lambda_1+\mu||K\Delta B|-a|$, then a minimizer $K^*$ of
\begin{equation}\inf\left\{P(K)+R(K), K\in\K^N,\ K\subset D,\ |K|=V_0\right\}\label{eq:minVol}\end{equation} verifying $\delta:=d(K^*,\partial D)>0$ is a $(\Lambda,\eps)$-q.m.p.c.c. (see Definition \ref{def:qmpcc}) for some $\Lambda=\Lambda\left(V_0,\delta,\mu,D\right)$ and $\eps=\eps\left(V_0,\delta,\mu,D\right)$. This result is an adaptation of \cite[Lemma 2.11]{LP23}.

Let $K^*$ be such a minimizer, and set $0<v_0 \leq V_0$ and $\delta:=d(K^*,\partial D)>0$. We introduce the class 
\[\mathcal{A}_{v_0,\delta}:=\{K\in\K^N,\ K\subset D,\ |K|\geq v_0,\ d(K,\partial D)\geq\delta\}\]
which is compact for $d_H$, by Proposition \ref{lem:compact} and continuity of the volume for the Hausdorff distance. Note that the set $K^*$ belongs to this class for $v_0=V_0$. Set 
\[\forall \eps>0, \ \mathcal{O}_{\eps}(K^*):=\{K\in \K^N,\ K\subset K^*, |K^*\setminus K|\leq \eps\}\]
and let for any convex body $K\subset D$ and $t\in[0,1]$ 
\[K_t:=(1-t)K+tD\]

We first claim that there exists constants $(\eps_0,c,t_0)\in(0,\infty)^3$ depending only on $V_0, \delta$ and $D$ (hence independent of the minimizer $K^*$) such that
\begin{equation}\label{eq:below_vol}\forall K\in\mathcal{O}_{\eps_0}(K^*),\ \forall t\in[0,t_0],\ |K_t|-|K|\geq ct.\end{equation}
Setting $f_K(t):=|K_t|$, it is shown in \cite[Lemma 2.11]{LP23} that $f_K$ is a polynomial in $t$ of degree $N$ with coefficients continuous in $K$ for $d_H$, and that $f_K'(0)$ is positive whenever $K\subsetneq D$. Set $\eps_0:=V_0/2$. By compactness of $\mathcal{A}_{\eps_0,\delta}$ inside $\{K\in \K^N,K\subsetneq D\}$ and continuity of $K\mapsto f_K'(0)$ for $d_H$, then one can find $c=c(V_0,\delta,D)>0$ such that for any $K\in\mathcal{A}_{\eps_0,\delta}$ it holds $f_K'(0)\geq c$. Any $K\in \mathcal{O}_{\eps_0}(K^*)$ verifies $|K|\geq V_0/2$, so that for such $K$ it holds $K\in\mathcal{A}_{\eps_0,\delta}$ hence $f_K'(0)\geq c$. Since the coefficients of the polynomial $f_K(t)$ are continuous in $K$ for $d_H$, they are uniformly bounded for $K\in \mathcal{A}_{\eps_0,\delta}$. This together with the lower bound on  $f_K'(0)$ yields the above estimate.

The existence of $C=C(D)$ such that
\[\forall t\in[0,1], \forall K\in\K^N \text{ with } K\subset D,\ P(K_t)-P(K)\leq Ct\]
is proven in \cite[equation (54)]{LP23}. As a consequence, this together with \eqref{eq:below_vol} gives
\begin{equation}\label{eq:perim_above}\forall K\in\mathcal{O}_{\eps_0}(K^*),\ \forall t\in[0,t_0], \ \ P(K_t)-P(K)\leq C'\left(|K_t|-|K|\right)\end{equation}
with $C':=C/c$.

Since any $K\in\mathcal{O}_{\eps_0}(K^*)$ verifies $|K|\geq V_0/2$ we can apply Proposition \ref{eq:lambdaR_reform} to get the existence of $C=C(V_0,D)$ such that for all $t\in[0,1]$ and $ K\in\mathcal{O}_{\eps_0}(K^*)$
\begin{align}\nonumber R(K_t)-R(K)&=\lambda_1(K)-\lambda_1(K_t)+\mu\left(||K_t\Delta B|-a|-||K\Delta B|-a|\right),
\\\nonumber & \leq C|K_t\setminus K|+\mu||K_t\Delta B|-|K\Delta B||,\\
&\leq (C+\mu)|K_t\setminus K|.\label{eq:R_above}
\end{align}

Let us now show that for $\eps:=\min\{\eps_0,ct_0\}$ there exists $\Lambda=\Lambda(V_0,\delta,D,\mu)$ such that a minimizer $K^*$ of \eqref{eq:minVol} is a minimizer of 
\begin{equation}\label{eq:minK*2}\inf\{P+R+\Lambda||K|-V_0|,\ K\in\mathcal{O}_{\eps}(K^*)\}.\end{equation}
Since $|K_{t_0}|-|K^*|=|K_{t_0}|-|K|+|K|-|K^*|\geq ct_0-\eps\geq0$ then by continuity of $t\mapsto|K_t|$ there exists $t\in[0,t_0]$ such that $|K_t|=|K^*|=V_0$. Hence by minimality and using \eqref{eq:perim_above} and \eqref{eq:R_above} we get
\[P(K^*)+R(K^*)\leq P(K_t)+R(K_t)\leq P(K)+R(K)+\Lambda||K|-V_0||\]
for some $\Lambda=\Lambda(V_0,\delta,D,\mu)$, which ensures the minimality of $K^*$ for \eqref{eq:minK*2}.

Therefore, if $K\in \K^N$ with $K\subset K^*$ and $|K^*\setminus K|\leq \eps$ then the computation leading to \eqref{eq:R_above} with $(K^*,K)$ in place of $(K_t,K)$ gives
\[P(K^*)-P(K)\leq \left(C+\mu+\Lambda\right)|K^*\setminus K|\]
so that $K^*$ is a $(\Lambda',\eps)$-q.m.p.c.c where $\Lambda':=\Lambda+C+\mu$ and $\eps$ only depend on $V_0, \delta, \mu$ and $D$. This finishes the proof of the first step.

\textbf{Step 2: selection procedure.} Although the \textit{selection principle} was first introduced in \cite{CL12}, the way we display the argument in this step is more inspired of \cite{AFM13}. 
Let $0<c<c^*$. Recall the notation $\mathcal{J}_c:=P-c\lambda_1$. Let us assume in order to obtain a contradiction that the conclusion of Theorem \ref{th:mainthm} is false. Then there exists a sequence of convex bodies $(K_j)_{j\in\N}$ such that 
\[\begin{cases} \forall j\in\N,\ \mathcal{J}_c(K_j)<\mathcal{J}_c(B),\\ |K_j\Delta B|\rightarrow0.\end{cases}\]

Thanks to Proposition \ref{prop:cvgconv_hausd}, $K_j\rightarrow B$ in the Hausdorff sense, so that there exists $D\in\K^N$ such that $K_j\subset D$ for every $j$. We can assume without loss of generality that $B\Subset \text{Int}(D)$. Thanks to Proposition \ref{eq:lambdaR_reform} the functional $\lambda_1$ is lower-semi-continuous for the volume distance, and we can apply the existence result \cite[Theorem 3.4 (i)]{LP23} to get that for any fixed $\mu>0$ there exists for each $j$ a solution to the problem 
\begin{equation}\label{eq:inf_select}\inf\left\{\mathcal{J}_c(K)+\mu\left||K\Delta B|-|K_j\Delta B|\right|, \ |K|=|B|, \ K\subset D, \ K\in \K^N\right\}.\end{equation}
Assuming that the value of $\mu$ has been fixed (we choose $\mu$ later on), we let $\widetilde{K_j}$ be a solution.

Thanks to (i) from Proposition \ref{lem:compact} there exists a convex body $\widetilde{K}\subset D$ with $|\widetilde{K}|=|B|$ such that (up to subsequence) $\widetilde{K_j}\rightarrow \widetilde{K}$ in the Hausdorff sense and in measure. 
%We can therefore find $D'\in\K^N$ such that 
%\[\forall j\in\N,\ D'\subset \widetilde{K_j}\subset D\] 
We have $\lambda_1(\widetilde{K_j})\rightarrow \lambda_1(\widetilde{K})$ using Proposition \ref{eq:lambdaR_reform}, and $P(\widetilde{K_j})\rightarrow P(\widetilde{K})$ by \cite[Proposition 2.4.3 (ii)]{BB05}. Now, from optimality and recalling $\mathcal{J}_c(K_j)<\mathcal{J}_c(B)$ we can write 
\begin{equation}\label{eq:optC_j}\mathcal{J}_c(\widetilde{K_j})+\mu||\widetilde{K_j}\Delta B|-|K_j\Delta B||\leq \mathcal{J}_c(K_j)< \mathcal{J}_c(B)\end{equation} so that we get at the limit 
\begin{equation}\label{eq:limitC}\mathcal{J}_c(\widetilde{K})+\mu|\widetilde{K}\Delta B|\leq \mathcal{J}_c(B).\end{equation}
Thanks to Proposition \ref{eq:lambdaR_reform} and using the isoperimetric inequality we get
\begin{align*}\mathcal{J}_c(\widetilde{K})-\mathcal{J}_c(B)&=\left(P(\widetilde{K})-P(B)\right)+c\left(\lambda_1(B)-\lambda_1(\widetilde{K})\right)\\
& \geq c(\lambda_1(B)-\lambda_1(\widetilde{K}))\\
&\geq -C|\widetilde{K}\Delta B|
\end{align*} 
where the constant $C$ only depends on $c$, $D$ and $N$. Injecting this into \eqref{eq:limitC} provides
\[-C|\widetilde{K}\Delta B|+\mu|\widetilde{K}\Delta B|\leq 0\] 
so that we get $\widetilde{K}=B$ if $\mu$ is chosen bigger than $C$ in \eqref{eq:inf_select}.

Therefore $\widetilde{K_j}\rightarrow B$ in measure and Hausdorff distance. Since we have chosen $B\Subset \text{Int}(D)$ we find a convex body $\widetilde{D}\subset D$ with $d(\widetilde{D},\partial D)>0$ and such that for $j$ sufficiently large
%D'\subset 
\[\widetilde{K_j}\subset \widetilde{D}.\]
By construction of the $\widetilde{K_j}$, we deduce from Step 1 that each $\widetilde{K_j}$ is a $(\Lambda,\eps)$-q.m.p.c.c.  with parameters independent of $j$. We can therefore apply Corollary \ref{thm:cor_LP} to get the existence of $h_j\in C^{1,1}(\partial B)$ such that up to subsequence
\[\widetilde{K_j}=(\text{Id}+h_j\nu_B)(B) \text{ and } \|h_j\|_{C^{1,\alpha}(\partial B)}\rightarrow0\] 
for $\alpha$ chosen to satisfy Theorem \ref{cor:min_smooth}. We can therefore apply Theorem \ref{cor:min_smooth} to deduce that for sufficiently large $j$,
\[\mathcal{J}_c(\widetilde{K_j})\geq \mathcal{J}_c(B)\]
But this enters in contradiction with \eqref{eq:optC_j}, thus concluding the proof of the Theorem.
\end{proof}

\section{Appendix}\label{sect:app}

\subsection{Quantified Inverse Function Theorem}

\begin{theorem}[Quantified IFT]\label{thm:quantIFT}
Let $N\in\N^*$. Let $V:=B_r(\widehat{x_0})\subset\R^N$ for some $\widehat{x_0}\in V$ and $r>0$. Let $f\in C^{1,1}(V,\R^N)$ with $Df(\widehat{x_0})$ invertible. Then there exists $V_0\subset V, W_0\subset \R^N$ and $\rho=\rho\left(\|f\|_{C^{1,1}},|Df(x_0)^{-1}|,r\right)$ only depending on the indicated parameters such that $f$ is a $C^{1,1}$ diffeomorphism from  $V_0$ onto $W_0$ and 
\[B_{\rho}(\widehat{x_0})\subset V_0,\ B_{\rho}(f(\widehat{x_0}))\subset W_0,\]
\[ \|f^{-1}\|_{C^{1,1}(W_0,V_0)}\leq C\left(\|f\|_{C^{1,1}(V,\R^N)},|Df(x_0)^{-1}|,r\right).\]
\end{theorem}

\begin{proof}
\textbf{Step 1:} By following the usual proof of the Inverse function Theorem we first show that $f$ is a $C^{1}$ diffeomorphism from $V_0'\subset V$ to $W_0':=B_{\delta}(f(\widehat{x_0}))$ with 
\[\delta=\delta\left(\|Df\|_{C^{0,1}(V)},|Df(\widehat{x_0})^{-1}|\right).\]
Since the proof is classical we only emphasize on the details needed to quantify the size of the neighborhood $W_0'$.

We will keep the notation $V$ for the set $V:=B_r(\widehat{x_0})$. Let for any $y\in\R^N$ the function $\phi_y:V\rightarrow\R^N$ be defined by
\[\forall x\in V,\ \phi_y(x):=x-\left(Df(\widehat{x_0})\right)^{-1}(f(x)-y).\]
Then $\phi_y$ is $C^{1,1}$, for $x\in V$ its differential $D\phi_y(x)=\text{Id}-\left(Df(\widehat{x_0})\right)^{-1}Df(x)$ is independent of $y$, and we have for $x\in V$,
\begin{align*}|D\phi_y(x)|&\leq |Df(\widehat{x_0})^{-1}||Df(x)-Df(\widehat{x_0})|,\\
&\leq |Df(\widehat{x_0})^{-1}|\|Df\|_{C^{0,1}(V)}|x-\widehat{x_0}|.\end{align*}
As a consequence there exists 
\[r'=r'\left(|Df(\widehat{x_0})^{-1}|,\|Df\|_{C^{0,1}(V)},r\right)>0\]
such that for all $x\in B_{r'}(\widehat{x_0})$ it holds $|D\phi_y(x)|\leq 1/2$. We thus get that for all $y\in\R^N$, $\phi_y$ is $1/2$-Lipschitz over $B_{r'}(\widehat{x_0})$. Now, we have 
\[|Df(\widehat{x_0})^{-1}(f(\widehat{x_0})-y)|\leq |Df(\widehat{x_0})^{-1}||f(\widehat{x_0})-y|\leq r'/2\]
when $y\in B_{\delta}(f(\widehat{x_0}))$ with $\delta:=\frac{r'}{2}|Df(\widehat{x_0})^{-1}|$. Hence, 
\begin{align*}\forall y\in B_\delta(f(\widehat{x_0})),\ \forall x\in \overline{B_{r'}}(\widehat{x_0}),\ |\phi_y(x)-\widehat{x_0}|&\leq |\phi_y(x)-\phi_y(\widehat{x_0})|+|\phi_y(\widehat{x_0})-\widehat{x_0}|,\\
&\leq |x-\widehat{x_0}|/2+r'/2,\\
&\leq r'.
\end{align*}
As a consequence, for all $y\in B_\delta(f(\widehat{x_0}))$, $\phi_y$ sends $\overline{B_{r'}}(\widehat{x_0})$ into itself and is $1/2$-Lipschitz. Therefore, for such $y$ the mapping $\phi_y$ has a unique fixed point in $\overline{B_{r'}}(\widehat{x_0})$, meaning that $f(x)=y$ for a unique $x\in \overline{B_{r'}}(\widehat{x_0})$. This gives the existence of $f^{-1}:W_0'\rightarrow V_0'$ with $W_0':=B_\delta(f(\widehat{x_0}))$ and $V_0':=f^{-1}(W_0')$, and one classically shows that $f$ is $C^1$ diffeomorphism from $V_0'$ to $W_0'$ with furthermore $Df^{-1}=(Df(f^{-1}))^{-1}$.

\textbf{Step 2:} Using an explicit expansion of the inverse mapping about $Df(\widehat{x_0})$ we have that 
\begin{equation}\label{eq:control_df-1}|Df(x)^{-1}|\leq 2|Df(\widehat{x_0})^{-1}|\end{equation}whenever $|Df(x)-Df(\widehat{x_0})|\leq 1/(2|Df(\widehat{x_0})^{-1}|)$. But this latter condition is fulfilled if $x\in V_0:=B_{\widetilde{r}}(\widehat{x_0})$ for some $\widetilde{r}$ depending on $r'$, $|Df(\widehat{x_0})^{-1}|$ and $\|f\|_{C^{1,1}(V)}$ only. Moreover, $f$ is a $C^1$ diffeomorphism from $V_0$ to $W_0:=f(V_0)$, and one can find $\widetilde{\delta}=\widetilde{\delta}\left(\widetilde{r},\|f^{-1}\|_{C^{0,1}(W_0)}\right)$ such that $W_0\supset B_{\widetilde{\delta}}(f(\widehat{x_0}))$. Now, thanks to \eqref{eq:control_df-1} and $Df^{-1}=(Df(f^{-1}))^{-1}$ we have
\begin{equation}\label{eq:Lipf-1}\|f^{-1}\|_{L^\infty(W_0)}\leq r, \ \|Df^{-1}\|_{L^\infty(W_0)}\leq 2|Df(\widehat{x_0})^{-1}|.\end{equation}
Letting $y,y'\in W_0$, since $f^{-1}(y), f^{-1}(y')\in V_0$ we can use \eqref{eq:control_df-1} and \eqref{eq:Lipf-1} to get
\begin{align*}|Df^{-1}(y)-Df^{-1}(y')|&\leq \|(Df)^{-1}\|_{L^\infty(V_0)}^2|Df(f^{-1}(y))-Df(f^{-1}(y'))|\\
&\leq C|y-y'|
\end{align*}
for some $C=C\left(|Df(\widehat{x_0})^{-1}|,\|f\|_{C^{1,1}(V)},r'\right)$. Combined with \eqref{eq:Lipf-1}, and recalling the estimate of $r'$ from Step 1 we thus get 
\[\|f^{-1}\|_{C^{1,1}(W_0)}\leq C\left(|Df(\widehat{x_0})^{-1}|,\|f\|_{C^{1,1}(V)},r\right).\]
Setting $\rho:=\min\{\widetilde{\delta},\widetilde{r}\}$ we have proved the Theorem.
\end{proof}

\subsection{Fractional Sobolev spaces} In this paragraph we state some standard facts about Sobolev spaces on the boundary of a $C^{1,1}$ open set $\Om\subset\R^N$. 

Let us first quickly recall the definition of the Sobolev-Slobodeckij spaces $W^{s,p}(\partial \Om)$ for $s\in[-2,2]$. For $k\in\{0,1,2\}$ and $p\in[1,\infty)$ the set $W^{k,p}(\partial\Om)$ is defined as the usual space of functions with $k$ first derivatives lying in $L^p(\partial\Om)$. For any $s\in(0,2)\setminus\{1\}$ and $p\in[1,\infty)$ we let the vector space
\[W^{s,p}(\partial \Om):=\left\{a\in W^{[s],p}(\partial \Om)\ :\ \forall |\alpha|=[s],\ \frac{|D^\alpha a(x)-D^\alpha a(y)|}{|x-y|^{\frac{N-1}{p}+\{s\}}}\in L^p(\partial\Om\times\partial\Om)\right\}\]
with $|\alpha|:=\sum_{i=1}^N\alpha_i$ for a multi-index $\alpha=(\alpha_1,\ldots,\alpha_N)\in\N^N$, and where $\{s\}$ and $[s]$ stand respectively for the fractional and integer parts of $s$. We then define the $W^{s,p}$ norm accordingly: for $a\in W^{s,p}(\partial\Om)$, we set
\[\|a\|_{W^{s,p}(\partial\Om)}:=\left(\|a\|_{W^{[s],p}(\partial\Om)}^p+\sum_{|\alpha|=[s]}|D^\alpha a|_{W^{\{s\},p}(\partial\Om)}^p\right)^{\frac{1}{p}},\]
with $|\cdot|_{W^{\{s\},p}(\partial\Om)}$ the $W^{\{s\},p}(\partial\Om)$ semi-norm, \textit{i.e.}
\[|D^\alpha a|_{W^{\{s\},p}(\partial\Om)}^p:=\iint_{\partial\Om\times\partial\Om}\frac{|D^\alpha a(x)-D^\alpha a(y)|^p}{|x-y|^{N-1+\{s\}p}}d\H^{N-1}_xd\H^{N-1}_y\]
where $\H^{N-1}$ denotes the $(N-1)$-dimensional Hausdorff measure. If $p'$ denotes the conjugate Hölder exponent  to $p$ (\textit{i.e.} $\frac{1}{p'}+\frac{1}{p}=1$), the space $W^{-s,p'}(\partial\Om)$ is defined as the dual space of $W^{s,p}(\partial\Om)$:
\[W^{-s,p'}(\partial\Om):=W^{s,p}(\partial\Om).\]
We can now state some facts about the spaces $W^{s,p}(\partial\Om)$.

\begin{prop}\footnote{The author is very grateful to the anonymous referee for pointing out the facts (iii) and (iv) and for providing short proofs, thus enabling to improve the results from Section \ref{sect:IT}.}\label{prop:factFspq}Let $N\geq2$ and $\Om\subset\R^N$ be a $C^{1,1}$ bounded open set. 
\begin{enumerate}[label=(\roman{*})]
\item \label{prop:Sobinj}Let $s\in(0,1)$. Then for all $\alpha\in(s,1)$ and $p\in[1,\infty)$ it holds
\begin{equation*}\label{eq:Hold_incl_Sob}\ C^{0,\alpha}(\partial \Om)\subset  W^{s,p}(\partial \Om),\ C^{1,\alpha}(\partial \Om)\subset  W^{1+s,p}(\partial \Om),\end{equation*}
with continuous injections.
\item \label{prop:difflaw} Let $s\in(0,1)$ and $p\in(1,\infty)$. Then there exists $C>0$ such that for any $a\in W^{1-s,p}(\partial \Om)$ it holds
\begin{equation*}\|\nabla_\tau a\|_{W^{-s,p}(\partial \Om)}\leq C\|a\|_{W^{1-s,p}(\partial \Om)}.\end{equation*}
\item \label{prop:prodCalpha}For each $\alpha\in(\frac{1}{2},1)$ the product law $C^{0,\alpha}(\partial\Om)\cdot H^{1/2}(\partial \Om)\subset H^{1/2}(\partial\Om)$ holds, meaning that for each $\alpha\in(\frac{1}{2},1)$ and any $(a,b)\in C^{0,\alpha}(\partial\Om)\times H^{1/2}(\partial \Om)$ it holds
%$s+1/p_1\in(0,1)$ and 
\[\|ab\|_{H^{1/2}(\partial\Om)}\leq C\|a\|_{C^{0,\alpha}(\partial\Om)}\|b\|_{H^{1/2}(\partial\Om)},\]
for some constant $C>0$ independent of $a$ and $b$.
\item\label{prop:prodH1/2} There exists $p\in(1,2)$ such that the product law $H^{1/2}(\partial\Om)\cdot H^{1/2}(\partial\Om)\subset W^{s,p}(\partial\Om)$ holds for any $s\in(0,\frac{1}{2})$, meaning that for each $s\in(0,\frac{1}{2})$ and any $(a,b)\in H^{1/2}(\partial\Om)\times H^{1/2}(\partial\Om)$ it holds
%$s+1/p_1\in(0,1)$ and 
\[\|ab\|_{W^{s,p}(\partial\Om)}\leq C\|a\|_{H^{1/2}(\partial\Om)}\|b\|_{H^{1/2}(\partial\Om)},\]
for some constant $C>0$ independent of $a$ and $b$.
\end{enumerate}
\end{prop} 

\begin{proof}
\begin{enumerate}[label=(\roman{*})]
%\item For the same statement over $\R^{N-1}$ see for instance \cite[Proposition p.20]{RS}.
\item If $a\in C^{0,\alpha}(\partial\Om)$, 
%then if $A\subset \R^{N-1}$ is bounded 
\[\forall x,y\in \partial\Om,\ \frac{|a(x)-a(y)|^p}{|x-y|^{N-1+sp}}\leq |a|^p_{C^{0,\alpha}(\partial\Om)}|x-y|^{-(N-1-p(\alpha-s))}\] where $|\cdot|_{C^{0,\alpha}(\partial\Om)}$ is the $C^{0,\alpha}$ semi-norm on $\partial\Om$. The integrability of \[(x,y)\in \partial\Om\times \partial\Om\mapsto |x-y|^{-(N-1-p(\alpha-s))}\] thus ensures the continuous injection $C^{0,\alpha}(\partial\Om)\subset W^{s,p}(\partial\Om)$. The injection $C^{1,\alpha}(\partial \Om)\subset  W^{1+s,p}(\partial \Om)$ then follows from applying this to $\nabla a$ for some $a\in C^{1,\alpha}(\partial \Om)$.
\item This claim is deduced from the same statement over $W^{s,p}(\R^{N-1})$ spaces by working in local charts (for the  $\R^{N-1}$ case see for instance \cite[Remark 8.10.14]{Bh}). 
\item Let $a\in C^{0,\alpha}(\partial\Om)$ and $b\in H^{1/2}(\partial \Om)$. Then the following chain of inequalities holds:
\begin{align*}|ab|_{H^{1/2}(\partial\Om)}&=\left(\iint_{\partial\Om\times\partial\Om}\frac{|a(x)b(x)-a(y)b(y)|^2}{|x-y|^N}d\H^{N-1}_xd\H^{N-1}_y\right)^{\frac{1}{2}}
\\&\leq \left(\iint_{\partial\Om\times\partial\Om}\frac{|a(x)-a(y)|^2}{|x-y|^N}|b(x)|^2d\H^{N-1}_xd\H^{N-1}_y\right)^{\frac{1}{2}}\\
&+\left(\iint_{\partial\Om\times\partial\Om}\frac{|b(x)-b(y)|^2}{|x-y|^N}|a(y)|^2d\H^{N-1}_xd\H^{N-1}_y\right)^{\frac{1}{2}}
\\&\leq |a|_{C^{0,\alpha}(\partial\Om)}\left(\int_{\partial\Om}\left(\int_{\partial\Om}\frac{d\H^{N-1}_y}{|x-y|^{N-2\alpha}}\right)|b(x)|^2d\H^{N-1}_x\right)^{\frac{1}{2}}
\\&+\|a\|_{L^\infty(\partial\Om)}|b|_{H^{1/2}(\partial\Om)}.
\end{align*}
Now as $\alpha\in(\frac{1}{2},1)$, $N-2\alpha<N-1$ so that 
\[\sup_{x\in \partial\Om}\int_{\partial\Om}\frac{d\H^{N-1}_y}{|x-y|^{N-2\alpha}}<\infty,\]
thus yielding 
\[|ab|_{H^{1/2}(\partial\Om)}\leq C|a|_{C^{0,\alpha}(\partial\Om)}\|b\|_{L^2(\partial\Om)}+\|a\|_{L^\infty(\partial\Om)}|b|_{H^{1/2}(\partial\Om)},\]
for some constant $C=C(N,\partial\Om,\alpha)$, which finishes the proof.
\item Let $s\in(0,\frac{1}{2})$, $a,b\in H^{1/2}(\partial\Om)$ and put
\[p=\begin{cases}
    \frac{2N-2}{2N-3}\in(1,2), \text{ for }N\geq3,\\
    \text{any exponent lying in }(1,2) \text{ if }N=2.
\end{cases}\]
Then we claim that $ab\in W^{s,p}(\partial\Om)$. In fact, one has
\begin{align*}
    |ab|_{W^{s,p}(\partial\Om)}&=\left(\iint_{\partial\Om\times\partial\Om}\frac{|a(x)b(x)-a(y)b(y)|^p}{|x-y|^{N-1+sp}}d\H^{N-1}_xd\H^{N-1}_y\right)^{\frac{1}{p}}
    \\&\leq \left(\iint_{\partial\Om\times\partial\Om}\frac{|a(x)-a(y)|^p}{|x-y|^{N-1+sp}}|b(x)|^pd\H^{N-1}_xd\H^{N-1}_y\right)^{\frac{1}{p}}
    \\&+\left(\iint_{\partial\Om\times\partial\Om}\frac{|b(x)-b(y)|^p}{|x-y|^{N-1+sp}}|a(y)|^pd\H^{N-1}_xd\H^{N-1}_y\right)^{\frac{1}{p}}
    \\&\leq |a|_{H^{1/2}(\partial\Om)}\left(\iint_{\partial\Om\times\partial\Om}\frac{|b(x)|^{\frac{2p}{2-p}}}{|x-y|^{N+\frac{2}{2-p}(sp-1)}}d\H^{N-1}_xd\H^{N-1}_y\right)^{\frac{2-p}{2p}}
    \\&+|b|_{H^{1/2}(\partial\Om)}\left(\iint_{\partial\Om\times\partial\Om}\frac{|a(y)|^{\frac{2p}{2-p}}}{|x-y|^{N+\frac{2}{2-p}(sp-1)}}d\H^{N-1}_xd\H^{N-1}_y\right)^{\frac{2-p}{2p}}
\end{align*}
where we used Hölder inequality with exponents $\frac{2}{p}$ and $\frac{2}{2-p}$ in the last line. Now thanks to the choice of $p$ it holds
\[\begin{cases}
    \frac{2p}{2-p}=\frac{2(N-1)}{N-2}=2^*_{1/2} \text{ for } N\geq3,\\
    \frac{2p}{2-p}\in[2,\infty)\text{ if } N=2,
\end{cases}\]
where $2^*_{1/2}$ is the critical ($N-1$-dimensional) Sobolev exponent related to $\frac{1}{2}$ and $2$, so that the continuous embedding 
\[H^{1/2}(\partial\Om)\subset L^{\frac{2p}{2-p}}(\partial\Om)\] holds. On the other hand, as 
\[\frac{2}{2-p}(sp-1)<-1\]
since $s\in(0,\frac{1}{2})$, we deduce
\[\sup_{x\in \partial\Om}\int_{\partial\Om}\frac{d\H^{N-1}_y}{|x-y|^{N+\frac{2}{2-p}(sp-1)}}<\infty.\]
These yield the conclusion
\[\|ab\|_{W^{s,p}(\partial\Om)}\leq C\|a\|_{H^{1/2}(\partial\Om)}\|b\|_{H^{1/2}(\partial\Om)}\]
for some constant $C=C(N,\partial\Om, s,p)$.

\end{enumerate}
\end{proof}

\subsection{Compactness in classes of convex sets}\label{sect:Apphausd}
In this paragraph we gather some classical facts about Hausdorff distance and compactness in some classes of convex bodies. 

If $C_1$ and $C_2$ are non-empty compact subsets of $\R^N$, the Hausdorff distance $d_H(C_1,C_2)$ between $C_1$ and $C_2$ is defined as the quantity \begin{equation}\nonumber d_H(C_1,C_2):=\max\left\{\sup_{x\in C_1}d(x,C_2),\sup_{x\in C_2}d(x,C_1)\right\}\end{equation} where $d(\cdot,\cdot)$ denotes the euclidean distance. The Hausdorff distance $d_{H}$ is a distance over the class of non-empty compact sets of $\R^N$. We say that a sequence $A_j\subset \R^N$ of non-empty compact sets converges in the Hausdorff sense when it converges for $d_H$.

\begin{prop}\label{lem:compact} 
Let $D\in \K^N$ and $0<V_0<|D|$ and set
\[\mathcal{C}_1:=\left\{K\text{ compact convex of }  \R^N, K\subset D\right\}\] 
\[\mathcal{C}_2:=\left\{K\in\K^N,\ |K|=V_0,\ K\subset D\right\}\] 
\begin{enumerate}[label=(\roman{*})]
\item The classes $\mathcal{C}_1$ and $\mathcal{C}_2$ are compact for the Hausdorff distance.
\item There exists $\eps=\eps(V_0, D)>0$ such that for each $K\in \mathcal{C}_2$ the inradius $r_K$ of $K$ satisfies $r_K\geq \eps$.
\end{enumerate}
\end{prop}
\begin{proof}

\begin{enumerate}[label=(\roman{*})]

\item The Blaschke selection Theorem states that the class $\mathcal{C}_1$ is compact for the Hausdorff distance (see for instance \cite[Theorem 1.8.7]{Sc}). Compactness of $\mathcal{C}_2$ then follows from the fact that $\mathcal{C}_2$ is a closed subset of $\mathcal{C}_1$, thanks to the continuity of the volume for $d_H$.
\item Let us show that the inradius mapping $K\in\K^N\mapsto r_K$ is l.s.c. for $d_H$. Let $K\in\K^N$ and $K_j\in \K^N$ with $K_j\rightarrow K$ in Haudorff distance, and let $r>0$ and $x\in \R^N$ be such that $B_r(x)\Subset \text{Int}(K)$. Thanks to \cite[Proposition 2.8, 2.]{LP23}, we have that $K_j\supset B_r(x)$ for large enough $j$, so that $\lim\inf r_{K_j}\geq r$. This is valid for any $r<r_K$, so that $\lim\inf r_{K_j}\geq r_K$, thus showing that $K\in\K^N\mapsto r_K$ is l.s.c. Since furthermore $\mathcal{C}_2$ is compact for $d_H$, we deduce that $K\in \mathcal{C}\mapsto r_K$ has a minimum $\eps>0$, thus finishing the proof.
\end{enumerate}
\end{proof}

\begin{prop}\label{prop:cvgconv_hausd}
Let $B\subset\R^N$ be the centered unit ball. Let $K_j\in \K^N$ be a sequence of convex bodies such that $|K_j\Delta B|\rightarrow0$. Then $K_j\rightarrow B$ in the Hausdorff sense.
\end{prop}

\begin{proof}
It suffices to show that there exists a bounded set $D\subset\R^N$ such that $K_j\subset D$ for each $j$, since the Blaschke selection Theorem then applies to provide compactness of the sequence $K_j$ for $d_H$ and therefore the convergence of the whole sequence $K_j$ to $B$. Fix $j\in\N$; since $|K_j\Delta B|\rightarrow0$ we can suppose $j$ large enough so that $K_j\cap B\neq\emptyset$. 

Suppose now that $K_j\not\subset B_2(0)$. Since $K_j\cap B\neq\emptyset$, by convexity of $K_j$ we can find $x^j\in K_j$ with $|x^j|=2$, which we will suppose (up to changing coordinates) to be written $x^j:=x=(0,\ldots,0,2)$. Let $x_0:=(-1,0\ldots,0)$ and set $x_1:=(1,0\ldots,0)$, $x_2=(0,1,0\ldots,0)$ until $x_{N-1}:=(0,\ldots,0,1,0)$. Let finally $C:=\text{conv}\{x,x_0,x_1\ldots,x_{N-1}\}$ and some ball $B'\Subset \text{Int}(C\setminus B)$. 

Let $f:\R^{N+1}\rightarrow\K^N$ be defined by 
\[f(y_1,\ldots,y_{N+1})=\text{conv}\{y_1,\ldots,y_{N+1}\}=\left\{\sum_{i=1}^{N+1}\lambda_iy_i,\ \sum_{i=1}^{N+1}\lambda_i=1,\ \lambda_i\geq0\right\}\]
Then $f$ is continuous for the Hausdorff distance, so that there exists $\eps>0$ such that if for all $0\leq i\leq N-1$, $|z_i-x_i|\leq\eps$ then $\text{conv}\{x,z_1,\ldots,z_N\}\supset B'$ (see for instance \cite[Proposition 2.8, 2.]{LP23}). Now, let $\delta:=\min_i\{|B\cap B_\eps(x_i)|\}>0$. Taking $j$ sufficiently large so that $|K_j\Delta B|< \delta$, this implies that for such $j$ there exists for each $i=0,\ldots,N-1$ some $y_i^j\in K_j\cap B_\eps(x_i)$. By convexity, $K_j\supset \text{conv}\{x,y^j_0\ldots,y^j_{N-1}\}$, which itself contains $B'$, thus giving $|K_j\setminus B|\geq |B'|$. This does not happen for sufficiently large $j$, and as a consequence there exists $j_0\geq0$ such that $K_j\subset B_2(0)$ for $j\geq j_0$. This proves the claim.
\end{proof}

\bibliographystyle{alpha}
\bibliography{stab.bib}

\end{document}